\documentclass[10pt]{amsart}
\usepackage{geometry}
\usepackage{mathtools}
\usepackage[utf8]{inputenc}
\usepackage[T1]{fontenc}
\usepackage{amsmath}
\usepackage{amsfonts}
\usepackage{amssymb}
\usepackage{mathtools}
\usepackage{extarrows}
\usepackage{amsthm}
\usepackage{mathrsfs}
\usepackage[mathcal]{eucal}
\usepackage{graphicx}
\usepackage{xcolor}
\usepackage[a-1b]{pdfx}
\usepackage[british]{babel}
\usepackage{tikz-cd}
\usepackage{csquotes}
\usepackage[backend=biber]{biblatex}
\usepackage[cmtip, all]{xy}
\usepackage{accents}

\DeclareRobustCommand{\SkipTocEntry}[5]{}

\usetikzlibrary{decorations.pathmorphing}
\theoremstyle{definition}
\newtheorem{defn}{Definition}[section]
\newtheorem{constr}[defn]{Construction}
\newtheorem{exmp}[defn]{Example}
\newtheorem{rmk}[defn]{Remark}

\newtheorem{defnprop}[defn]{Definition/Proposition}
\theoremstyle{plain}
\newtheorem{introthm}{Theorem}

\newtheorem{thm}[defn]{Theorem}
\newtheorem{prop}[defn]{Proposition}
\newtheorem{corollary}[defn]{Corollary}
\newtheorem{lemma}[defn]{Lemma}
\numberwithin{equation}{section}

\newcommand{\sE}{{\mathcal E}}

\newcommand{\sH}{{\mathcal H}}

\newcommand{\sK}{{\mathcal K}}

\newcommand{\sO}{{\mathcal O}}

\newcommand{\sS}{{\mathcal S}}

\newcommand{\sU}{{\mathcal U}}
\newcommand{\sV}{{\mathcal V}}
\newcommand{\sW}{{\mathcal W}}

\newcommand{\A}{{\mathbb A}}

\newcommand{\C}{{\mathbb C}}
\newcommand{\E}{{\mathbb E}}
\newcommand{\F}{{\mathbb F}}
\newcommand{\G}{{\mathbb G}}

\renewcommand{\L}{{\mathbb L}}

\newcommand{\N}{{\mathbb N}}
\renewcommand{\P}{{\mathbb P}}

\newcommand{\V}{{\mathbb V}}

\newcommand{\Z}{{\mathbb Z}}

\newcommand{\MGL}{{\operatorname{MGL}}}

\newcommand{\MSL}{{\operatorname{MSL}}}
\newcommand{\MSp}{{\operatorname{MSp}}}
\newcommand{\BGL}{{\operatorname{BGL}}}
\newcommand{\BSp}{{\operatorname{BSp}}}
\newcommand{\GL}{{\operatorname{GL}}}
\newcommand{\Sp}{{\operatorname{Sp}}}
\newcommand{\Gr}{{\operatorname{Gr}}}
\newcommand{\HGr}{{\operatorname{HGr}}}
\newcommand{\HP}{{\operatorname{HP}}}
\newcommand{\Spec}{{\operatorname{Spec}}}

\newcommand{\SH}{{\operatorname{SH}}} 
\newcommand{\DM}{{\operatorname{DM}}}
\newcommand{\Th}{{\operatorname{Th}}}
\newcommand{\Spc}{{\operatorname{Spc}}}
\newcommand{\Sm}{{\operatorname{Sm}}}
\newcommand{\CH}{{\operatorname{CH}}}
\newcommand{\Hom}{{\operatorname{Hom}}}
\newcommand{\Map}{{\operatorname{Map}}}
\renewcommand{\th}{{\operatorname{th}}} 
\renewcommand{\deg}{{\operatorname{deg}}}
\renewcommand{\dim}{{\operatorname{dim}}} 
\renewcommand{\top}{{\operatorname{top}}}
\newcommand{\Mod}{{\operatorname{Mod}}}
\newcommand{\rnk}{{\operatorname{rank}}}
\newcommand{\Aut}{{\operatorname{Aut}}}
\newcommand{\Sym}{{\operatorname{Sym}}}
\newcommand{\id}{{\operatorname{id}}}
\newcommand{\op}{{\operatorname{op}}}
\newcommand{\Zar}{{\operatorname{Zar}}}
\newcommand{\colim}{{\operatorname*{colim}}}
\newcommand{\Anan}{{\operatorname{Anan}}}
\newcommand{\Ext}{{\operatorname{Ext}}}

\newcommand{\CF}{{\operatorname{CF}}}
\newcommand{\chr}{{\operatorname{char}}}

\newcommand{\Laz}{{\mathbb{L} \text{az}}}

\newcommand{\Sch}{{\operatorname{Sch}}}

\newcommand{\Obj}{{\operatorname{Obj}}}
\newcommand{\Ker}{{\operatorname{Ker}}}
\newcommand{\Locfree}{{\operatorname{LocFree}}}
\newcommand{\SPP}{{\mathcal{S}\text{p}}}
\newcommand{\hs}{{\operatorname{hs}}}
\newcommand{\Fl}{{\operatorname{Fl}}}
\newcommand{\MU}{{\operatorname{MU}}}

\title{Generators of the Algebraic Symplectic Bordism Ring}
\date{}
\author{Pietro Gigli}
\address{Université de Bourgogne Europe,
Institut de Mathématiques de Bourgogne, 21078 Dijon, France}
\email{Pietro.Gigli@u-bourgogne.fr}

\begin{document}

\begin{abstract}
    In this paper, we study the $\eta$-completed part of the motivic spectrum MSp constructed by Panin and Walter, representing the universal $\Sp$-oriented cohomology theory. In particular, we investigate the inclusion $(\MSp^\wedge_\eta)^*\hookrightarrow \MGL^*$ of the cofficient rings, by studying the motivic Adams spectral sequence associated to MSp, mimiking a strategy used by Levine,Yang, Zhao for $\MSL^*$. In order to give a description of $(\MSp^\wedge_\eta)^*$, we refine the Pontryagin-Thom construction in a way that allows one to obtain symplectic bordism classes from a large family of varieties that carry a certain "symplectic twist", and we prove a criterion to select generators among these classes.
\end{abstract}

\thanks{The author was partially supported by the DFG through the grant  LE 2259/8-1 and by the ERC through the project QUADAG.  This paper is part of a project that has received funding from the European Research Council (ERC) under the European Union's Horizon 2020 research and innovation programme (grant agreement No. 832833).\\
\includegraphics[scale=0.08]{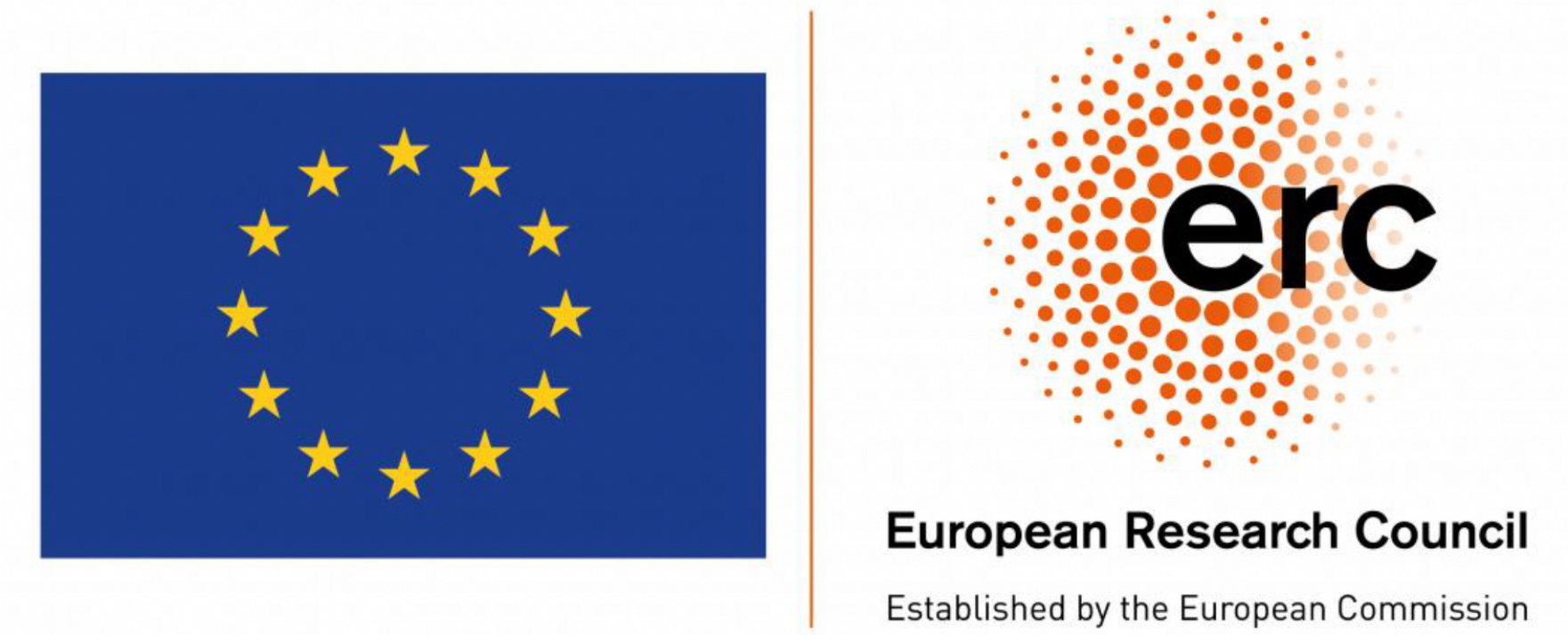}}

\maketitle

\tableofcontents

\section{Introduction}

\addtocontents{toc}{\SkipTocEntry}
\subsection*{Motivation}
We work over a perfect field $k$ of characteristic not 2, and consider cohomology theories over smooth $k$-schemes. The main goal of this paper is to obtain a partial computation of the coefficient ring of algebraic symplectic cobordism, the universal symplectically oriented cohomology theory for schemes. A complete description of this ring is at present unknown.

The strongest notion of "orientation" for a cohomology theory is that of $\GL$-orientation, which corresponds to the data of Thom classes for vector bundles, with appropriate compatibilities among them. Voevodsky's algebraic cobordism spectrum $\MGL \in \SH(k)$ \cite[Section 6.3]{voe:homotopy_theory} is the most natural analogue of the complex cobordism spectrum $\MU \in \SH$ \cite[Chapter 1, Section 5]{C-F:Cobordism} in our setting, and has a natural $\GL$-orientation by construction. Just as $\MU$ represents the universal complex oriented cohomology theory (\cite[Lemma 4.1.13]{ravenel:cobordism}), similarly $\MGL$ is universal among $\GL$-oriented theories, namely, every $\GL$-oriented cohomology theory $\sE^{*,*}(-)$ is identified by a map $\MGL\to \sE$ of commutative ring spectra in $\SH(k)$. Hoyois \cite{Hoyois:AlgCob} and Spitzweck \cite{Spitzweck:AlgCob} showed that, after inverting $\chr k$ if $\chr k \neq 0$, the algebraic diagonal $\MGL^*:=\MGL^{2*,*}(\Spec k)$ of the coefficient ring of $\MGL$ is isomorphic to the Lazard ring $\L$, the universal coefficient ring for formal group laws, which is known to be the polynomial ring in infinite variables $\Z[x_1,x_2,\ldots]$ \cite[Part II, Theorem 7.1]{Adams:homotopy}. Moreover, through a Pontryagin-Thom construction, one can associate to each smooth proper variety $X$ over $k$ a class $[X]_\MGL \in \MGL^{-\dim X}$ (see \cite[Definition 3.1]{lev:ellcoh}), and $\MGL^*$ is generated by such classes (\cite[Theorem 3.4 (4)]{lev:ellcoh}). In summary, we have the following.

\begin{introthm}
\label{introthm:MGL*}
    Let $p$ be the exponential characteristic of $k$, namely $p=\chr k$ if $\chr k>0$ and $p=1$ otherwise. Then we have an isomorphism of graded ring
    $$\MGL^*[1/p]\simeq \Z[1/p][x_1,x_2, \ldots]$$
    with $x_i$ in degree $-i$, and $\MGL^*[1/p]$ is generated by classes $[X]_\MGL$, $X$ a smooth and proper $k$-scheme.
\end{introthm}

Nevertheless, in $\A^1$-homotopy theory various "interesting" cohomology theories are not $\GL$-oriented, examples include Chow-Witt theory and Hermitian $K$-theory. This led to defining weaker notions of orientation (\cite{Panin-Walter:BO}, \cite{panwal:grass}), which consists of a theory of Thom classes for vector bundles with additional structures instead of all vector bundles.

Panin and Walter \cite[Section 6]{Panwal-cobordism} constructed the motivic spectrum $\MSp$ representing algebraic symplectic cobordism, a motivic analogue of quaternionic symplectic cobordism $\MSp^\top$ (\cite{Thomcompl},\cite[Chapter 1, Section 5]{C-F:Cobordism}), and proved that this is universal for having an orientation for symplectic vector bundles \cite[Theorem 4.5]{Panwal-cobordism}. The above mentioned Pontryagin-Thom construction can be adapted to construct a class $[X]_\MSp \in \MSp^{-\dim X}$ for each smooth proper $k$-variety $X$ with a stably symplectic structure on the tangent bundle. Knowing the complete structure of the coefficient ring $\MSp^{*,*}$ would give useful insights in understanding $\Sp$-orientations and formal ternary laws. It is then natural to ask whether $\MSp^*$ has a description similar to that of $\MGL^*$ given by Theorem \ref{introthm:MGL*}. It turns out that for $\MSp$ the picture is much less clear, and in fact the structure of the coefficient ring of $\MSp$ is mostly unknown. We can immediately make a couple of observations:

1) $(\MSp^\top)^*$ is not polynomial (see e.g. \cite{Ray:SymplBordism}). Roughly, this arises from the fact that the sum of a non-trivial symplectic bundle with itself can be a trivial symplectic bundle, and this gives indecomposable torsion elements of order 2. For example, Ray \cite{Ray:Tors} gave explicit computations of indecomposable torsion elements $\phi_i \in (\MSp^\top)^{-8i+3}$ of order $2$ for all $i\ge 1$. Based on the situation in classical homotopy theory, we expect that $\MSp^*$ is also not polynomial and contains infinitely many torsion elements.

2) If we want to compute generators of $\MSp^*$ through the Pontryagin-Thom construction, we realize that there are not many known examples of symplectic varieties in algebraic geometry. Standard examples like $K3$-surfaces and their Hilbert schemes of points are computationally difficult to handle, compared to projective spaces and Milnor hypersurfaces in the $\GL$-oriented case. 

A key fact that needs to be considered to simplify the problem, is that the category of motivic spectra $\SH(k)$, after inverting $2$, has a splitting
$$\SH(k)[1/2] \simeq \SH(k)^- \oplus \SH(k)^+$$
induced by the involution $\P^1 \wedge \P^1 \xrightarrow{1:1}\P^1 \wedge \P^1$,
and so does the coefficient ring of $\MSp$. A special role in this splitting is played by the motivic Hopf map $\eta:\G_m \to 1_k$, inasmuch the projection $\SH(k)[1/2]\to \SH(k)^-$ corresponds to localizing at $\eta$ and the projection $\SH(k)[1/2]\to \SH(k)^+$ corresponds to a certain $\eta$-completion $(-)_\eta^\wedge$. Bachmann and Hopkins studied the $\eta$-inverted symplectic bordism ring, establishing the relation $\MSp[1/\eta]^*\cong W(k)[x_1,x_2, \ldots]$, where $W(k)$ is the Witt ring of $k$, and $x_i$ is in degree $2i$ (\cite[Theorem 8.7]{Bachmann:eta-periodic}). This takes care of the minus-part of $\MSp[1/2]^*$.

In this paper, we investigate in details the plus-part of $\MSp[1/2]^*$. As in topology, the main tool for this purpose is the Adams spectral sequence. Nevertheless, the algebraic Adams spectral sequence differs significantly from its topological analogue, one of the main differences being the non-nilpotence of the motivic Hopf map $\eta$.

\addtocontents{toc}{\SkipTocEntry}
\subsection*{Strategy and final results}
The inclusion $\Sp_* \subset \GL_*$ induces a natural map $\Phi:\MSp \to \MGL$, which in turn induces the map $(\MSp^\wedge_\eta)^* \to (\MGL^\wedge_\eta)^*=\MGL^*$, $\MGL$ being $\eta$-completed. We want to prove that $\Phi_*$ is a split embedding. We face this problem by fixing an odd prime $\ell$ and comparing the mod-$\ell$ Adams spectral sequences associated to the two spectra. We adopt the strategy used in \cite{lev:ellcoh} to prove that the Adams spectral sequence for $\MSp$ has the form
$$E_2^{s,t,u}= \Ext_{A^{*,*}}^{s,(t-s,u)}(H^{*,*}(\MSp),H^{*,*}) \Rightarrow (\MSp_{H\Z /\ell}^\wedge)^{t,u},$$
and it degenerates at the second page (Propositions \ref{prop: a.s.s.MSp}-\ref{prop:a.s.s.MSp2}), analogously to that of $\MGL$. We figure a presentation of the $E_2$-page as a polynomial $Z/\ell$-algebra, and comparing this with the $E_2$-page of the spectral sequence for $\MGL$ given by \cite[Proposition 5.7]{lev:ellcoh}, we prove that $(\MSp_{\eta,\ell}^\wedge)^{2*}\cong(\MGL_\ell^\wedge)^{2*}$, while $(\MSp_{\eta,\ell}^\wedge)^{2*+1}=0$. In particular, $(\MSp_{\eta,\ell}^\wedge)^*$ is polynomial over $\Z/\ell$ and concentrated in non-positive even degrees.
Along the way, by using a known criterion to detect generators in $\MGL^*$, we explicitely compute elements $[Y_{2d}]_\MSp \in \MSp^{-2d}$, for $d\ge 1$, which generate $(\MSp_{\eta,\ell}^\wedge)^*$ as subalgebra of $(\MGL^\wedge_\ell)^*$ via $\Phi_*$.
Finally, by working on the $\ell$-torsion of $(\MSp_{\eta,\ell}^\wedge)^*$ one prime $\ell$ at a time, we globalize the previous results, and we conclude the following.

\begin{thm}
    Let $({\overline{\MSp}^\wedge_\eta})^*[1/2p]$ be the quotient of $({\MSp^\wedge_\eta})^*[1/2p]$ by its maximal subgroup that is $\ell$-divisible for all odd primes $\ell \neq p$. Then $({\overline{\MSp}^\wedge_\eta})^*[1/2p] \cong \Z[1/2p][y_1,y_2, \ldots]$, with $y_i$ in degree $-2i$. Moreover, $y' \in {\overline{\MSp}^\wedge_\eta})^{-2d}[1/2p]$ is a generator if and only if the following is satisfied:
    \begin{equation*}
        c_{(2d)}(\Phi_* y')=
        \begin{cases}
            \lambda \in \Z[1/2p]^\times & \text{for} \; \; 2d \neq \ell^r-1, \; \; \text{if all prime} \;  \ell \neq 2,p; \; \forall r \ge 1 \\
            \lambda' \cdot \ell, \; \lambda' \in \Z_\ell^\times & \text{if} \; \; 2d= \ell^r-1, \; \; \ell \; \text{prime}, \; \ell \neq 2,p; \; r\ge 1,
        \end{cases}
    \end{equation*}
    with $c_{(2d)}(y)$ a certain characteristic number associated to universal Newton classes in motivic cohomology.
\end{thm}
See Theorem \ref{thm:FinalResult}.

The strategy is completely parallel to that used by Levine, Yang and Zhao in \cite{lev:ellcoh} to compute the $\eta$-completed SL-cobordism ring $(\MSL^\wedge_{\eta})^*$, but there are critical differences in our situation to take into account. First, some cellularity properties for $\MSp$ and its motive $\MSp \wedge H\Z$ are required in order to prove the convergence of the Adams spectral sequence. For $\MGL$ and $\MSL$ this cellularity is induced by the Schubert filtration on ordinary Grassmannians, but this can not be extended to our case due to the different geometry of quaternionic Grassmannians. We then need to study the geometry of quaternionic Grassmannians, following \cite{panwal:grass}, and use a double-inductive argument to get the required properties. This procedure was also used in \cite[Proposition 3.1]{Rond:Cellularity}. Also, we already mentioned that the application of the Pontryagin-Thom construction to compute classes in the symplectic bordism ring is not very helpful. We then use an isomorphism of Thom spaces due to Ananyevskiy (\cite[Lemma 4.1]{Ana:Slor}), which is also a special instance of \cite[Proposition 2.2]{Ron:Theta}, to extend the symplectic Pontryagin-Thom construction to varieties that are not necessarily symplectic but carry a certain "symplectic twist", namely varieties $X$ for which the Thom space $\Th(T_X)$ of the tangent bundle is stably isomorphic to the Thom space of a virtual vector bundle admitting a virtual symplectic form. Lastly, in order to compute characteristic numbers of our new twisted classes, we need to modify the classical Voevodsky's construction of the degree map in motivic cohomology in order to take the symplectic twist into account.

\addtocontents{toc}{\SkipTocEntry}
\subsection*{Organization of the paper}
In \S \ref{section:preliminaries} we fix all the necessary background for our work. In particular, we start by recalling the basic notions from motivic homotopy theory and the formalism of six functors for the sake of the reader. We then review the theory of $\GL$-orientations and $\Sp$-orientations, including discussions on the universal motivic spectra $\MGL$ and $\MSp$ in order to fit all the needs for our future purposes. In \S \ref{section:TwistedDegree} we use Ananyveskiy's isomorphism $\Th(E\oplus L)\simeq \Th(E \oplus L^\vee)$ to define a twisted version of Voevodsky's degree map, which we use then to compute the degree of a certain kind of twisted cohomology classes valued in motivic cohomology. In \S \ref{section:Cellularity} we show that $\MSp$ is a cellular spectrum and we prove a similar statement for the $H\Z$-module $\MSp \wedge H\Z$. In \S \ref{section:A.s.s} we study the convergence of the mod-$\ell$ Adams spectral sequence for $\MSp$ by using the main results of \S \ref{section:Cellularity}, and give a presentation of the $E_2$-page as a trigraded $\Z/\ell$-algebra. In \S \ref{section:Generators} we use the results of \S \ref{section:A.s.s} to prove that the $\eta$-completed symplectic bordism ring is a polynomial subalgebra of the algebraic bordism ring, ad we use the computations of \S \ref{section:TwistedDegree} to prove a criterion to detect polynomial generators by looking at some characteristic numbers.

\addtocontents{toc}{\SkipTocEntry}
\subsection*{Notations and conventions}
All schemes in this paper are assumed to be quasi-compact quasi-separated. Starting from \S \ref{subsection:GL-orient.}, we work over a perfect field $k$ of exponential characteristic $p\neq 2$, and we write $\sH_\bullet(k) \coloneqq \sH_\bullet(\Spec k)$, $\SH(k) \coloneqq \SH(\Spec k)$ for the unstable and stable motivic homotopy category over $k$. For $\mathcal{C}$ Ab-enriched category, we write $[A,B]_\mathcal{C}$ for the abelian group $\Hom_{\mathcal{C}}(A,B)$, in particular we use this notation for cohomology groups. For $\sE^{*,*}(-)$ cohomology theory, we write $\sE^{*,*}\coloneqq \sE^{*,*}(\Spec k)=\oplus_{a,b \in \Z}\sE^{a,b}(\Spec k)$ for the (bigraded) coefficient ring, $\sE^*:=\sE^{2*,*}=\oplus_{a \in \Z}\sE^{2a,a}(\Spec k)$ for its (graded) diagonal, and $\sE^a:=\sE^{2a,a}(\Spec k).$ For $F\dashv G$ adjoint functors, we occasionally use the notation $\eta_{(F,G)}$ and $\epsilon_{(F,G)}$ for the counit and unit of the adjunction respectively. Starting from \S \ref{section:A.s.s}, $\ell$ will be a fixed odd prime different from $p$, $\nu_\ell:\Z \to \Z$ will denote the $\ell$-adic valuation, and $G^\wedge_\ell$ will denote the $\ell$-adic completion of a (possibly graded) abelian group $G$.
 
 \addtocontents{toc}{\SkipTocEntry}
 \subsection*{Acknowledgements}
 This paper is based on the author's Ph.D thesis \cite{gigli:thesis}. The author wholeheartedly thanks his supervisor M. Levine for his dedication and priceless mentoring work. The author also thanks A. Ananyevskiy for reviewing the thesis and for his many useful comments which improved this paper, as well as the whole ESAGA group in Essen for creating a lively and stimulating working environment, the DFG foundation, ERC and CNRS for financial support.
 
 \section{Preliminaries}
 \label{section:preliminaries}
 
 \subsection{Background on $\A^1$-homotopy theory}
 
 Let $S$ be a Noetherian scheme of finite Krull dimension. Let $\Sch/S$ be the category of quasi-projective $S$-schemes, and $\Sm/S$ the full subcategory of $\Sch/S$ of smooth quasi-projective $S$-schemes. 
 
 We follow the classical approach of Morel and Voevodsky. $\Spc(S)$ and $\Spc_{\bullet}(S)$ denote respectively the categories of Nisnevich sheaves of simplicial sets and Nisnevich sheaves of pointed simplicial sets on $\Sm/S$. We call objects in $\Spc(S)$ and $\Spc_{\bullet}(S)$ \emph{spaces} and \emph{pointed spaces} respectively. By identifying a scheme $U \in \Sm/S$ with the Nisnevich sheaf associated to the constant representable presheaf $\text{Hom}_{\Sm/S}(-,U)$ we get the embedding $\Sm/S \hookrightarrow \Spc(S)$, and similarly $(\Sm/S)_+\hookrightarrow \Spc_\bullet(S)$. We also have the \emph{smash product} $\wedge$ on $\Spc_\bullet(S)$.
 
 By \cite[Theorem 3.7]{voe:homotopy_theory}, $\Spc(S)$ has a closed model structure, with $\A^1$-weak equivalences as in \cite[Definition 3.4]{voe:homotopy_theory}. These weak equivalences are intended to be an algebraic analogue of homotopy equivalences of topological spaces. The resulting homotopy category is denoted by $\sH(S)$.
 
 The model structure on $\Spc(S)$ induces a model structure on $\Spc_\bullet(S)$, and the resulting homotopy category is denoted by $\sH_\bullet(S)$. We refer to objects in $\sH_\bullet(S)$ as \emph{motivic spaces}.
 
 \begin{rmk}
\label{rmk:colims}
    $\Spc(S)$ and $\Spc_{\bullet}(S)$ have all small limits and colimits. On the other hand, despite homotopy categories are not complete nor cocomplete in general, one can always take homotopy limits and homotopy colimits in $\Spc(S)$ and $\Spc_{\bullet}(S)$ as objects in $\sH(S)$ and $\sH_{\bullet}(S)$ respectively.
\end{rmk}

The smash product $\wedge$ in $\Spc_{\bullet}(S)$ induces a smash product $\wedge$ in $\sH_{\bullet}(S)$. With this product, $\sH_{\bullet}(S)$ acquires a symmetric monoidal structure, with unit being the \emph{$0$-sphere} $S^0\coloneqq S_+$. This follows by the properties of weak equivalences (see \cite[Section 3, pp.585-586]{voe:homotopy_theory}).

In $\sH_\bullet(S)$, there are two "circles": the \emph{simplicial circle} $S^1$ (see for instance \cite{voe:homotopy_theory} before lemma 3.8) and the \emph{punctured line} $\G_m \coloneqq (\A^1_S\setminus \{0\},1)$. One defines then the \emph{motivic spheres}
    $$S^{p,q} \coloneqq (S^1)^{\wedge (p-q)} \wedge \G_m^{\wedge q}.$$
For $p \ge q \ge 0$, we have the suspension endofunctor $\Sigma^{p,q}: \sH_{\bullet}(S) \to \sH_{\bullet}(S)$ defined by $\Sigma^{p,q}X \coloneqq S^{p,q} \wedge X$.

There is a canonical isomorphism $S^{2,1} \simeq \P^1 =: (\P^1_S,\infty)$ in $\sH_{\bullet}(S)$, given by the fact that both these spaces are homotopy pushouts of $\A^1 \hookleftarrow \G_m \hookrightarrow \A^1$ in $\Spc_{\bullet}(S)$, with $\A^1$ pointed at $1$.

    For $\pi: V\to S$ vector bundle, $\P(V)$ denotes the associated projective space $\operatorname{Proj}(\Sym^*\sV)\xrightarrow{q}S$, with $\sV$ sheaf of sections of $V^\vee$. The \emph{Thom space of $V$}, denoted by $\Th(V)$, is the quotient $\P(V \oplus \mathcal{O}_S)/\P(V)$ in $\Spc_{\bullet}(S)$, with the obvious distinguished point, which exists because of Remark \ref{rmk:colims}. Equivalently, $\Th(V)$ is the quotient $V/ V - s_0(S)$, with $s_0:S \to V$ the zero section.

\begin{rmk}
\label{rmk:Thomspaces}
    \begin{enumerate}
        \item From the properties of the smash product, it follows that, if $V_1,V_2 \to S$ are two vector bundles over $S$, one has $\Th(V_1 \oplus V_2) \simeq \Th(V_1) \wedge \Th(V_2)$.
\item Since $\P^1 \simeq \A^1/\G_m$, in $\sH_{\bullet}(S)$, we have in particular $\P^1 \simeq \Th(\mathcal{O}_S)$ in $\sH_{\bullet}(S)$. By the previous point, one also gets $\Th(\mathcal{O}_S^n)\simeq (\P^1)^{\wedge n} \simeq S^{2n,n}$. If $S \in \Sm/S'$, we can consider $\Th(\mathcal{O}_S^n)$ as a motivic space over $S'$. Technically, it corresponds to $S^{2n,n}\wedge S_+ \in \sH_{\bullet}(S')$.
\item Let $g:X \to Y$ be a map in $\Sm/S$, and $V\to X$, $W \to Y$ two vector bundles with a map $f:V \to W$ arising from a fiberwise injective map $f_X:V\to g^*W$. Then the Thom space construction gives naturally a map $\Th(g,f):\Th_X(V) \to \Th_Y(W)$ in $\sH_{\bullet}(S)$. If the map $f$ is obvious from the context, for instance if $V=g^*W$, we will just write $\Th(g)$. 
    \end{enumerate}
\end{rmk}

There exist various equivalent stabilization procedures which produce the stable homotopy category $\SH(S)$. One model is given by $\P^1$-spectra. In general, for $T \in \Spc_{\bullet}(S)$, a \emph{$T$-spectrum} $\E$ is a sequence of pointed spaces $(E_0,E_1, E_2, \ldots)$ with bonding maps $\epsilon_r:T \wedge E_r \to E_{r+1}$, and morphisms of $T$-spectra are componentwise morphisms of pointed spaces such that all the obvious squares with the bonding maps commute. There is also a notion of \emph{stable weak equivalences} of $T$-spectra \cite[before Definition 5.1]{voe:homotopy_theory}. The stable homotopy category $\SH(S)$ is the left Bousfield localization of $\P^1$-spectra with respect to these equivalences (see \cite[Definition 5.7]{voe:homotopy_theory}, \cite[Theorem 2.9]{Jardine:SymSpectra} or \cite[Definition 3.3]{Hovey:Spectra}). Objects in $\SH(S)$ are called \emph{motivic $\P^1$-spectra over $S$}.

The category $\SH(S)$ has again all small limits and colimits. There is an \emph{infinite $\P^1$-suspension functor} $\Sigma_{\P^1}^\infty:\sH_{\bullet}(S)\to \SH(S)$ defined by  
    $$\Sigma_{\P^1}^\infty(X) \coloneqq (X, \P^1 \wedge X, (\P^1)^{\wedge 2} \wedge X, \ldots, (\P^1)^{\wedge r}\wedge X, \ldots),$$
    with bonding maps being the identities, and suspension endofunctors $\Sigma_{\P^1}, \Sigma_{S^1}, \Sigma_{\G_m}: \SH(S) \to \SH(S)$. The main features of $\SH(S)$ are recalled in the following theorem.
    
    \begin{thm}[\cite{voe:homotopy_theory}, Proposition 5.4, Theorem 5.6]
    \label{thm:SHsymmonoidal}
    $\SH(S)$ is additive, with coproduct $\oplus$ induced by the wedge sum of pointed spaces, and has a triangulated structure with shift functor being $\Sigma_{S^1}$. Also, $\SH(S)$ is symmetric monoidal, with monoidal product $\wedge$ induced by the smash product of pointed spaces, and monoidal unit the sphere spectrum $1_S \coloneqq \Sigma_{\P^1}^\infty S_+$, and the following properties hold:
    \begin{enumerate}
        \item For $\E \in \SH(S)$ and $X \in \Spc_{\bullet}(S)$, the motivic spectrum $\E \wedge \Sigma_{\P^1}^\infty X$ is canonically isomorphic to the spectrum $(E_0 \wedge X, E_1 \wedge X, E_2 \wedge X, \ldots)$ with bonding maps $\epsilon_i \wedge \id_X$.
        \item For $\E \in \SH(S)$ and $\{\F_\alpha\}_\alpha$ a collection of motivic spectra in $\SH(S)$, we have a canonical isomorphism
        $$(\oplus_\alpha \F_\alpha) \wedge \E \xrightarrow{\sim} \oplus_\alpha(\F_\alpha \wedge \E).$$
    \end{enumerate}
    \end{thm}
    
Basically by construction of $\SH(S)$ and Property (1) of Theorem \ref{thm:SHsymmonoidal}, $\P^1$ is invertible in $\SH(S)$, then $\Sigma_{\P^1},\Sigma_{S^1}, \Sigma_{\G_m}$ are invertible endofunctors. One can then define $\Sigma_{\P^1}^n$, $\Sigma_{S^1}^n$, $\Sigma_{\G_m}^n$ for any $n \in \Z$. By adopting a notation analogous to the unstable motivic spheres, one writes $\Sigma^{p,q}\coloneqq \Sigma_{S^1}^{p-q}\circ \Sigma_{\G_m}^q$ for all $p,q$ in $\Z$, and we have $\Sigma_{\P^1} \simeq \Sigma^{2,1}$.

Every motivic spectrum $\E \in \SH(S)$ defines bigraded cohomology functors $\E^{p,q}(-): \Spc(S)^{\text{op}} \to \text{Ab}$ and homology functors $\E_{p,q}(-): \Spc(S) \to \text{Ab}$ for any integer $p,q$, in the category of abelian groups, through:
$$\E^{p,q}(X) \coloneqq [\Sigma_{\P^1}^\infty X_+, \Sigma^{p,q}\E]_{\SH(S)}, \; \; \, \E_{p,q}(X) \coloneqq [1_S, \Sigma^{-p,-q}\E \wedge \Sigma_{\P^1}^\infty X_+]_{\SH(S)}.$$
For $f:X \to Y$ in $\Spc(S)$, composing with $\Sigma_{\P^1}^\infty f$ defines the ordinary cohomology pullback $\E^{p,q}(f)=f^*$ and the ordinary homology pushforward $\E_{p,q}(f)=f_*$. Let us note that $\E^{p,q}(S)=\E_{-p,-q}(S)$.

\begin{defn}
    A commutative monoid object $\sE$ in the symmetric monoidal category $\SH(S)$ is called a \emph{motivic commutative ring spectrum}.
\end{defn}
In particular, a motivic commutative ring spectrum $\sE$ is equipped with a multiplication map $\mu_{\sE}:\sE \wedge \sE \to \sE$ and a unit map $\epsilon_\sE:1_S \to \sE$.

For $\sE$ motivic commutative ring spectrum, the multiplication map $\mu_\sE$ induces a natural cup product $\cup:\sE^{p,q}(X) \times \sE^{p',q'}(X) \to \sE^{p+p',q+q'}(X)$, which makes $\sE^{*,*}(X)$ a graded ring, and a natural cap product $\cap:\sE^{p,q}(X) \times \sE_{p_1,q_1}(X) \to \sE_{p_1-p,q_1-q}(X)$. Also, for $\E \in \SH(S)$, we have a canonical pairing $\langle -,-\rangle:\sE^{p,q}(\E)\times \sE_{p_1, q_1}(\E)\to \sE_{p_1-p, q_1-q}(S)$ defined as
\begin{equation}\label{eqn:Pairing}
\langle \alpha,\gamma\rangle: 1_S\xrightarrow{\gamma}\Sigma^{-p_1, -q_1}\sE\wedge \E
\xrightarrow{\id_\sE\wedge\alpha}\Sigma^{p-p_1,q-q_1}\sE\wedge\sE
\xrightarrow{\mu_\sE} \Sigma^{p-p_1,q-q_1}\sE.
\end{equation}

\begin{rmk}\label{rmk:HurewiczMap} If $\sE\in \SH(S)$ is a commutative ring spectrum, then for any $\E \in \SH(S)$ we have the {\em motivic $\sE$-Hurewicz map} $h_\sE:\E_{a,b}(S)\to \sE_{a,b}(\E)$ defined by $h_\sE(f):=(\epsilon_\sE\wedge\id)\circ f:1_S\to \Sigma^{-a,-b}\sE\wedge \E.$ For $\alpha\in \sE^{p,q}(\E)$, $\gamma_0\in \E_{p_1,q_1}(S)$, we then have $\langle \alpha,h_\sE(\gamma_0)\rangle =\alpha\circ \gamma_0\in \sE_{p_1-p,q_1-q}(S)$.
\end{rmk}

Note that, if $X \in \Spc(S)$ has structure map $p$, the cohomology pullback $p^*$ makes the coefficient ring $\sE^{*,*}(X)$ a bigraded module over $\sE^{*,*}(S)$. The $\sE$-cohomology $\sE^{*,*}(-)$ can then be seen as a functor from $\Spc(S)^\op$ to bigraded $\sE^{*,*}(S)$-algebras.

There exist various standard models for $\SH(S)$ as homotopy category of some symmetric monoidal model category. One such model is given by symmetric spectra (see \cite{Jardine:SymSpectra} or \cite{Hovey:SymmSpectra}).
 
 \begin{defn}
\label{defn:highlyStructuredRings}
    A motivic commutative ring spectrum $\sE \in \SH(S)$ is said to be a \emph{highly structured motivic commmutative ring spectrum} if it admits a model $\sE^{\hs}$ as a commutative monoid object in a certain symmetric monoidal model category $\SPP^{\hs}_S$ such that $\text{Ho}(\SPP^{\hs}_S) \simeq \SH(S)$ as symmetric monoidal categories. For $\sE$ a highly structured motivic commutative ring spectrum with model $\sE^{\hs}\in \SPP^\hs_S$, the homotopy category $\text{Ho}(\Mod_{\sE^{\hs}})$ of $\sE^{\hs}$-module objects in $\SPP^{\hs}_S$ gives the \emph{category of $\sE$-modules} $\Mod_\sE$. Moreover, the usual free-forget adjunction on $\sE^{\hs}$-modules in $\SPP^\hs_S$ induces the adjunction
\[
\text{Free}_\sE:\SH(S)\xymatrix{\ar@<3pt>[r]&\ar@<3pt>[l]}\Mod_\sE:\text{Forget},
\]
where, for a motivic spectrum $\E \in \SH(S)$, $\text{Free}_\sE(\E)$ is the $\sE$-module $\E \wedge \sE$ with module map $\id_\E \wedge \mu_\sE: \E \wedge \sE \wedge \sE \to \E \wedge \sE$ (we use the right modules notation).
\end{defn}

\begin{exmp}
    Voevodsky's \emph{Eilenberg-Maclane spectrum} $H\Z$ is the motivic spectrum representing \emph{motivic cohomology}. Over a field of characteristic $0$, the construction is described in \cite[Section 6.1]{voe:homotopy_theory}, but the definitions has been extended over more general bases in many ways (see \cite{rond:modules}, \cite{Hoy:Steenrod}, \cite{Spitzweck:HZ}, \cite{deglise:mixmot}, \cite{Hoyois:localization}). $H\Z$ is a highly structured motivic ring spectrum. This is proved for instance in \cite[Example 3.4]{DunRon:Functors} through a different model for $H\Z$, which is nevertheless equivalent to Voevodsky's Eilenberg-Maclane spectrum by \cite[Lemma 4.6]{DunRon:Functors}. 
\end{exmp}

This allows one to define the category $\Mod_{H\Z}$ of $H\Z$-modules. For $S=\Spec k$, $k$ a field, we let $\DM(k) \coloneqq \Mod_{H\Z}$. Usually, $\DM(k)$ denotes Voevodsky's triangulated category of motives constructed in \cite{voev:MotivicHomology}, but the main result of \cite{rond:modules} shows that this is equivalent to $\Mod_{H\Z}$ in characteristic $0$, and in characteristic $p$ the same is true after inverting $p$, by \cite[Theorem 5.8]{Hoy:Steenrod}. So for our purposes in this work there should be no confusion.

We also recall a partial computation of motivic cohomology.

\begin{thm}[\cite{MazWei:lectures}, Theorem 19.1, Theorem 19.3, Corollary 4.2]
\label{thm:MazWeiHZ}
    For $X \in \Sm/k$, with $k$ a perfect field, $H\Z^{p,q}(X)=0$ if either $q<0$, $p > 2q$ or $p>q +\dim(X)$. Moreover, $H\Z^{0,0}(\Spec k)=\Z$, $H\Z^{p,0}(\Spec k)=0$ for $p \neq 0$, and $H\Z^{p,1}(\Spec k)=0$ for $p \neq 1$.
\end{thm}

\subsection{The six functors}

As $S$ varies among quasi-compact quasi separated schemes over a fixed base scheme $B$ Noetherian of finite Krull dimension, the categories $\SH(S)$ acquire the formalism of Grothendieck's six operations, developed in \cite{ayoub:sixfunctors}. We now recall the main features.

\begin{defnprop}
\label{prop:fourfunctors}
    For every morphism $f:S \to T$ of schemes, we have two adjoint functors
    $$
    \begin{tikzcd}
    f^*: \SH(T) \arrow[r, shift left] & \SH(S): f_*, \arrow[l, shift left]
    \end{tikzcd}
    $$
    where $f^*$ is symmetric monoidal and compatible with the unstable functor $f^*$ through infinite $\P^1$-suspension functors. $f^*$ and $f_*$ are called, respectively, \emph{inverse image along $f$} and \emph{direct image along $f$}. Furthermore, if $f$ is separated of finite presentation, we have another pair of adjoint functors
    $$\begin{tikzcd}
    f_!: \SH(S) \arrow[r, shift left] & \SH(T): f^!, \arrow[l, shift left]
    \end{tikzcd}
    $$
    called, respectively, \emph{exceptional direct image along $f$} and \emph{exceptional inverse image along $f$}.
\end{defnprop}

In almost all cases in this paper, we will work on $\Sch/B$, and morphisms of quasi-projective schemes are automatically separated of finite presentation. Then we will usually not worry about the existence of the exceptional functors.

Sending $S \in \Sch/B$ to $\SH(S)$, and a map of schemes $f$ to the functor $f^*$, defines a pseudofunctor
$$\SH(-): \Sch/B^{\text{op}} \to \mathbf{Tr}^\otimes$$
from $\Sch/B$ to the category of symmetric monoidal triangulated categories. From now on, for the sake of the exposition, we will focus on the case $S \in \Sch/B$. Differently, the statements involving $f_!$ and $f^!$ only hold for $f$ separated of finite presentation.

The construction of the functors $f^*,f_*,f^!,f_!$ is discussed in \cite[Chapter 1]{ayoub:sixfunctors}, as are the properties listed in the following proposition.

\begin{prop}
\label{prop:sixfunctors}
    The four functors for $\SH(-)$ defined in Proposition \ref{prop:fourfunctors} satisfy the following properties:
    \begin{enumerate}
        \item For every $f$, there is a natural transformation $f_! \to f_*$ that is invertible if $f$ is proper.
        \item If $f$ is an open immersion, there is a natural isomorphism $f^* \xrightarrow{\sim} f^!$.
        \item If $f$ is smooth, $f^*$ has a left adjoint, denoted by $f_{\#}$.
        \item (Projection formula) The exceptional direct image satisfies projection formula against the inverse image. Namely, for $f:S \to T$, $\E \in \SH(S)$, $\F \in \SH(T)$, there is a canonical isomorphism in $\SH(T)$
        $$\F \wedge f_! \E \xrightarrow{\sim} f_!(f^* \F \wedge \E).$$
        \item (Base change) For a cartesian square
        $$\begin{tikzcd}
         S' \arrow[d, "g'"] \arrow[r, "f'"] & T' \arrow[d, "g"] \\
         S \arrow[r, "f"] & T, \\
        \end{tikzcd}$$
        there are canonical natural isomorphisms $g^* f_! \xrightarrow{\sim} f'_!g'^*$ and $g'_*f'^! \xrightarrow{\sim}f^!g_*.$
    \end{enumerate}
\end{prop}

In addition to this, the $\A^1$-invariance property of $\SH(S)$ can be stated in terms of the four functors as follows.

\begin{prop}[$\A^1$-invariance]
\label{prop:hom-invariance}
    If $\pi:E \to X$ is a vector bundle, $\pi^*$ is fully faithful. In particular, the maps $\epsilon_{(\pi^*,\pi_*)}:\id_{\SH(X)} \to \pi_*\pi^*$ and $\eta_{(\pi_\#,\pi^*)}:\pi_\# \pi^* \to \id_{\SH(X)}$ are invertible.
\end{prop}

\begin{defn}
    The family of functors $(f^*,f_*,f_!,f^!)$, together with the monoidal product bifunctor $\wedge$ and the internal-Hom bifunctor $\underline{\Hom}$, are called the \emph{six functors} for $\SH(-)$.
\end{defn}

\begin{rmk}
    Let $p:X \to S$ a smooth map. Working unstably, one easily sees that $p_\#((\id_X)_+)=(X\xrightarrow{p}S)_+=X_+$. Thus, in $\SH(S)$, $p_\# 1_X = p_\# (\Sigma^\infty_{\P^1}(\id_X)_+) \simeq \Sigma_{\P^1}^\infty X_+$. 
\end{rmk}

We also mention a well known reformulation of Morel-Voevodsky localization theorem in the language of six functors. For a proof see for instance \cite[Corollary 3.2.3]{ayoub:sixfunctors}.

\begin{prop}
\label{prop:localization}
    Let $i:Z \to S$ be a closed immersion, with complementary open immersion $j:U \to S$. Then we have the distinguished triangles
    $$i_!i^! \xrightarrow{\eta_{(i_!,i^!)}} \id_{\SH(S)} \xrightarrow{\epsilon_{(j^*,j_*)}} j_*j^* \to i_!i^![1] \; \; \; \text{and} \; \; \; j_!j^! \xrightarrow{\eta_{(j_!,j^!)}} \id_{\SH(S)} \xrightarrow{\epsilon_{(i^*,i_*)}} i_*i^* \to j_!j^![1]$$
    of endofunctors of $\SH(S)$, called the localization sequences.
\end{prop}

Let now $\pi:V \to S$ be a vector bundle, with zero section $s_0:X \hookrightarrow V$. Then the adjunction
$$
\begin{tikzcd}
    \Sigma^V \coloneqq \pi_\# s_{0!}: \SH(S) \arrow[r, shift left] & \SH(S): s_0^! \pi^* =: \Sigma^{-V} \arrow[l, shift left]
\end{tikzcd}
$$
is a self-equivalence of $\SH(S)$.
\begin{defn}
    The endofuntors $\Sigma^V$ and $\Sigma^{-V}$ are called the \emph{$V$-suspension} and the \emph{$V$-desuspension} respectively. Often, one calls them generically the \emph{Thom equivalences}.
\end{defn}

All the results about Thom equivalences that we are going to state here are discussed in \cite[Section 1.5]{ayoub:sixfunctors} and recalled in \cite[Section 2.1]{deglise-jin-khan}.

Thom equivalences are compatible with the monoidal structure, in the sense that, for $V\to S$ vector bundle, and $\E,\F \in \SH(S)$, one has $\Sigma^V\E \wedge \F \simeq \Sigma^V(\E \wedge \F)$, and analogously for $\Sigma^{-V}$. Then, the functor $\Sigma^V(-)$ can be read as $\Sigma^V1_S \wedge (-)$, and the same for $\Sigma^{-V}(-)$. We also have canonical isomorphisms
\begin{equation}
\label{eq:thomcompatibilities}
    f^*\Sigma^V \simeq \Sigma^{f^*V}f^*, \; \; \Sigma^Vf_* \simeq f_*\Sigma^{f^*V}, \; \; \Sigma^Vf_! \simeq f_!\Sigma^{f^*V}, \; \; f^!\Sigma^V \simeq \Sigma^{f^*V}f^!.
\end{equation}

Let now $K(S)$ be the Thomason-Trobaugh $K$-theory space of perfect complexes on $S$ (\cite[Definition 3.1]{Thomason:K-theory}), and let $\sK(S)$ denote the fundamental groupoid of $K(S)$. Let $\sV$ be a locally free sheaf of finite rank over $S$, and let us consider the associated vector bundle $V =\V(\sV) \coloneqq \Spec_{\mathcal{O}_X}\Sym^*(\sV)$. Again, we have endofunctors of $\Sigma^\sV \coloneqq \Sigma^V$ and $\Sigma^{-\sV}\coloneqq \Sigma^{-V}$ of $\SH(S)$. The assignment $\sV \to \Sigma^\sV$ naturally extends to a functor of groupoids $\Sigma^{(-)}: \sK(S) \to \Aut(\SH(S))$, where $\Aut(\SH(S))$ is the groupoid of autoequivalences of $\SH(S)$ with their natural isomorphisms as morphisms.

\begin{rmk}
    Let us note that a distinguished triangle $\sV' \to \sV \to \sV''\to \sV'[1]$ of perfect complexes on $S$ induces canonically an identification $\sV \simeq \sV' + \sV''$ in $\sK(S)$, inducing the isomorphism $\Sigma^\sV \simeq \Sigma^{\sV' +\sV''} \simeq \Sigma^{\sV'}\circ \Sigma^{\sV''}.$
In particular, this gives the isomorphism
\begin{equation}
\label{eq:smash-stablethom}
    \Sigma^\sV 1_S \simeq \Sigma^{\sV'}1_S \wedge \Sigma^{\sV''}1_S.
\end{equation}
The canonical distinguished triangles $\sV\to \sV\oplus\sV'\to \sV'\to \sV[1]$ and $\sV'\to \sV\oplus\sV'\to \sV\to \sV'[1]$ give canonical isomorphisms $\Sigma^\sV\circ\Sigma^{\sV'}\simeq \Sigma^{\sV\oplus \sV'}\simeq \Sigma^{\sV'}\circ \Sigma^\sV$. Similarly, the distinguished triangle $\sV\to 0\to \sV[1]\to \sV[1]$ gives the canonical isomorphism $\Sigma^{\sV}\circ \Sigma^{\sV[1]}\simeq\Sigma^0=\id$, so, for $\sV$ a locally free sheaf, we have $\Sigma^{-\sV}=(\Sigma^{\sV})^{-1}=\Sigma^{\sV[1]}$. We can thus extend the notation $\Sigma^\sV$ to virtual perfect complexes by setting $\Sigma^{\sV-\sV'}:=\Sigma^{\sV}\circ\Sigma^{\sV'[1]}$.
\end{rmk}

The Morel-Voevodsky relative purity theorem \cite[\S 3 Theorem 2.23]{morvoe:homotopytheory} has a reformulation in the stable setting in terms of six functors:

\begin{thm}[\cite{deglise-jin-khan}, item 2.1.8]
\label{thm:MorVoePurity}
    Let $i:Z \hookrightarrow X$ be a closed immersion of smooth schemes over $S$ locally of finite type, with structure maps $p_Z$, $p_X$. Then there are natural isomorphisms
    $$p_{X\#}i_* \simeq p_{Z\#}\Sigma^{N_i} \; \; \; \text{and} \; \; \; \Sigma^{-N_i}p_Z^* \simeq i^!p_X^*,$$
    where $N_i$ is the normal bundle of $i$.
\end{thm}

On the other hand, Ayoub's purity theorem \cite[Section 1.6]{ayoub:sixfunctors} implies the following:

\begin{thm}[\cite{deglise-jin-khan}, item 2.1.7]
\label{thm:ayoubpurity}
    If $f$ is a smooth morphism locally of finite type, there is a canonical isomorphism
    \begin{equation}
    \label{eq:AyoubPurity}
        \Sigma^{T_f}f^* \xrightarrow{\sim} f^!,
    \end{equation}
    where $T_f$ is the relative tangent bundle of $f$.
\end{thm}

\begin{rmk}
\label{rmk:smoothsharp}
If $f$ is a smooth morphism, taking the left adjoints of both sides of the natural isomorphism \eqref{eq:AyoubPurity} gives the natural isomorphism
$f_\#\Sigma^{-T_f}\simeq f_!.$
\end{rmk}

\begin{rmk}
\label{rmk:smoothsharp-properties}
    Proposition \ref{prop:sixfunctors}(4) and Remark \ref{rmk:smoothsharp} imply that, for $f$ smooth, $f_\#$ satisfies a projection formula against $f^*$. Also, $f_{\#}$ verifies the same compatibility as $f_!$ in \eqref{eq:thomcompatibilities} with Thom equivalences, since Thom equivalences commute with each other.
\end{rmk}

\begin{rmk}
\label{rmk:etalemaps}
    If $f$ is étale, $T_f$ is the zero bundle $S \to S$, so $\Sigma^{T_f}=\id_{\SH(S)}$. Thus, Theorem \ref{thm:ayoubpurity} gives an isomorphism $f^* \xrightarrow{\sim} f^!$. If $f$ is an open immersion, this is the isomorphism of Proposition \ref{prop:sixfunctors} (2) . Moreover, in this case we also have $f_\# \simeq f_!$ by Remark \ref{rmk:smoothsharp}.
\end{rmk}

For $\pi:V \to S$ a vector bundle, let $j:V^0 \to V$ be the open complement of the zero section $s_0:S \to V$. By Proposition \ref{prop:localization}, we have a distinguished triangle 
\begin{equation}
\label{eq:thom-loc}
    j_!j^!1_V \to 1_V \to s_{0*}s_0^*1_V \to j_!j!1_V[1].
\end{equation}
Let $\pi^0:V^0 \to S$ be the restriction of $\pi$. By applying $\pi_\#$ to \eqref{eq:thom-loc}, and using $j_! \simeq j_\#$ from Remark \ref{rmk:etalemaps} and $j^! \simeq j^*$ from Proposition \ref{prop:sixfunctors}(2), we obtain the distinguished triangle
\begin{equation}
\label{eq:thom-loc2}
    \pi^0_\# 1_{V^0} \to \pi_\# 1_V \to \pi_\# s_{0*}1_S \simeq \Sigma^V 1_S \to \pi^0_\# 1_{V^0}[1],
\end{equation}
where the isomorphism follows because $s_{0*}\simeq s_{0!}$ by Proposition \ref{prop:sixfunctors}(1). This can be rewritten as
\begin{equation}
\label{eq:thom-loc3}
    \Sigma_{\P^1}^\infty V_+^0 \to \Sigma_{\P^1}^\infty V_+ \to \Sigma^V1_S \to \Sigma_{\P^1}^\infty V_+^0[1],
\end{equation}
from which we see that $\Sigma^V1_S \simeq \Sigma^\infty _{\P^1}\Th(V)$. For this reason, the motivic spectrum $\Sigma^V 1_S$ is also called the \emph{stable Thom space of $V$}.

In view of this discussion on stable Thom spaces, we can say that the relation \eqref{eq:smash-stablethom} is a generalization of Remark \ref{rmk:Thomspaces} (1).

\begin{rmk}
By recalling Remark \ref{rmk:Thomspaces} (2), we have $\Sigma^{\mathcal{O}_S} \simeq \P^1 \wedge (-) \simeq \Sigma^{2,1}$ as endofunctors of $\SH(S)$. By taking wedge powers of $\P^1$, one also gets $\Sigma^{\mathcal{O}_S^r}\simeq \Sigma^{2r,r}$.
\end{rmk}

\subsection{$\GL$-oriented theories}\label{subsection:GL-orient.} From now on, we focus on the case $S=\Spec k$, with $k$ perfect field of exponential characteristic $p\neq2$.

The definition of a $\GL$-orientation is due to Panin \cite[Definition 3.1]{Pan:oriented}. In what follows, we adopt the conventions used in \cite{Ana:Slor}.

\begin{defn}
\label{defn:thomclasstheory}
A \emph{$\GL$-orientation} for a commutative ring spectrum $\sE\in \SH(k)$ consists of the assignment of Thom classes $\th^\sE(V)\in\sE^{2r,r}(\Th(V))=[p_\#\Sigma^V1_X,\Sigma^{2r,r}\sE]_{\SH(k)}$ for each vector bundle $V\to X$, $X\in \Sm/k$, with $r=\rnk(V)$, satisfying the following axioms: 
\begin{equation}\label{enum:ThomClassAxioms}
\end{equation}
\begin{enumerate} 
\item Normalization: For $V=\V(\sO_X^r)$ over some smooth scheme $p:X \to \Spec k$, $\th^\sE(V)\in \sE^{2r, r}(\Th(V))$ is the image of the unit $1_{\sE^{0,0}(X)}$ under the suspension isomorphism $\sE^{0,0}(X) = \sE^{0,0}(p_\#1_X)  \xrightarrow{\sim} \sE^{2r, r}(p_\#\Sigma^{2r,r}1_X) = \sE^{2r, r}(\Th(V))$. 
\item Naturality: Given vector bundles $V\to X$, $W\to Y$, $X,Y\in \Sm/k$, a morphism $f:Y\to X$ and an isomorphism $\alpha:W\xrightarrow{\sim} f^*V$ of vector bundles on $Y$, let $\Th(\alpha,f): \Th(W)\to \Th(V)$ be the obvious induced morphism of motivic spaces. Then
$$\Th(\alpha,f)^*(\th^\sE(V))=\th^\sE(W).$$ 
\item Multiplicativity: Given an exact sequence of vector bundles over $X\in \Sm/k$
\[
0\to V'\to V\to V''\to0
\]
of ranks $r', r,r''$, respectively, one has the corresponding cup product map $\cup:\sE^{2r', r'}(\Th(V'))\otimes \sE^{2r'', r''}(\Th(V''))\to \sE^{2r, r}(\Th(V))$ thanks to the relation \eqref{eq:smash-stablethom}. Then
\[
\th^\sE(V)=\th^\sE(V')\cup\th^\sE(V'').
\]
\end{enumerate}
\end{defn}

Let $V\to X$ be a rank $r$ vector bundle on $X\in \Sm/k$, with structure morphism $p:X\to \Spec k$, and let $\sV$ be the sheaf of sections of $V^\vee$. Since $\Sigma_{\P^1}^\infty \Th(V) \simeq p_\#(\Sigma^\sV1_X)$, we may use the adjunction $p_\#\dashv p^*$ to view $\th^\sE(V)$ as an element $\th_\sE(\sV)\in \Map_{\SH(X)}(\Sigma^{\sV-\sO_X^r}1_X, p^*\sE)$.

\begin{lemma}
\label{lemma:ThomClassNaturalTransformation}
    The assignment $\Locfree(X) \ni \sV\mapsto \th_\sE(\sV) \in \Map_{\SH(X)}(\Sigma^{\sV-\sO_X^r}1_X, p^*\sE),$ with $r=\rnk \sV$, extends to a natural transformation of functors of groupoids from $\sK(X)$ to $\SH(X)_\simeq$
$$
[\th_\sE(-):(\Sigma^{(-)-\sO_X^{\rnk(-)}}1_X)\to c_{p^*\sE}]:\sK(X)\to \SH(X)_\simeq,
$$
where $\SH(X)_\simeq$ is the underlying groupoid of $\SH(X)$, and $c_{p^*\sE}$ is the constant functor with value $p^*\sE$.
\end{lemma}

\begin{proof}
    The functor $\Sigma^{(-)}:\sK(X)\to \Aut(\SH(X))$ gives rise to a functor of groupoids 
    \begin{equation}
    \label{eq:FunctorOfGroupoids}
        \Sigma^{(-)-\sO_X^{\rnk(-)}}1_X:\sK(X)\to \SH(X)_\simeq.
    \end{equation}
Composing \eqref{eq:FunctorOfGroupoids} with the canonical functor of groupoids $i:\Locfree(X) \to \sK(X)$, we get a functor of groupoids $\Locfree(X) \to \SH(X)_\simeq$. Let $\alpha:\sW\to \sV$ be an isomorphism in  $\Locfree(X)$, giving the induced isomorphism $\Th(\alpha):\Sigma^{\sW-\sO_X^{\rnk(\sW)}}1_X\to \Sigma^{\sV-\sO_X^{\rnk(\sV)}}1_X$. By \eqref{enum:ThomClassAxioms}(2), we see that $\th_\sE(\sV)\circ \th(\alpha)= \th_\sE(\sW)$, and we can then see the assignment $\Locfree(X) \ni \sV \mapsto \th_\sE(\sV)$ as a natural transformation of functors of groupoids 
\begin{equation}
    \label{eq:NaturalTransformation}
    [\th_\sE(-):(\Sigma^{(-)-\sO_X^{\rnk(-)}}1_X)\circ i\to c_{p^*\sE}]: \Locfree(X)\to \SH(X)_\simeq,
\end{equation}
from the functor $(\Sigma^{(-)-\sO_X^{\rnk(-)}}1_X)\circ i$ to the constant functor $c_{p^*\sE}$ with value $p^*\sE$.

If $q:\tilde{X}\to X$ is a morphism in $\Sm/k$, the two functors commute with $q^*$. Thus, taking $q$ as a Jouanolou cover of $X$, we can assume that $X$ is affine. 

If $\sV \in \Locfree(X)$, with $X$ affine, we have a canonical isomorphism $$\sV\oplus \sO_X^r-\sO_X^{\rnk(\sV\oplus \sO_X^r)} \xrightarrow{\sim} \sV-\sO_X^{\rnk(\sV)}$$
in $\sK(X)$, which gives the corresponding isomorphism in $\SH(X)$ after applying $\Sigma^{(-)}1_X$. We then have a diagram
   
$$
\begin{tikzcd}
        \Sigma^{\sV\oplus \sO_X^r-\sO_X^{\rnk(\sV\oplus \sO_X^r)}}1_X \arrow[rr, "\sim"] \arrow[dr, swap, "\th_\sE(\sV\oplus \sO_X^r)"] & & \Sigma^{\sV-\sO_X^{\rnk(\sV)}}1_X \arrow[dl, "\th_\sE(\sV)"] \\
        & p^*\sE. &
\end{tikzcd}
$$
It follows by \eqref{enum:ThomClassAxioms}(1,3) that this diagram commutes. Thus, we can extend $\th_\sE(-)$ to differences $\sV-\mathcal{O}_X^r$ by letting $\th_\sE(\sV-\mathcal{O}_X^r)\coloneqq \th_\sE(\sV)$.

We recall from \cite{Gra:Quillen} that the groupoid $\sK(X)$ is obtained from the groupoid $\Locfree(X)$ by inverting the autoequivalence $\oplus \mathcal{O}_X$. Therefore, we can extend the natural transformation \eqref{eq:NaturalTransformation} to a natural transformation
$$[\th_\sE(-):(\Sigma^{(-)-\sO_X^{\rnk(-)}}1_X)\to c_{p^*\sE}]:\sK(X)\to \SH(X)_\simeq.$$
\end{proof}

The natural transformation $\th_\sE(-)$ is multiplicative in the sense that, given $v,v'\in \sK(X)$, we have a canonical isomorphism
\[
\Sigma^{v+v'-\sO_X^{\rnk(v+v')}}1_X\simeq \Sigma^{v-\sO_X^{\text{rank}(v)}}1_X\wedge_X\Sigma^{v'-\sO_X^{\rnk(v')}}1_X.
\]

\begin{lemma}\label{lem:ThomMult}
For $v,v'\in \sK(X)$, $\th_\sE(v+v')$ is the composition
\begin{multline}\label{multline:ThomMult}
\Sigma^{v+v'-\sO_X^{\rnk(v+v')}}1_X\simeq \Sigma^{v-\sO_X^{\rnk(v)}}1_X\wedge_X\Sigma^{v'-\sO_X^{\rnk(v')}}
1_X\\\xrightarrow{\th_\sE(v)\wedge\th_\sE(v')}p^*\sE\wedge_Xp^*\sE\xrightarrow{\mu_{p^*\sE}} p^*\sE\notag
\end{multline}
where $\mu_{p^*\sE}$ is the multiplication on $p^*\sE$.
\end{lemma}
The proof is a direct consequence of the multiplicativity of Thom classes \eqref{enum:ThomClassAxioms}(3).

\begin{defn} For $v\in \sK(X)$, $X\in \Sm/k$, let $\th_{\sE}^f(v):\Sigma^{v-\sO_X^{\rnk(v)}}p^*\sE\to p^*\sE$ be the composition  
\[
\Sigma^{v-\sO_X^{\rnk(v)}}p^*\sE\simeq\Sigma^{v-\sO_X^{\rnk(v)}}1_X\wedge_Xp^*\sE\xrightarrow{\th_\sE(v)\wedge\id_{p^*\sE}}p^*\sE\wedge_X p^*\sE\xrightarrow{\mu_{p^*\sE}} p^*\sE.
\]
Also, for $\sE$ a highly structured motivic commutative ring spectrum, let $\th^f_{\sE\Mod}(v):\Sigma^{v-\sO_X^{\rnk(v)}}p^*\sE\to p^*\sE$ denote the map in $\Mod_{p^*\sE}$ arising from the morphism $\th_\sE(v)$ in $\SH(X)$ by the Free-Forgetful adjunction between $\SH(X)$ and $\Mod_{p^*\sE}$.
\end{defn}

Applying the forgetful functor to $\th^f_{\sE\Mod}(v)$ gives back the map $\th_{\sE}^f(v)$, and for arbitrary $\sE$, we recover $\th_\sE(v)$ from $\th_{\sE}^f(v)$ by composing with the unit $\Sigma^{v-\mathcal{O}_X^{\rnk(v)}} 1_{p^*\sE^{0,0}(X)}$.

\begin{lemma}\label{lem:ThomIso} For $X\in \Sm/k$, $v\in \sK(X)$, $\sE\in \SH(k)$ oriented, the map $\th_{\sE}^f(v):\Sigma^{v-\sO_X^{rnk(v)}}p^*\sE\to p^*\sE$ in $\SH(X)$ is an isomorphism in $\SH(k)$, and for $\sE$ highly structured, $\th^f_{\sE\Mod}(v)$ is an isomorphism in $\Mod_{p^*\sE}$.
\end{lemma}

\begin{proof} Since the forgetful functor $\Mod_{p^*\sE}\to\SH(X)$ reflects isomorphisms, it suffices to see that $\th_{\sE}^f(v)$ is an isomorphism in $\SH(X)$. It follows from the normalization axiom \eqref{enum:ThomClassAxioms}(1) that for $v=[\sO_X^r]$, $\th_\sE(v):1_X\to p^*\sE$ is the unit in $\sE^{0,0}(X)$, which easily implies the lemma for $v=[\sO_X^r]$. In general, we take a Zariski open cover $\sU=\{U_i\}_i$ for $X$ for which the restriction of $v$ to each $U_i$ is $r_i[\sO_{U_i}]$ for some integer $r_i$. The result then follows from the gluing axiom for the sheaf $\SH(-)$ on $X_\Zar$.
\end{proof}

Analogously to $\th_\sE(-)$, $\th_{\sE}^f(-)$ can be viewed as a natural isomorphism of functors of groupoids from $\sK(X)$ to $\SH(X)_\simeq.$

\begin{prop}
\label{prop:K_0ExtensionTh(-)}
Let $\sE\in \SH(k)$ be oriented and take $X\in \Sch/k$. \\[5pt]
1. The functor of groupoids $\Sigma^{(-)}p_X^*\sE:\sK(X)\to \SH(X)_\simeq$ given by $v\mapsto \Sigma^vp_X^*\sE$ descends to a map
\[
\Sigma^{(-)}p_X^*\sE:K_0(X)\to \Obj(\SH(X)),\ v\mapsto \Sigma^vp_X^*\sE.
\]
2. The natural isomorphism $[\th_{\sE}^f(-):(\Sigma^{(-)-\sO_X^{\rnk(-)}}p_X^*\sE)\to c_{p_X^*\sE}]:\sK(X)\to \SH(X)_\simeq$ descends to a natural isomorphism
\[
[\th_{\sE}^f(-):(\Sigma^{(-)-\sO_X^{\rnk(-)}}p_X^*\sE\to c_{p_X^*\sE}]:K_0(X)\to \SH(X)_\simeq. \]
\end{prop}

\begin{proof} (2) follows directly from (1).

For (1), let $\phi:v\to v'$ be an isomorphism in $\sK(X)$,  let $r=\rnk(v)=\rnk(v')$ and let $\Sigma^\phi(p_X^*\sE):\Sigma^{v}p_X^*\sE\simeq \Sigma^{v'}p_X^*\sE$ be the isomorphism induced from $\Sigma^\phi:\Sigma^v \simeq \Sigma^{v'}$. This gives us the commutative diagram of isomorphisms in $\SH(X)$
\[
\xymatrixcolsep{40pt}
\xymatrix{
\Sigma^{v-\sO_X^r}p_X^*\sE\ar[d]_{\Sigma^{-\sO_X^r}\Sigma^\phi(p_X^*\sE)}\ar[r]^-{\th_{\sE}^f(v)}&p_X^*\sE\\
\Sigma^{v'-\sO_X^r}p_X^*\sE\ar[ru]_{\th_{\sE}^f(v')}
}
\]
In other words $\Sigma^{\phi}(p_X^*\sE)=\Sigma^{\sO_X^r}(\th_{\sE}^f(v)^{-1}\circ \th_{\sE}^f(v))$. This shows that $\Sigma^{\phi}(p_X^*\sE)$ is independent of the choice of $\phi:v\to v'$, proving (1). 
\end{proof}

For $V_1, V_2$ vector bundles on $X\in \Sm/k$, we have the Thom class $\th_{\sE}^f(V_1-V_2)$ defined as $\th_{\sE}^f([\sV_1]-[\sV_2])$, where $\sV_i$ is the sheaf of sections of $V_i^\vee$.
\begin{lemma}\label{lem:ThomClassVirtualBundles} Let $V_1\to X$, $V_2\to X$ vector bundles on $X\in \Sm/k$ of rank $r_1, r_2$, and let $\sV_1, \sV_2$ be the respective locally free sheaves of sections of $V_1^\vee, V_2^\vee$. Then $\th_{\sE}^f(V_1-V_2)\in [\Sigma^{[\sV_1]-[\sV_2]-(r_1-r_2)[\sO_X]}p^*\sE, p^*\sE]_{\SH(X)}$ is the composition
\begin{multline*}
    \Sigma^{[\sV_1]-[\sV_2]-(r_1-r_2)[\sO_X]}p^*\sE\xrightarrow{\Sigma^{r_2[\sO_X]-[\sV_2]}\th_{\sE}^f(V_1)}\Sigma^{-[\sV_2]+r_2[\sO_X]}p^*\sE \\ \xrightarrow{(\Sigma^{-[\sV_2]+r_2[\sO_X]}\th_{\sE}^f(V_2))^{-1}}p^*\sE
\end{multline*}
In particular, for $V\to X$ a vector bundle of rank $r$, we have
$$\th_{\sE}^f(-V)=(\Sigma^{-[\sV]+r[\sO_X]}\th_{\sE}^f(V))^{-1}.$$
\end{lemma} 

The proof is a direct computation.

\begin{defn}
\label{defn:DualMap_Properpushforward}
\begin{enumerate}
    \item Let $f:X \to Y$ a proper map in $\Sm/k$, and $p_X, p_Y$ the structure maps of the schemes. Since $f_*\simeq f_!$ by Proposition \ref{prop:sixfunctors} (1), we have the natural transformation
  $$p_{Y!}\xrightarrow{p_{Y!}\epsilon_{(f^*,f_*)}}p_{Y!}f_*f* \simeq p_{Y!}f_!f^*=p_{X!}f^*.$$
    Applying this to $1_Y$ gives the \emph{dual map} of $f$ in $\SH(k)$:
    $$f^\vee: p_{Y!}1_Y \to p_{X!}1_X.$$
    \item In the same setting as (1), we suppose $X,Y$ have dimension $d_X,d_Y$ respectively, and $f$ has relative dimension $d$. Let $\sE \in \SH(k)$ $\GL$-oriented. The dual map $f^\vee$ gives rise to a pullback $(f^\vee)^*:\sE^{a,b}(p_{X!}1_X) \to \sE^{a,b}(p_{Y!}1_Y)$, and Remark \ref{rmk:smoothsharp} gives $p_{X!}\simeq p_{X\#}\Sigma^{-T_{p_X}}=p_{X\#}\Sigma^{-T_X}$. Then, for all $a,b$ in $\Z$, we have
    \begin{multline*}
        \sE^{a,b}(p_{X!}1_X) = [p_{X!}1_X,\Sigma^{a,b}\sE]_{\SH(k)} \cong [p_{X\#}\Sigma^{-T_X}1_X, \Sigma^{a,b}\sE]_{\SH(k)} \cong [1_X,\Sigma^{T_X}p_X^*\Sigma^{a,b}\sE]_{\SH(X)} \\ \cong [1_X,\Sigma^{a,b}\Sigma^{T_X}p_X^*\sE]_{\SH(X)} \xrightarrow[\sim]{\circ \Sigma^{a,b}\th_{\sE}^f(T_X)}[1_X,\Sigma^{a,b}\Sigma^{2d_X,d_X}p_X^*\sE]_{\SH(X)} \\ \cong [1_X,p_X^*\Sigma^{a+2d_X,b+d_X}\sE]_{\SH(X)} \cong [p_{X\#}1_X,\Sigma^{a+2d_X,b+d_X}\sE]_{\SH(k)}=\sE^{a+2d_X,b+d_X}(X).
    \end{multline*}
    In the same way, we have $\sE^{a,b}(p_{Y!}1_Y) \cong \sE^{a+2d_Y,b+d_Y}(Y).$ Thus, by reindexing, we have that $(f^\vee)^*$ induced the \emph{proper pushforward for $\sE$ along $f$}:
    $$f_*:\sE^{a,b}(X) \to \sE^{a-2d,b-d}(Y).$$
\end{enumerate}
\end{defn}

It is a standard computation to verify the following compatibilities between pullback and proper pushforward:

\begin{prop}\label{prop:PushPull} Let $\sE\in \SH(k)$ be an oriented commutative ring spectrum.\\[5pt]
1.(Push-pull formula) Let
\[
\xymatrix{
Y'\ar[r]^{g'}\ar[d]^{f'}&Y\ar[d]^f\\
X'\ar[r]^g&X
}
\]
be a transverse cartesian square, with $X,X', Y,Y'$ in $\Sm/k$ and with $f$ proper. Then $f'$ is also proper and $g^*f_*=f'_*g^{\prime *}$.\\[2pt]
2. (Projection formula) Let $f:Y\to X$ be a proper map in $\Sm/k$. Then for $x\in \sE^{*,*}(X)$, $y\in \sE^{*',*'}(Y)$, we have $f_*(f^*(x)\cdot y)=x\cdot f_*(y)$.
\end{prop}

\begin{defn}\label{defn:stwist} For $Y\in \Sm/k$ of dimension $d$ over $k$, a {\em stable  twist of $-T_Y$} is a pair  $(v, \vartheta)$, where $v$ is a virtual vector bundle of virtual rank $-d$, and $\vartheta$ is an isomorphism $\Sigma^{-T_Y}1_Y\xrightarrow{\sim} \Sigma^v1_Y$ in $\SH(Y)$.  
\end{defn}

\begin{defn}\label{defn:twistclass} Let $\sE\in \SH(k)$ be an oriented commutative ring spectrum. For $p_Y:Y\to \Spec k\in \Sm/k$ proper of dimension $d$ over $k$, let $(v, \vartheta)$ be a stable twist of $-T_Y$. Let $\th_\sE(v)':\Sigma^{2d,d}p_{Y\#}\Sigma^{v}1_Y\to
\sE$ be the map corresponding to $\th_\sE(v):\Sigma^{2d,d}\Sigma^v1_Y\to p_Y^*\sE$ by adjunction. We define the class $[Y, v,\vartheta]_\sE\in \sE^{-2d,-d}(\Spec k)$ as the composition
\[
1_k\xrightarrow{p_Y^\vee}p_{Y\#}\Sigma^{-T_Y}1_Y\xrightarrow{p_{Y\#}(\vartheta)}
p_{Y\#}\Sigma^{v}1_Y\xrightarrow{\Sigma^{-2d,-d}\th_\sE(v)'}\Sigma^{-2d,-d}\sE,
\]
where $p_Y^\vee$ is the dual map of $p_Y$ of Definition \ref{defn:DualMap_Properpushforward} (1) composed with the purity isomorphism $p_{Y!}1_Y\simeq p_{Y\#}\Sigma^{-T_Y}1_Y$ (Remark \ref{rmk:smoothsharp}). We use the shorthand $[Y]_\sE :=[Y,-T_Y, \id]_\sE$.
\end{defn}

\begin{defn}\label{defn:Degree}
For $X \in \Sm/k$ proper of dimension $d_X$ and structure map $p_X$, we have the \emph{degree map} for motivic cohomology $\deg_k:H\Z^{2d_X,d_X}(X)\to \Z$, defined by the proper pushforward $p_{X*}$, since $H\Z^{0,0}(\Spec k)=\Z$ by Theorem \ref{thm:MazWeiHZ}.
\end{defn}

\subsubsection*{Chern classes}

For $\sE$ $\GL$-oriented, $V \to X$ rank $r$ vector bundle, X $\in \Sm/k$, one has Chern classes $c_i(V) \in \sE^{2i,i}(X)$, $i=1,\ldots, r$ satisfying the standard axioms, see \cite[Definition 3.26]{Pan:oriented}. For $\sV$ the sheaf of sections of a vector bundle $V$, we define $c_i(\sV)$ to be $c_i(V)$.

For $i>0$, the Chern classes $c_i(V)$ are nilpotent elements in the graded ring $\oplus_n\sE^{2n,n}(X)$ (see \cite[Theorem 3.27]{Pan:oriented}). 

Panin has shown (see\cite[Lemma 3.33, Theorems 3.5,  3.27,  3.36]{Pan:oriented}) that, for a commutative ring spectrum $\sE$, giving a theory of Chern classes is equivalent to giving a theory of Thom classes. Under this correspondence, given a vector bundle $V\to X$, the top Chern class $c_{\rnk(V)}(V)\in\sE^{2{\rnk(V)}, {\rnk(V)}}(X)$ is the pullback $z_0^*\th^\sE(V)$ of $\th^\sE(V)$ along the zero section $z_0:X_+ \to \Th(V)$ in $\sH_\bullet(k)$. This yields in particular an expression for the unique non-trivial Chern class $c_1(L)\in \sE^{2, 1}(X)$ for a line bundle $L\to X$. Also, $c_{\rnk(V)}(V)$ can also be described in the following terms (Panin showed this for line bundles but the argument extends to finite rank).

\begin{lemma}\label{lem:FirstChernClassFacts} Let $V\to X$ be a rank $r$ vector bundle on $X\in \Sm/k$, and let $\sE\in \SH(k)$ be an oriented ring spectrum. If $s_0:X\to V$ is the zero section and $s:X\to V$ an arbitrary section, then $c_r(V)=s^*(s_{0*}(1^\sE_X))$.
\end{lemma}

The following is \cite[Theorem 3.9]{Pan:oriented}.

\begin{thm}[Projective Bundle Formula]\label{thm:PBF} Let $\sE\in \SH(k)$ be an oriented commutative ring spectrum. Let $X$ be in $\Sm/k$, and let $V\to X$ be a rank $n+1$ vector bundle on $X$. Let $\P(V)$ denote the projective vector bundle $\operatorname{Proj}(\Sym^*\sV)\xrightarrow{q}X$, with tautological quotient invertible sheaf $q^*\sV\twoheadrightarrow\sO_V(1)$. Let us write $\xi:=c_1(\sO_V(1))\in \sE^{2,1}(X)$. Then $\sE^{*,*}(\P(V))$ is a free $\sE^{*,*}(X)$-module, via the pullback $q^*$, with basis $1, \xi,\ldots, \xi^n$.
\end{thm}

The lower Chern classes $c_i(V)\in \sE^{2i,i}(X)$, $0<i<{\rnk(V)}$, for a vector bundle $V\to X$ are determined using Grothendieck's formula
\begin{equation}\label{eqn:GrothFormula}
\xi^{\rnk(V)}+\sum_{i=1}^{\rnk(V)}(-1)^ic_i(V)\cdot \xi^{\rnk(V)-i}=0;\quad \xi:=c_1(\sO_V(1))\in \sE^{2,1}(\P(V)).
\end{equation}
Combining Theorem \ref{thm:PBF} with the Grothendieck formula \eqref{eqn:GrothFormula} gives a description of $\sE^{*,*}(\P(V))$ as a $\sE^{*,*}(X)$-algebra.

For any $\GL$-oriented theory $\sE$ we have an isomorphism of $\Z$-graded rings $\sE^{2*,*}(k)[[x_1,x_2]]\cong \sE^{2*,*}(\P^\infty\times\P^\infty)$, with $x_i$ mapping to $p_i^*(c_1(\sO(1)))$, and via this isomorphism a uniquely defined $F=F(x_1,x_2)\in \sE^{2*,*}(k)[[x_1,x_2]]$ such that $F(x_1,x_2)=c_1(x_1 \otimes x_2)$. Panin showed \cite[Proposition 3.37]{Pan:oriented} that given two line bundles $L_1, L_2$ on $X\in \Sm/k$, one has $c_1(L_1\otimes L_2)=F(c_1(L_1), c_1(L_2))$, where the right-hand side makes sense because $c_1(L_i)$ are nilpotent, and by \cite[Proposition 3.38]{Pan:oriented} $F(x, y)$ is a rank one formal group law over $\sE^{2*,*}(k)$.

\begin{exmp}
\label{exmp:orientedHZ}
    Motivic cohomology $H\Z$ has a canonical orientation with additive formal group law $F(x,y)=x + y$. See for example \cite[Section 11.3]{deglise:mixmot}. 
\end{exmp}

\subsubsection*{Newton classes}

For $X\in \Sm/k$ and $V\to X$ a rank $r$ vector bundle, with associated flag bundle $\pi:\text{Fl}(V)\to X$, the pullback $\pi^*V$ admits a canonical filtration
\[
0=F^{r+1}V\subset F^rV\subset\ldots\subset F^1V\subset F^0V=\pi^*V
\]
with subquotients $L_i:=F^iV/F^{i+1}V$ all line bundles. For $\sE\in \SH(k)$ an oriented commutative ring spectrum, the {\em $\sE$-Chern roots} of $V$ are the element $\xi_i^\sE(V):=c_1^\sE(L_i)\in \sE^{2,1}(\text{Fl}(V))$. If $\sE=H\Z$, we write $\xi_i$ for $\xi_i^{H\Z}$, and we call the classes $\xi_i(V)$ simply the \emph{Chern roots of $V$}. The splitting principle tells us that $\pi^*:\sE^{2*,*}(X)\to \sE^{2*,*}(\text{Fl}(V))$ is injective, with image the elements of the form $P(\xi^\sE_1,\ldots, \xi^\sE_r)$ for $P\in \sE^{2*,*}[x_1,\ldots, x_r]^{\Sigma_r}$ a symmetric polynomial.

\begin{defn}
\label{defn:NewtonClassesAndStuff}
1. Let $m(x,t)$ be the formal product $m(x,t):=\prod_{i,j=1}^\infty 1+x_i^jt_j$, and $I:=(i_1\ge i_2\ge\ldots \ge i_r\ge 0)$ a partition. We let $|I|:=\sum_j i_j$, and $t_I:=t_{i_1}\cdots t_{i_r}$. Writing $m(x,t)$ as $\sum_I m_I(x)t_I$ gives the {\em monomial symmetric function} $m_I(x_1, x_2,\ldots)$, which is a homogeneous polynomial in $x_1, x_2,\ldots$ of degree $|I|$. \\[2pt]
2. Let $\sE\in \SH(k)$ be an oriented commutative ring spectrum and $V\to X$ a rank $r$ vector bundle on $X\in \Sm/k$. For a partition $I$, the {\em $I$-th Conner-Floyd Chern class} $c_I^\sE(V)\in \sE^{2|I|,|I|}(X)$ is the element corresponding to $m_I(\xi_1^\sE(V),\ldots, \xi_r^\sE(V), 0,0,\ldots)$ via the splitting principle. For $\sE=H\Z$, we write simply  $c_I(V)$. \\[2pt]
3. We call the \emph{$n$-th Newton class of $V$} the Conner-Floyd Chern class $c_{I}(V)$ relative to the partition $I=(n)$, namely:
$$c_{(n)}(V) \coloneqq \sum_{i=1}^{r}\xi_i(V)^n \in H\mathbb{Z}^{2n,n}(X).$$
4. For a variety $X$, smooth and proper over $k$ of dimension $n$, $s_n(X)$ denotes the number
    $$s_n(X) \coloneqq \text{deg}_k(c_{(n)}(T_X)) \in \mathbb{Z}.$$
\end{defn}

\begin{lemma}
\label{lemma:newtonclasses}
    \begin{enumerate}
        \item[(1)] If $V\to X$ is a vector bundle on $X\in \Sm/k$, and $n>\dim_kX$, then 
        $c_{(n)}(V)=0$.
        \item[(2)] Newton classes are additive: if $V_1, V_2 \to X$ are two vector bundles, we have
        $$c_{(n)}(V_1 \oplus V_2) = c_{(n)}(V_1)+c_{(n)}(V_2).$$
        \item[(3)] Newton classes are natural: given $V\to X$ a vector bundle and $f:Y\to X$ a morphism in $\Sm/k$, we have $c_{(n)}(f^*(V))=f^*(c_{(n)}(V))$ in $H\Z^{2n,n}(Y)$.
        \item[(4)] If $r \ge 2$ and $X = \prod_{i=1}^rX_i$ is the product of $r$ smooth proper varieties, with $1 \le \dim_k X_i <n$ for all $i$ and $n$ a positive integer, then $c_{(n)}(T_X)=0$. In particular, the characteristic numbers $s_n$ vanish on decomposable varieties.
    \end{enumerate}
\end{lemma}

\begin{proof}
    (1) follows from the fact that for $X\in \Sm/k$,  $H\mathbb{Z}^{2n,n}(X)=\CH^n(X)=0$ for $n>\dim_kX$. (2) follows from the fact that if $V_1$ has Chern roots $\xi_1, \ldots, \xi_r$ and $V_2$ has Chern roots $\xi_{r+1}, \ldots, \xi_{r+s}$, then $V_1\oplus V_2$ has Chern roots $\xi_1, \ldots, \xi_{r+s}$. (3) follows from the naturality of the Chern roots with respect to pullback. 
    (4) follows from the other three, since $T_X\simeq \bigoplus_{i=1}^r   p_i^*T_{X_i}$, with $p_i:X \to X_i$ the projections. 
\end{proof}

\begin{rmk}
\label{rmk:NewtonClasses} For a partition $I=i_1\ge i_2\ge \ldots\ge i_s\ge 0$, let $I'=(i'_1,i'_2, \ldots, i'_{i_1}, 0,0,\ldots)\in \N^\infty$ be the tuple defined by $i'_m=\#\{j\mid i_j=m\}$. This gives $t_I=t^{I'}=\prod_{j=1}^{i_1}t_j^{i'_j}$. The generating function $m(x,t)$ can be rewritten as $m(x,t)=\sum_Im_I(x)t^{I'}$. If we define $\langle I'\rangle:= \sum_ij\cdot i'_j$, we have  $|I|=\langle I'\rangle$. Given a vector bundle $V\to X$ on $X\in \Sm/k$, we have the \emph{Conner-Floyd Chern polynomial} 
\[
c^\CF(V)=\sum_Ic_I(V)t_I=\sum_Ic_I(V)t^{I'}.
\]
For $0\to V'\to V\to V''\to 0$ an exact sequence of vector bundles on $X$, the Chern roots of $V$ are those of $V'$ together with those of $V''$, which easily implies $c^\CF(V)= c^\CF(V')\cdot c^\CF(V'')$. This shows that the assignment $V\mapsto c^\CF(V)$ extends to a group homomorphism
\[
c^\CF(-):K_0(X)\to (1+\oplus_JH\Z^{2\langle J\rangle, \langle J\rangle}(X)t^J)^\times.
\]
In particular, for each partition $I$ (for instance, $I=(n)$), we have a well-defined map (of sets)
\[
c_I:K_0(X)\to H\Z^{2|I|,|I|}(X).
\]
 \end{rmk}

\subsection{Algebraic cobordism}

The universal oriented cohomology theory is Voevodsky's algebraic cobordism $\MGL$. Briefly, we can consider the spaces $\BGL_r:=\colim_n\Gr(r,n) \in \Spc(k)$ with respect to the natural inclusions $\Gr(r,n)\hookrightarrow\Gr(r,n+1)$ and $\MGL_r=\Th(E_r) \in \Spc_\bullet(k)$ with $E_r\to \BGL_r$ tautological bundle. There are natural inclusions $i_r:\BGL_r\hookrightarrow\BGL_{r+1}$, and isomorphisms $i_r^*E_{r=1}\simeq \mathcal{O}\oplus E_r$, which induce the bonding maps $\epsilon_r:\Th(\mathcal{O}\oplus E_r)\simeq T \wedge \MGL_r \to \MGL_{r+1}$, with $T=\A^1/(\A^1\setminus\{0\})$. $\MGL$ is then the $T$-spectrum $(\MGL_0,\MGL_1,\ldots)$, ad since $T$ is weakly equivalent to $\P^1$, the homotopy category of $T$-spectra is equivalent to $\SH(k)$.

It is easy to see that this is equivalent to define $\MGL \in \SH(k)$ as the colimit $\colim_r\Sigma^{-2r,-r}\Sigma_{\P^1}^\infty\MGL_r$, with respect to the maps $\Sigma^{-2r-2,-r-1}\Sigma^\infty_{\P^1}\epsilon_r$. In particular, we have natural maps $\alpha_r:\Sigma_{\P^1}^\infty\MGL_r \to \Sigma^{2r,2}\MGL$.
\begin{rmk}
    It is shown in \cite[Section 2.1]{Panin:algcob} that $\MGL$ has the structure of a highly structured motivic commutative ring spectrum.
\end{rmk}

\begin{rmk}
\label{rmk:oriented_MGL}
$\MGL$ has a natural $\GL$-orientation where, for $X\in \Sm/k$ affine, $V\to X$ a rank $r$ vector bundle, and $\eta_V:X\to \BGL_r$ a classifying map of $V$, which is defined up to $\A^1$-homotopy, the $\MGL$-Thom class of $V$ is given by
$$\th^\MGL(V):\Sigma_{\P^1}^\infty \Th(V) \xrightarrow{\Th(\eta_V)}\Sigma_{\P^1}^\infty \MGL_r \xrightarrow{\alpha_r}\Sigma^{2r,r}\MGL.$$
A standard Jouanolou argument allows us to extend this construction to the case where $X$ is not affine. All the axioms for a theory of Thom classes are straightforward to check.
\end{rmk}

The $\GL$-orientation on $\MGL$ of Remark \ref{rmk:oriented_MGL} is universal in the following sense:

\begin{thm}[\cite{Panin:algcob}, Theorem 2.7]
    For any commutative ring spectrum $\sE$, the map of sets
    \begin{multline*}
        \{\text{morphisms of commutative ring spectra} \;  \MGL \to \sE \; \text{in} \; \; \SH(k)\} \\ \to \{\text{orientations on the ring spectrum} \; \sE\}
    \end{multline*}
    that takes a map $\varphi: \MGL \to \sE$ to the $\GL$-orientation of $\sE$ given by $\th^\sE(V):= \Sigma^{2r,r}\varphi \circ \th^\MGL(V)$ is a bijection. 
\end{thm}

The $\GL$-orientation on $\MGL$ induces a formal group law on $\MGL^*$. This in turn gives a map $\Laz \to \MGL$ from the Lazard ring, which is the ring of coefficient of the universal formal group law $F^{\text{univ}} \in \Laz[[x,y]]$. Spitzweck (\cite{Spitzweck:AlgCob}) and Hoyois (\cite{Hoyois:AlgCob}) proved that, after inverting the exponential characteristic of $k$, this map is an isomorphism of graded ring, and that $\MGL^{2n+m,n}(\Spec k)[1/p]=0$ for all $m>0$. Also, it is well-known since Quillen that $\Laz$ is isomorphic to the polynomial ring $\Z[x_1,x_2,\ldots]$ with $x_i$ in degree $-i$ for all $i>0$ (see e.g. \cite[Theorem 4.4.9]{koch:bordism} for a proof). In particular, we have the following.
\begin{thm}
\label{thm:comparisonLazard}
After inverting the exponential characteristic of $k$, $\MGL^*$ is isomorphic to the polynomial ring $\Z[x_1,x_2, \ldots]$ with $\mid x_i \mid =-i$.
\end{thm}

Through a well-established algebraic version of the Pontryagin-Thom construction (see \cite[Section 3]{lev:ellcoh}), we can associate to any smooth proper variety over $k$ a class in the algebraic cobordism ring $\MGL^*$, as follows.

\begin{defn}
    \label{defn:MGLclasses}
If $X$ is a smooth proper variety $X \in \Sm/k$ of dimension $d$, we call the \emph{$\MGL$-class of $X$} the class $[X]_\MGL=[X,-T_X,\id]_\MGL \in \MGL_{2d,d}(\Spec k)=\MGL^{-2d,-d}(\Spec k)$ constructed as in Definition \ref{defn:twistclass}. Explicitly, if $p:X \to \Spec k$ is the structure map, $[X]_\MGL$ is given by the composition
\begin{multline*}
    1_k \xrightarrow{p^\vee} p_\# \Sigma^{-T_X}1_X \xrightarrow{p_\#\Sigma^{-T_X}\epsilon_{p^*\MGL}} p_\#\Sigma^{-T_X}p^*\MGL \xrightarrow{p_\#\Sigma^{-2d,-d}\th_{\MGL}^f(-T_X)} \\ p_\#\Sigma^{-2d,-d} p^*\MGL \simeq p_\#p^*\Sigma^{-2d,-d}\MGL \xrightarrow{\eta_{(p_\#,p^*)}(\Sigma^{-2d,-d}\MGL)}\Sigma^{-2d,-d}\MGL.
 \end{multline*}
\end{defn}

\begin{rmk}
    Taking $p_*:\MGL^{-2d,-d}(p_{X\#}\Sigma^{-T_X}1_X) \cong \MGL^{0,0}(X) \to \MGL^{-2d,-d}(\Spec k)$ the proper pushforward, $[X]_\MGL$ can be seen as the class $p_*\th^{\MGL}(-T_X)$.
\end{rmk}

\begin{prop}[\cite{lev:ellcoh}, Theorem 3.4 (4)]
\label{prop:MGL-classes}
    After inverting  the exponential characteristic of $k$, $\MGL_{2d,d}(\Spec k)$ is generated by the classes $[X]_\MGL$, for $X$ smooth proper $k$-scheme with $\dim_k(X)=d$.
\end{prop}

It is well known (see for example \cite[Theorem 2.2]{Panin:algcob}) that $H\Z^{*,*}(\BGL_r)=H\Z^{*,*}(k)[c_1,\ldots, c_r],$ with $c_r$ in bidegree $(2r,r)$ being the usual $r$-th Chern class $c_r(E_m)$ for $m\ge r$, defined by the Grothendieck formula \eqref{eqn:GrothFormula}. Moreover, the cohomology pullback $i_r^*:H\Z^{a,b}(\BGL_{r+1})\to H\Z^{a,b}(\BGL_r)$ is surjective, thus, the inverse system $\{H\Z^{a,b}(\BGL_r), i_r^*\}_r$ satisfies the Mittag-Leffler condition, and its $\lim^1$-term vanishes. Using this and the Thom isomorphism $H\Z^{a+2r, b+r}(\MGL_r)\cong H\Z^{a,b}(\BGL_r)$ given by the $\GL$-orientation of $H\Z$, it is easy to prove the following.

\begin{prop}
\label{prop:MotCohOfMGL}
    We have $H\Z^{*,*}(\MGL) \cong H\Z^{*,*}(k)[c_1,c_2, \ldots],$
    with $c_r$ in bidegree $(2r,r)$. The same holds if we replace $H\Z$ with $H\Z/\ell$ for any prime $\ell$.
\end{prop}

For each $r\ge 0$, we have the flag bundle $\Fl_r :=\Fl(E_r)\xrightarrow{\pi_r} \BGL_r$ associated to the tautological bundle $E_r \to \BGL_r$. By construction of the flag bundles, the inclusions $i_r:\BGL_r \hookrightarrow \BGL_{r+1}$ induce maps $\Fl(i_r):\Fl_r \to \Fl_{r+1}$, and we have
$$\Fl(i_r)^*(\pi_{r+1}^*E_{r+1})\simeq \pi_r^*(i_r^*E_{r+1})\simeq \pi_r^*(E_r \oplus \mathcal{O}_{\BGL_r})\simeq L_1^{(r)} \oplus \ldots \oplus L_r^{(r)} \oplus \mathcal{O}_{\Fl_r}.$$
Since $\pi_{r+1}^*E_{r+1}\simeq L_1^{(r+1)}\oplus \ldots \oplus L_{r+1}^{(r+1)}$, we can suppose $\Fl(i_r)^*(L_j^{(r+1)}) = L_j^{(r)}$ for all $j<r$, and then $\Fl(i_r)^*\xi_j(E_{r+1})=\xi_j(E_r)$ for all $j<r$. We define the bundle $\Fl \to \BGL$ through $\Fl:=\colim_r(\Fl_0 \xrightarrow{\Fl(i_0)} \Fl_1 \xrightarrow{\Fl(i_1)}\ldots)$, and the class $\xi_j(E_r)$, for some $r>>0$, gives then a class $\xi_j \in \Fl$. The maps $\pi_r:\Fl_r \to \BGL_r$ induce a map $\Fl\to \BGL$ by passing to the colimit, and $H\Z^{*,*}(\Fl)\cong \lim_rH\Z^{*,*}(\Fl_r)$, thus, the map $\pi^*:H\Z^{2*,*}(\BGL)\to H\Z^{2*,*}(\Fl)$ satisfies the splitting principle.

\begin{defn}
\label{defn:univNewtonClasses}
    We call the \emph{$n$-th universal Newton class} the class $c_{(n)} \in H\Z^{2n,n}(\MGL)\cong H\Z^{2n,n}(\BGL)$ corresponding to the symmetric function $\sum_i \xi_i^n \in H\Z^{2n,n}(\Fl)$ as in Definition \ref{defn:NewtonClassesAndStuff}(4).
\end{defn}

Let $h:= h_{H\Z}:\MGL_{*,*}(\Spec k) \to H\Z_{*,*}(\MGL)$ be the motivic Hurewicz map as defined in Remark \ref{rmk:HurewiczMap}. For any cohomological class $\alpha \in H\mathbb{Z}^{a,b}(\MGL)$ and any homological class $\gamma \in H\mathbb{Z}_{c,d}(\MGL)$, we have the pairing $\langle \alpha,\gamma \rangle \in H\mathbb{Z}^{a-c,b-d}(\Spec k)$, as in \eqref{eqn:Pairing} with $\sE=H\Z$.

\begin{defn}
    For $x \in \MGL^{-n}=\MGL_{2n,n}(\Spec k)$, and $c_{(n)} \in H\Z^{2n,n}(\MGL)$ the $n$-th universal Newton class (Definition \ref{defn:univNewtonClasses}), we denote by $c_{(n)}(x)$ the integer
$$c_{(n)}(x) := \langle c_{(n)},h(x) \rangle \in H\Z^{0,0}(\Spec k)=\Z.$$
\end{defn}

If $\ell$ is an odd prime different from $\chr k$, \cite{lev:ellcoh} gives the following generating criterion for the $\ell$-adic completion $(\MGL^*)^\wedge_\ell$.

\begin{prop}[\cite{lev:ellcoh}, from proof of Theorem B, Section 6.6]
\label{prop:CriterionMGL}
    A family of elements $\{x_d'\}_{d \ge 1}$, with $x_d' \in (\MGL^{-d})^\wedge_\ell$ forms a family of polynomial generators of $(\MGL^*)^\wedge_\ell$ if and only if
    \begin{equation*}
        \nu_\ell(c_{(d)}(x_d'))=
        \begin{cases}
            1 \; \; \; \; \text{for} \; \; d= \ell^r-1, \; r\ge 1 \\
            0 \; \; \; \text{for} \; \; d \neq \ell^r-1 \; \forall r \ge 1.
        \end{cases}
    \end{equation*}
\end{prop}

\subsection{$\Sp$-oriented theories}
We recall that a rank $2r$ \emph{symplectic vector bundle} over $X \in \Sm/k$ is a pair $(V,\omega)$, where $V \to X$ is a rank $2r$ vector bundle, and $\omega: V \otimes V \to \mathcal{O}_X$ is a symplectic form on it, namely a non-degenerate antisymmetric bilinear form.

\begin{exmp}
\label{exmp:sympl_bundles}
\begin{enumerate}
    \item The trivial rank $2r$ vector bundle $\mathcal{O}_X^{2r}$ can be equipped with the standard rank $2r$ symplectic form $\phi_{2r}$, represented by the standard $2r \times 2r$ symplectic matrix
$$\phi_{2r} =
\begin{pmatrix}
     & & &  &  & 1\\
     & &  &  &\ldots & \\
     & & & 1& &\\
     & & -1 &  & &  \\
     & \ldots &  & & & \\
     -1 & & & & &
\end{pmatrix}.$$
\item For $V\to X$ any rank $r$ vector bundle, the rank $2r$ vector bundle $V \oplus V^\vee \to X$ can always be equipped with the rank $2r$ standard symplectic form $\phi_{2r}$. Concretely, this is the form
    $$\phi_{2r}((v,v^*),(w,w^*)) \coloneqq <w^*,v> -<v^*,w>,$$
    where $<-,->:V^\vee \times V \to \mathcal{O}_X$ is the canonical pairing.
\end{enumerate}
\end{exmp}

The axioms defining a $\GL$-orientation as in Definition \ref{defn:thomclasstheory} can be used to define weaker kinds of orientation, namely a theory of Thom classes for vector bundles carrying additional structures, as elaborated by Panin and Walter in \cite{Panin-Walter:BO} and \cite{panwal:grass}. We are interested in the following.

\begin{defn}
    An $\Sp$-orientation for a commutative ring spectrum $\sE \in \SH(k)$ consists in a theory of Thom classes $\th_{\Sp}^\sE(V,\omega) \in \sE^{4r,2r}(\Th(V))$ for each rank $2r$ symplectc vector bundle $(V,\omega)\to X$, $X \in \Sm/k$, satisfying the axioms of Definition \ref{defn:thomclasstheory}, after replacing the trivial vector bundle $\mathcal{O}^r_X$ with the trivial symplectic bundle $(\mathcal{O}_X^{2r},\phi_{2r})$, morphisms and exact sequences of vector bundles with morphisms and exact sequences of symplectic vector bundles.
\end{defn}

It is immediate to note that any $\GL$-orientation induces an $\Sp$-orientation.

We now want to get a theory of symplectic Thom isomorphisms as in the case of $\GL$-oriented theories. We first need to define symplectic analogs of the fundamental groupoid $\sK(X)$ and the Grothendieck group $K_0(X)$.

\begin{defn} Let us fix a scheme $X\in \Sm/k$.\\[5pt]
1. Let $\sS{p}(X)$ denote the groupoid whose objects are the symplectic vector bundles $(V,\omega)$ on $X$, and whose morphisms are isomorphisms of vector bundles compatible with the given symplectic forms.\\[2pt]
2. Let  $\sS{p}(X)_{st}$ be the category where objects are pairs of symplectic vector bundles on $X$ $((V,\omega), (V', \omega'))$, and where a morphism $((V_1,\omega_1), (V_1', \omega_1'))\to
((V_2,\omega_2), (V_2', \omega_2'))$ is a triple $((V,\omega), f,f')$ with $(V,\omega)$ a
symplectic vector bundle on $X$, and
$$f:(V_1\oplus V,\omega_1\oplus \omega)\xrightarrow{\sim} (V_2,\omega_2), \; \; f':(V'_1\oplus V,\omega'_1\oplus \omega)\xrightarrow{\sim} (V'_2,\omega'_2)$$ are isomorphisms of symplectic vector bundles. If we have two morphisms
$$((V_1,\omega_1),(V_1',\omega_1'))\xrightarrow{((V_f,\omega_f),f,f')} ((V_2,\omega_2),(V_2',\omega_2')) \xrightarrow{((V_g,\omega_g),g,g')} ((V_3,\omega_3),(V_3',\omega_3')),$$
their composition is the triple $((V_f \oplus V_g,\omega_f \oplus \omega_g),g \circ (f\oplus \id_{V_g}), g' \circ (f' \oplus \id_{V_g}))$. \\ [2pt]
3. Let $\sK^\Sp(X)$ be the localization of $\sS{p}(X)_{st}$ with respect to the morphisms 
$$((V'',\omega''), \id_{V\oplus V''},\id_{V'\oplus V''}):((V,\omega), (V', \omega'))\to ((V\oplus V'',\omega\oplus \omega''), (V'\oplus V'', \omega'\oplus\omega'')).$$
4. We have the commutative monoid $(\sS{p}(X)/\text{iso},\oplus)$ of isomorphism classes in $\sS{p}(X)$, with operation induced by the direct sum of symplectic  bundles. Let $K_0^\Sp(X)$ denote its group completion.
\end{defn}

\begin{rmk} 1. $\sK^\Sp(X)$ is a symmetric monoidal groupoid: the inverse of $((V,\omega), f,f'):((V_1,\omega_1), (V_1', \omega_1'))\to
((V_2,\omega_2), (V_2', \omega_2'))$ is $((V,\omega), \id_{V_1 \oplus V},\id_{V_1'\oplus V})^{-1}\circ (0, f^{-1}, f^{\prime-1})$, and the symmetric monoidal product is induced by direct sum. \\[2pt]
2. Sending $((V_1,\omega_1), (V_1', \omega_1'))$ to the formal difference of classes $[(V_1,\omega_1)]- [(V_1', \omega_1')]\in K_0^\Sp(X)$ descends to an isomorphism of abelian groups $ \sK^\Sp(X)/\text{iso}\cong  K_0^\Sp(X)$. \\[2pt]
3. Sending $((V_1,\omega_1), (V_1', \omega_1'))$ to  the complex $V_1\oplus V_1'[1]\in \sK(X)$ extends to a functor of symmetric monoidal groupoids $\sK^\Sp(X)\to \sK(X)$, which induces the corresponding map of abelian groups $K_0^\Sp(X)\to K_0(X)$ by passing to isomorphism classes. In particular, sending $((V_1,\omega_1), (V_1', \omega_1'))$ to $\Sigma^{V_1-V_1'}1_X\in \SH(X)$ extends to a functor of groupoids $\Sigma^{(-)}:\sK^\Sp(X)\to \SH(X)_{\text{iso}}$.
\end{rmk}

\begin{defn}
Let $\sE\in \SH(k)$ be an $\Sp$-oriented ring spectrum. Given a rank $2r$ symplectic bundle, $(V,\omega)$ on $X\in \Sm/k$,  the Thom class $\th^\sE_\Sp(V,\omega)$ induces a morphism
\[
\th^{\Sp}_\sE(V,\omega):\Sigma^{V-\sO_X^{2r}}1_X\to p_X^*\sE
\]
in $\SH(X)$, where $p_X:X\to \Spec k$ is the structure map, which in turn induces the morphism
\[
\th^{\Sp,f}_{\sE}(V,\omega):\Sigma^{V-\sO_X^{2r}}p_X^*\sE\to p_X^*\sE
\]
in $\SH(X)$. If $\sE$ is highly structured, $\th^{\Sp}_\sE(V,\omega)$ also induces the map 
\[
\th^{\Sp, f}_{\sE\Mod}(V,\omega):\Sigma^{V-\sO_X^{2r}}p_X^*\sE\to p_X^*\sE 
\]
in $\Mod_{p^*\sE}$, by the free-forget adjunction.
\end{defn}

\begin{lemma}
\label{lemma:beforethom}
    The map $\th^{\Sp,f}_{\sE}(V,\omega)$ is an isomorphism in $\SH(X)$, and if $\sE$ is highly structured, the map $\th^{\Sp, f}_{\sE\Mod}(V,\omega)$ is an isomorphism in $\Mod_{p^*\sE}$.
\end{lemma}

\begin{proof}
    The statement for $\th^{\Sp, f}_{\sE\Mod}(V,\omega)$ follows immediately from that for $\th^{\Sp,f}_{\sE}(V,\omega)$, by applying the forgetful functor. We prove the statement for $\th^{\Sp,f}_{\sE}(V,\omega)$.

    Let us suppose that $(V,\omega)$ is a trivial symplectic vector bundle, which means $V=\mathcal{O}_X^{2r}$, $\omega=\phi_{2r}$. Then we have $\Sigma^V=\Sigma^{4r,2r}$, and it is easy to see that $\th_{\sE}^\Sp(V,\omega)$ can be rewritten as
    $$\th^{\Sp,f}_{\sE}(V,\omega)=\Sigma^{4r,2r}\mu_{p_X^*\sE} \circ \Sigma^{4r,2r}(\epsilon_{p_X^*\sE} \wedge 1) = \Sigma^{4r,2r}\id_{p_X^*\sE},$$
    and hence is an isomorphism.
    
Let then $(V,\omega)$ be a general rank $2r$ symplectic vector bundle over $X$. If $x\in X$ is a point, through a skew-symmetric version of the Gram-Schmidt process, we can find (not-uniquely) a Zariski open neighbourhood $U_x$ of $x$ and a decomposition 
$V \mid_{U_x}= \langle v_1, \ldots, v_{2r} \rangle$
such that $\omega\mid_{U_x}$ is represented by the standard symplectic matrix $\phi_{2r}$ with respect to the generating sections $v_1, \ldots v_{2r}$. We have then a Zariski open cover $\mathcal{U}=\{U_{\alpha}\}_{\alpha}$ trivializing $(V,\omega)$ as symplectic vector bundle. By quasi-compactness we can suppose $\mathcal{U}=\{U_{\alpha_1},\ldots,U_{\alpha_n}\}$ finite.

To simplify the notation, we write $\Psi$ for $\th^{\Sp,f}_{\sE}(V,\omega)$. Let $X_i=\cup_{j=1}^iU_{\alpha_i}$ with open immersion $j_i:X_i\to X$. We show by induction on $i$ that $j_i^*\Psi:j_i^*\Sigma^Vp_X^*\sE\to j_i^*\Sigma^{4r,2r}p_X^*\sE$ is an isomorphism, the case $i=1$ following by our construction of $U_{\alpha_1}$.

Let us take $i>1$ and write $X_i=X_{i-1}\cup U_{\alpha_i}$.  Let $j:X_{i-1}\cap U_{\alpha_i}\hookrightarrow X_i$, $j_1:X_{i-1}\hookrightarrow X_i$ and $j_2:U_{\alpha_i}\hookrightarrow X_i$ be the evident open immersions. Lemma \ref{lemma:Mayer-Vietoris} below shows that there exists a Mayer-Vietoris  distinguished triangle of endofunctors of $\SH(X_i)$:
\begin{equation}
    \label{eq:Mayer_Vietoris}
    j_\#j^* \to j_{1\#}j_1^* \oplus j_{2\#}j_2^* \to \id_{\SH(X_i)}\to  j_\#j^*[1].
\end{equation}
Applying \eqref{eq:Mayer_Vietoris} to the map of spectra $j_i^*\Psi$ gives the commutative diagram in $\SH(X_i)$
\[
\xymatrix{
j_\#j^*j_i^*\Sigma^Vp_X^*\sE \ar[r]\ar[d]^{j_\#j^*j_i^*\Psi}& \hbox{$\begin{matrix}j_{1\#}j_{i-1}^*\Psi\\\oplus\\ j_{2\#}j_2^*j_i^*\Psi\end{matrix}$}\ar[r]\ar[d]^{j_{1\#} j_{i-1}^*\Psi\oplus j_{2\#}j_2^*j_i^*\Psi} &j_i^* \Sigma^Vp_X^*\sE \ar[r]\ar[d]^{j_i^*\Psi} & j_\#j^*j_i^*\Sigma^Vp_X^*\sE[1] \ar[d]^{j_\#j^*j_i^*\Psi[1]}\\
j_\#j^*j_i^*\Sigma^{(4r,2r)}p_X^*\sE \ar[r]&\hbox{$\begin{matrix}j_{1\#}j_{i-1}^*\Sigma^{(4r,2r)}p_X^*\sE\\\oplus \\j_{2\#}j_2^*j_i^*\Sigma^{(4r,2r)}p_X^*\sE\end{matrix}$} \ar[r] & j_i^*\Sigma^{(4r,2r)}p_X^*\sE \ar[r] & j_\#j^*j_i^*\Sigma^{(4r,2r)}p_X^*\sE[1],
}
\]
with rows distinguished triangles. Using the induction hypothesis, we see that the vertical maps in the first, second, and fourth columns are isomorphisms, hence
$j_i^*\Psi$ is an isomorphism and the induction goes through.
\end{proof}

\begin{lemma}
\label{lemma:Mayer-Vietoris}
    Let $X=U_1 \cup U_2$ be the union of two Zariski open subsets $U_1, U_2$, and let $j_1:U_1 \to X$, $j_2: U_2 \to X$ and $j:U_1 \cap U_2 \to X$ the respective open inclusions. Then we have a Mayer-Vietoris distinguished triangle
    $$j_\#j^* \to j_{1\#}j_1^* \oplus j_{2\#}j_2^* \to \id_{\SH(X)} \to j_\#j^* \to j_{1\#}j_1^*[1]$$
    of endofunctors of $\SH(X)$.
\end{lemma}

\begin{proof}
    Let us consider the diagram
$$
\begin{tikzcd}
    U_1 \cap U_2 \arrow[r, "t_2"] \arrow[d, swap, "t_1"] \arrow[dr, "j"] & U_2\arrow[d, "j_2"] \\
    U_1 \arrow[r, "j_1"] & U_1 \cup U_2 =X.
\end{tikzcd}
$$
Let $i_1: X-U_1 \to X$ be the closed complement of $j_1$, and smilarly, let $i_2$ be the closed complement of $t_2$. From Proposition \ref{prop:localization}, the localization sequence for the pair $(j_1,i_1)$ is the distinguished triangle 
\begin{equation}
\label{eq:MV1}
    j_{1!}j_1^*\to \id_{\SH(X)} \to i_{1*}i_1^* \to j_{1!}j_1^*[1]
\end{equation}
of endofunctors of $\SH(X)$. Analogously, the localization sequence for the pair $(t_2,i_2)$ is the distinguished triangle
$$t_{2!}t_2^* \to \id_{\SH(X)} \to i_{2*}i_2^* \to t_{2!}t_2^*[1].$$
By applying $j_{2!}(-)j_2^*$ to the last triangle, and using $j_2\circ t_2=j$, we get the distinguished triangle
\begin{equation}
\label{eq:MV2}
     j_!j^* \to j_{2!}j_2^* \to j_{2!}i_{2*}i_2^*j_2^* \to j_!j^*[1].
\end{equation}
Moreover, we have a canonical distinguished triangle
\begin{equation}
\label{eq:MV3}
    j_{1!}j_1^* \to j_{1!}j_1^* \oplus j_{2!}j_2^* \to j_{2!}j_2^* \to j_{1!}j_1^*[1]. 
\end{equation}
Since $i_1 \simeq j_2 \circ i_2$ on $U_2-(U_1\cap U_2)$, we have that $j_{2!}i_{2*}i_2^*j_2^* \simeq i_{1*}i_1^*$. Thus, we can apply the octahedral axiom to the triangles \eqref{eq:MV1}, \eqref{eq:MV2} and \eqref{eq:MV3}, and we get the distinguished triangle
$$j_!j^* \to j_{1!}j_1^* \oplus j_{2!}j_2^* \to \id_{\SH(X)} \to j_!j^* \to j_{1!}j_1^*[1].$$
Finally, for any $f$ open immersion, we have $f_!\simeq f_\#$ by Remark \ref{rmk:etalemaps}. This concludes the proof.
\end{proof}

We have proved that $\th^{\Sp,f}_{\sE}(V,\omega)$ is an isomorphism. We may then define the map of spectra
$$\th^{\Sp,f}_{\sE}(-(V,\omega)):\Sigma^{\sO_X^{2r}-V}p_X^*\sE\to p_X^*\sE$$
by $\th^{\Sp,f}_{\sE}(-(V,\omega)):=\Sigma^{\sO_X^{2r}-V}(\th^{\Sp,f}_{\sE}(V,\omega)^{-1})$, and then define 
$$\th^{\Sp}_{\sE}(-(V,\omega)):\Sigma^{\sO_X^{2r}-V}1_X\to p_X^*\sE$$
by composing $\th^{\Sp,f}_{\sE}(-(V,\omega))$ with the map $\Sigma^{\sO_X^{2r}-V}1_X\to
\Sigma^{\sO_X^{2r}-V}p_X^*\sE$ induced by the unit $\epsilon_\sE$ of $\sE$. Finally, if $(V',\omega')$ is another symplectic vector bundle of rank $2r'$, we define
\[
\th^{\Sp}_{\sE}((V,\omega)-(V',\omega')):\Sigma^{V-V'- (\sO_X^{2r}-\sO_X^{2r'})}1_X\to p_X^*\sE
\]
as the composition
\begin{multline*}
\Sigma^{V-V'- (\sO_X^{2r}-\sO_X^{2r'})}1_X\simeq
\Sigma^{V- \sO_X^{2r}}(\Sigma^{-V'+ \sO_X^{2r'}}1_X)\\\xrightarrow{\Sigma^{V- \sO_X^{2r}}(\th^{\Sp}_{\sE}(-(V',\omega'))}
\Sigma^{V- \sO_X^{2r}}p_X^*\sE\xrightarrow{\th^{\Sp,f}_{\sE}(V,\omega)}
p_X^*\sE.
\end{multline*}
Again, we also have the $p_X^*\sE$-extension
\[
\th^{\Sp,f}_{\sE}((V,\omega)-(V',\omega')):\Sigma^{V-V'- (\sO_X^{2r}-\sO_X^{2r'})}p^*_X\sE\to p_X^*\sE,
\]
given by 
$\th^{\Sp,f}_{\sE}(V,\omega)\circ \Sigma^{V-V'-\sO_X^{2r}-\sO_X^{2r'}}(\th^{\Sp,f}_{\sE}(V',\omega')^{-1})$, and similarly, if $\sE$ is highly structured, the map $\th^{\Sp, f}_{\sE\Mod}((V,\omega)-(V',\omega'))$ of $p_X^*\sE$-modules.

The multiplicativity of the symplectic Thom isomorphisms, as stated in Lemma \ref{lem:SympThomMult} below, follows as for the $\GL$-oriented case from the multiplicativity of the symplectic Thom classes, following the symplectic version of  \eqref{enum:ThomClassAxioms}(3).

\begin{lemma}\label{lem:SympThomMult} Let $\sE\in \SH(k)$ be symplectically oriented, and take $X\in \Sm/k$. 
For $v,v'\in \sK^\Sp(X)$, $\th_\sE^\Sp(v+v')$ is the composition
\begin{multline}\label{multline:SympThomMult}
\Sigma^{v+v'-\sO_X^{\rnk(v+v')}}1_X\simeq \Sigma^{v-\sO_X^{\rnk(v)}}1_X\wedge_X\Sigma^{v'-\sO_X^{\rnk(v')}}
1_X\\\xrightarrow{\th_\sE^\Sp(v)\wedge\th^\Sp_\sE(v')}p^*\sE\wedge_Xp^*\sE\xrightarrow{\mu_{p^*\sE}} p^*\sE.\notag
\end{multline}
\end{lemma}

\begin{prop}\label{prop:K_0SpExtensionTh(-)} Let $\sE\in\SH(k)$ be a symplectically oriented ring spectrum and let $p_X:X\to \Spec k$ be in $\Sm/k$.\\[5pt]
1. Let $c_{p_X^*\sE}:\sK^\Sp(X)\to \SH(X)_{\id}$ be the constant functor with value $ p_X^*\sE$, with $\SH(X)_{\id}$ considered as a category with only identity morphisms. Then sending $(V,\omega)-(V',\omega')$ to $\th^{\Sp}_{\sE}((V,\omega)-(V',\omega'))$ extends to a natural transformation of functors
\[
(\th^{\Sp}_{\sE}(-):\Sigma^{(-)-\sO_X^{\rnk(-)}}1_X\to c_{p_X^*\sE}):\sK^\Sp(X)\to \SH(X). 
\]
2. The functor of groupoids $\Sigma^{(-)-\sO_X^{\rnk(-)}}p^*_X\sE:\sK^\Sp(X)\to (\SH(X))_\simeq$ descends (up to equivalence) to a functor $\Sigma^{(-)-\sO_X^{\rnk(-)}}p^*_X\sE:K_0^\Sp(X)\to \SH(X)$, where we consider 
$K_0^\Sp(X)$ as a category with only identity morphisms.\\[2pt]
3. For each $(v,\omega)\in K^\Sp_0(X)$, let $\th^{\Sp,f}_{\sE}(v, \omega):
\Sigma^{v-\sO_X^{\rnk(v)}}p^*_X\sE\to p_X^*\sE$ be the $\sE$-linear extension of  
$\th^{\Sp,f}_{\sE}(v, \omega)$. Then the maps $\th^{\Sp,f}_{\sE}(v, \omega)$ define a natural isomorphism of functors of groupoids
\[
\th^{\Sp,f}_{\sE}(-):\Sigma^{(-)-\sO_X^{\rnk(-)}}p^*_X\sE\to c_{p_X^*\sE}:K_0^\Sp(X)\to (\SH(X))_\simeq. 
\]
\end{prop}

The proof is similar to that of Proposition~\ref{prop:K_0ExtensionTh(-)} for the $\GL$-oriented case, but using Lemma~\ref{lem:SympThomMult} to show that the maps $\th^{\Sp}_{\sE}(-)$ are natural in $\sK^\Sp(X)$. We omit the details.
 
 \begin{rmk} For $\sE$ highly structured, we may refine Proposition~\ref{prop:K_0SpExtensionTh(-)} by replacing $\th^{\Sp,f}_{\sE}$ with $\th^{\Sp, f}_{\sE\Mod}$ and the target category $\SH(X)$ with $\Mod_\sE$.
\end{rmk}

\subsubsection*{Borel classes} The symplectic analogue of a theory of Chern classes is a theory of Borel classes. For $\sE$ $\Sp$-oriented, $(V,\omega) \to X$ rank $2r$ symplectic vector bundle, one has Borel classes $b_i(V,\omega) \in \sE^{4i,2i}(X)$, $i=1, \ldots, r$ satisfying the axioms of \cite[Definition 14.1]{panwal:grass}. 

Alternatively, one can define the first Borel class of a rank $2$ symplectic vector bundle $(V,\omega) \to X$ as $-z_0^*\th_\Sp^\sE(V,\omega)$, with $z_0:X_+\to \Th(V)$ the zero section. Next, to any symplectic vector bundle $(V,\omega) \to X$, one has the associated quaternionic projective bundle $\HP(V,\omega) \to X$ defined by $\HP(V,\omega):=\Gr(2,V)/\Gr\Sp(2,V,\omega)$, where $\Gr\Sp(2,V,\omega)$ is the closed subscheme of $\Gr(2,V)$ parametrizing rank $2$ vector subbundles of $V$ over which the restriction of $\omega$ is degenerate. The following is \cite[Theorem 8.2]{panwal:grass}.
\begin{thm}[Quaternionic projective bundle formula]
\label{thm:QuatProjBundle}
    Let $\sE \in \SH(k)$ be an $\Sp$-oriented commutative ring spectrum. Let $(V,\omega) \to X$ be a rank $2r$ symplectic vector bundle over $X \in \Sm/k$, and $\HP(V,\omega) \xrightarrow{q}X$ the associated quaternionic projective bundle. Let $(\mathcal{U}, \omega \mid_{\mathcal{U}})\to \HP(V,\omega)$ be the tautological rank $2$ symplectic vector bundle over $\HP(V,\omega)$, and $\xi \coloneqq b_1(\mathcal{U}, \omega \mid_{\mathcal{U}}) \in \sE^{4,2}(\HP(V,\omega))$ its first Borel class. Then $\sE^{*,*}(\HP(V,\omega))$ is a free $\sE^{*,*}(X)$-module via the pullback $q^*$, with basis $1, \xi, \ldots, \xi^r$. Also, for $1\le i \le r$, there exist unique elements $b_i(V,\omega) \in \sE^{4i,2i}(X)$ 
    such that
    $$\xi^r + \sum_{i=1}^r(-1)^ib_i(V,\omega) \cup \xi^{r-i}=0,$$
    and if $(V,\omega)$ is the trivial, then $b_i(V,\omega)=0$ for all $1\le i \le r$. 
\end{thm}
If $r=1$, one deduces from the discussion on symplectic Thom isomorphisms that the element $b_1(V,\omega)$ given by Theorem \ref{thm:QuatProjBundle} coincides with the previous description of the first Borel class of a rank $2$ symplectic bundle, and one can define the Borel classes as the elements $b_i(V,\omega)$ given by Theorem \ref{thm:QuatProjBundle}. This approach is equivalent to the axiomatic one by \cite[Theorem 14.4]{panwal:grass}.

\begin{rmk}
    The sign in the definition of the first Borel class is purely conventional, and it is chosen in order to satisfy the traditional relation $b_i(V)=(-1)^ic_{2i}$ for cohomology theories that are both $\Sp$-oriented and $\GL$-oriented.
\end{rmk}

\subsection{Symplectic cobordism}
A symplectic analogue of the algebraic cobordism spectrum $\MGL$ is the symplectic cobordism spectrum $\MSp$ defined in \cite[Section 6]{Panwal-cobordism}. The construction is similar to that of $\MGL$. If we equip $\A_k^{2n}$ with the standard symplectic form $\phi_{2n}$, we can define the \emph{quaternionic Grassmannian} $\HGr(r,n)$ as the open subscheme $j:\HGr(r,n) \xhookrightarrow{j} \Gr(2r,2n)$ parametrizing the $2r$-dimensional subspaces of $\A_k^{2n}$ on which the restriction of $\phi_{2n}$ is non-degenerate. We denote by $E_{2r,2n}^\Sp$ the restriction $j^*E_{2r,2n}$. In particular, $E_{2r,2n}^\Sp$ is a subbundle of the trivial bundle $\mathcal{O}_{\HGr(r,n)}^{2n}$. $\phi_{2n}$ defines a symplectic form on $\mathcal{O}_{\HGr(r,n)}^{2n}$, and its restriction to $E_{2r,2n}^\Sp$ is non-degenerate. We call $(E_{2r,2n}^\Sp,\phi_{2n} \mid_{E_{2r,2n}^\Sp})$ the \emph{tautological rank $2r$ symplectic subbundle}. We will occasionally omit the symplectic form in the notation. One defines $\BSp_{2r}\in \sH(k)$ as the colimit $\colim_N \HGr(r,rN)$ with respect to the natural inclusions, the vector bundle $E_{2r}^\Sp \to \BSp_{2r}$ as $E_{2r}^\Sp \coloneqq \colim_N E_{2r,2rN}^\Sp$, and finally $\MSp_{2r} \coloneqq \Th(E_{2r}^\Sp) \in \sH_{\bullet}(k).$ The canonical inclusion $i_r:\BSp_{2r} \to \BSp_{2r+2}$ gives a canonical isomorphism $i^*E_{2r+2}^\Sp \simeq \mathcal{O}_{\BSp_{2r}} \oplus E_{2r}^\Sp \oplus \mathcal{O}_{\BSp_{2r}}$. Analogously to the case of $\MGL$ we get a $T^{\wedge 2}$-spectrum $\MGL=(\MGL_0,\MGL_2,\ldots, \MGL_{2r},\ldots)$, and this defines an object in $\SH(k)$ because the homotopy category of $T^{\wedge 2}$-spectra is equivalent to $\SH(k)$ (\cite[Theorem 3.2]{Panwal-cobordism}). Of course, this is equivalent to defining $\MSp$ as $\colim_r \Sigma^{-4r,-2r}\MSp_{2r} \in \SH(k)$.

\begin{rmk}
It is shown in \cite[Section 6]{Panwal-cobordism} that $\MSp$ lifts to a commutative monoid in the category of symmetric $\P^2$-spectra. Thus, $\MSp$ is a highly structured commutative ring spectrum in $\SH(k)$
\end{rmk}

$\MSp$ comes with a theory of symplectic Thom classes, whose construction is completely analogous to that of $\MGL$, using $\BSp_{2r}$ as classifying space for rank $2r$ symplectic vector bundles, and we have the following universality theorem.

\begin{thm}[\cite{PanWal:MSpKtheory}, Theorem 4.5]
\label{thm:universalitySp}
       For any commutative ring spectrum $\sE$, the map of sets
    \begin{multline*}
        \{\text{morphisms of commutative ring spectra} \;  \MSp \to \sE \; \text{in} \; \; \SH(k)\} \\ \xrightarrow{\sim} \{\text{$\Sp$-orientations on the ring spectrum} \; \sE\}.
    \end{multline*}
    that takes a map $\varphi:\sE \to \MSp$ to the $\Sp$-orientation of $\sE$ given by $\th_\Sp^\sE(V,\omega)\coloneqq \Sigma^{4r,2r}\varphi \circ \th_\Sp^\MSp(V,\omega)$ is a bijection.
\end{thm}

\begin{rmk}
\label{rmk:MSp-MGLThomClasses}
    The inclusions $\HGr(r,n) \hookrightarrow \Gr(2r,2n)$ induce maps $\BSp_{2r} \to \BGL_{2r}$, and, since $E_{2r}^\Sp$ is the pullback of $E_{2r}$, we also get maps of Thom spaces $\MGL_{2r} \to \MSp_{2r}$ in $\sH_{\bullet}(k)$. These maps in turn define a natural map $\Phi:\MSp \to \MGL$ in $\SH(k)$, and one notes that the $\Sp$-orientation on $\MGL$ given by $\Phi$ is the $\Sp$-orientation induced by the universal $\GL$-orientation on $\MGL$.
\end{rmk}

\begin{prop}[\cite{Panwal-cobordism}, Theorems 9.1, 9.3]
\label{prop:CohBSpMSp} Let $E^\Sp_{2r}\to \BSp_{2r}$ be the tautological rank $2r$ symplectic bundle, with Thom space $\MSp_{2r}$, and let  $\sE$ be a $\Sp$-oriented commutative ring spectrum in $\SH(k)$.\\[5pt]
1. $\sE^{*,*}(\BSp_{2r})$ is the ring of homogeneous elements in the power series ring $\sE^{*,*}(\Spec k)[[b_1,\ldots, b_r]]_h$, where $b_i:=b_i(E^\Sp_{2r})\in \sE^{4i,2i}(\BSp_{2r})$ is the $i$-th Borel class of $E^\Sp_{2r}$.\\[2pt]

2. Let $z_{2r}:\BSp_{2r}\to \MSp_{2r}$ be the map induced by the zero-section of $E^\Sp_{2r}$. Then $z_{2r}^*:\sE^{*,*}(\MSp_{2r})\to \sE^{*,*}(\BSp_{2r})$ is injective, with image the ideal generated by $b_r$.
\end{prop}

From this, similarly to how we have shown in Proposition \ref{prop:MotCohOfMGL} for the motivic cohomology of $\MGL$, they obtain a computation of the $\sE^{*,*}$-cohomology of $\MSp$, as follows. 

\begin{thm}[\cite{Panwal-cobordism}, Theorem 13.1]
\label{thm:cohomologyofmsp}
    Let $\sE \in \SH(k)$ any motivic commutative ring spectrum with an $\Sp$-orientation. Then we have an isomorphism of bigraded rings:
    $$\sE^{*,*}(\MSp)\cong \sE^{*,*}(\Spec k)[[b_1,b_2,\ldots]]_h$$
    with the variables $b_i$ in bidegree $(4i,2i)$.
\end{thm}

\begin{rmk}
\label{rmk:homogeneous}
If there is a $k_0$ such that $\sE^{a,b}(\Spec k)=0$ for $a<-4k_0$, $b<-2k_0$, then $\sE^{*,*}(\Spec k)[[b_1,\ldots, b_r]]_h$ is just the polynomial ring $\sE^{*,*}(\Spec k)[b_1,\ldots, b_r]$. For instance, this is the case for $\sE=H\Z$ or $\sE=H\Z/\ell$.
\end{rmk}

We do not have as nice description of the coefficient ring of $\MSp$ as we have for $\MGL$ through Theorem \ref{thm:comparisonLazard} and Propositions \ref{prop:MGL-classes}, \ref{prop:CriterionMGL}. The main goal of this paper is a partial computation of the graded ring $\MSp^*$.

\section{A twisted degree map}
\label{section:TwistedDegree}

For any $E$ vector bundle on $X \in \Sm/S$ and any $L$ line bundle on $X$, Ananyevskiy (\cite[Lemma 4.1]{Ana:Slor}) gives an isomorphism in $\sH_{\bullet}(S)$.
$$\rho: \Th(E\oplus L)\xrightarrow{\sim} \Th(E \oplus L^{\vee})$$
Through $\Sigma^{\infty}_{\P^1}$, the isomorphism passes to the stable setting. Then, for line bundles $L_1, \ldots, L_n$ over $X$, we denote by $\Anan_{L_1, \ldots, L_n}$ the isomorphism
\begin{equation}
    \label{eqn:AnanIso}
    \Anan_{L_1,\ldots, L_n}: \Sigma^{\oplus_{i=1}^n L_i}1_X \xrightarrow{\sim} \Sigma^{\oplus_{i=1}^n L_i^\vee}1_X
\end{equation}
in $\SH(X)$ obtained through Ananyevskiy's construction, with $S=X$. In the same way, for all $e \in K_0(X)$ and line bundles $L_i$ as before, we obtain an isomorphism
\begin{multline}
\label{eq:Anaconstruction}
    \Anan_{e;-L_1, \ldots, -L_n} \coloneqq \Sigma^{e-\oplus_{i=1}^nL_i^\vee-
    \oplus_{i=1}^nL_i} \Anan_{L_1^\vee,\ldots, L_n^\vee}: \Sigma^{e-\sum_{i=1}^n[L_i]}1_X\xrightarrow{\sim} \\ \Sigma^{e-\sum_{i=1}^n[L_i^\vee]}1_X.
\end{multline}

\begin{lemma}\label{lem:Anan} Let us suppose $X \in \Sm/k$ irreducible, and use the shorthand $\Anan \coloneqq \Anan_{e;-L_1,\ldots,-L_n}$ for the isomorphism \eqref{eq:Anaconstruction}. Let also $d\coloneqq \rnk(e)-n$. For $p:X\to \Spec k$ in $\Sm/k$, the diagram in $\Mod_{p^*H\Z}$
$$
\begin{tikzcd}
\Sigma^{e- \sum_{i=1}^n[L_i] - d[\mathcal{O}_X]}1_X\wedge p^*H\Z \arrow[rr, "\Sigma^{-2d,-d} \Anan\wedge\id_{p^*H\Z}", "\sim"'] \arrow[dr, swap, "\th_{H\Z\Mod}^f (e- \sum_{i=1}^nL_i)", "\sim"'] && \Sigma^{e-\sum_{i=1}^n[L_i^\vee]-d[\mathcal{O}_X]}1_X\wedge p^*H\Z \arrow[dl, "(-1)^n\th^f_{H\Z\Mod} (e-\sum_{i=1}^nL_i^\vee)", "\sim"'] \\
& p^*H\Z &
\end{tikzcd}
$$
commutes.
\end{lemma}

\begin{proof}
    By using a Jouanolou cover of $X$, we may assume that $X$ is affine. In that case, there exists a vector bundle $E$ over $X$ such that $-e=[E]-s[\sO_X]$, and let $r \coloneqq s-d$. In this way, we are reduced to prove the commutativity of the diagram
    \begin{equation}
    \label{eq:anatriangle}
         \begin{tikzcd}
        \Sigma^{2r,r}\Sigma^{-[E +\oplus_i^nL_i]}1_X \wedge p^*H\Z \arrow[rrrr, "\Sigma^{2r,r}\Anan_{E;L_1,\ldots,L_n} \wedge \id_{p^*H\Z}", "\sim"'] \arrow[drr, swap, "\th^f_{H\Z \Mod} (-E-\oplus_i^nL_i)", "\sim"'] &&&&  \Sigma^{2r,r}\Sigma^{-[E +\oplus_i^nL_i^{\vee}]}1_X \wedge p^*H\Z \arrow[dll, "(-1)^n\th^f_{H\Z \Mod} (-E-\oplus_i^nL_i^\vee)", "\sim"'] \\
        &&p^*H\Z&&
    \end{tikzcd}
    \end{equation}
    in $\Mod_{p^*H\Z}$. Let $\rho:\Th(E +\oplus_{i=1}^nL_i)\xrightarrow{\sim}\Th(E + \oplus_{i=1}^nL_i^\vee)$ the isomorphism in $\sH_{\bullet}(k)$ defined by Ananyevskiy.
    
    The commutativity of \eqref{eq:anatriangle} is equivalent to
    $$(\Sigma^{2r,r}\Anan_{E;L_1,\ldots,L_n} \wedge \id_{p^*H\Z})^*\th_{H\Z}^f (E\oplus_i^nL_i^\vee) =(-1)^n \th_{H\Z}^f (E+\oplus_i^nL_i),$$
    and this is in turn equivalent to
    $$\rho^*\th^{H\Z} (E+\oplus_i^nL_i^\vee) = (-1)^n \th^{H\Z} (E+\oplus_i^nL_i).$$
    Lemma \ref{lemma:anatwists} below shows in detail the computation for $n=1$. The result immediately follows by induction on $n$.
\end{proof}

\begin{lemma}
\label{lemma:anatwists}
    Let $X \in \Sm/k$ irreducible. Let $E \to X$ a vector bundle of finite rank over $X$ and $L \to X$ a line bundle over $X$, with dual bundle $L^{\vee}\to X$. Let $\rho: \Th(E\oplus L) \xrightarrow{\sim} \Th(E \oplus L^\vee)$ the isomorphism in $\SH(k)$ defined as in \cite[Lemma 4.1]{Ana:Slor}. Then $\rho^*\th^{H\mathbb{Z}}(E + L^{\vee})=-\th^{H\mathbb{Z}}(E+L).$
\end{lemma}

\begin{proof}
    Since we are working in the stable setting, we simplify the notation by writing $\rho$ for $\Sigma^{\infty}_{\P^1}\rho:p_\#\Sigma^{E +L}1_X \xrightarrow{\sim} p_\#\Sigma^{E+L^\vee}1_X$, with $p\coloneqq p_X:X \to \Spec k$ the structure map. Let us briefly review the construction of $\rho$ from \cite{Ana:Slor}.

    Let $\pi_L:L\to X$ and $\pi_{L^\vee}:L^\vee \to X$ the structure morphisms of the two vector bundles. Let us denote by $L^0$ and $(L^\vee)^0$ the complements of the zero sections, and by $\pi_{L^0}$, $\pi_{(L^\vee)^0}$ the restrictions of the relative morphisms. Let us note that $\A^1$-invariance (Proposition \ref{prop:hom-invariance}) implies $\pi_{L\#}1_L \simeq 1_X \simeq \pi_{L^\vee \#}1_{L^\vee}$. The localization sequence associated to the inclusion $L^0 \hookrightarrow L$ as in \eqref{eq:thom-loc2} gives a distinguished triangle 
    \begin{equation}
    \label{eq:locseq1}
        \pi_{L^0\#}1_{L^0} \to 1_X \to \Sigma^L1_X \to \pi_{L^0\#}1_{L^0}[1], 
    \end{equation}
    and analogously, we have the distinguished triangle
    \begin{equation}
    \label{eq:locseq2}
        \pi_{(L^\vee)^0\#}1_{(L^\vee)^0} \to 1_X \to \Sigma^{L^\vee}1_X \to \pi_{(L^\vee)^0\#}1_{(L^\vee)^0}[1].
    \end{equation}
    There is a unique isomorphism of $X$-schemes $f:L^0\xrightarrow{\sim}(L^{\vee})^0$ given by taking, fiberwise, the vector $v$ to the functional $f_v$ such that $f_v(v)=1$. By taking $\Sigma_{\P^1}^\infty f$ in $\SH(X)$, this gives an isomorphism between the left hand terms of \eqref{eq:locseq1} and \eqref{eq:locseq2}, that, together with $\id_X$, induces the isomorphism $\Anan_L:\Sigma^L1_X \xrightarrow{\sim} \Sigma^{L^\vee}1_X$ on the cofibers. Applying $p_\# \circ \Sigma^E$ to $\Anan_L$ gives the isomorphism $\rho: p_\#\Sigma^{E+L}1_X \xrightarrow{\sim} p_\#\Sigma^{E +L^\vee}1_X$.

    From this construction, we see that $\rho=\id_{p_\# \Sigma^E1_X}\wedge \rho_L$, with $\rho_L \coloneqq p_\# \Anan_L$, and since Thom classes are additive, we have $\rho^*\th^{H\mathbb{Z}}(E + L^{\vee})=\th^{H\Z}(E) \cup \rho_L^*\th^{H\Z}(L^\vee)$. This reduces to the case $E=0$ and $\rho=\rho_L$.

    Next, we reduce to the case $X=\Spec k(X)$. The Thom isomorphism in $H\mathbb{Z}$-cohomology makes $H\mathbb{Z}^{**}(p_{\#}\Sigma^L1_X)$ and $H\mathbb{Z}^{**}(p_{\#}\Sigma^{L^{\vee}}1_X)$ two free $H\mathbb{Z}^{**}(X)$-modules of rank $1$, generated respectively by $\th^{H\mathbb{Z}}(L)$ and $\th^{H\mathbb{Z}}(L^{\vee})$, both in bidegree $(2,1)$. $\rho_L^*$, as an isomorphism of $H\mathbb{Z}^{**}(X)$-modules, sends generators to generators, then $\rho_L^*\th^{H\mathbb{Z}}(L^{\vee})=\alpha \cdot \th^{H\mathbb{Z}}(L)$, with $\alpha \in H\mathbb{Z}^{0,0}(X)=\mathbb{Z}.$ Since $X$ is irreducible, there is a unique generic point $i:\Spec k(X) \to X$. Let then
    $$\rho_{i^*L}: p_{k(X)\#}\Sigma^{i^*L}1_{k(X)} \xrightarrow{\sim} p_{k(X)\#}\Sigma^{i^*L^{\vee}}1_{k(X)}$$
    be the isomorphism obtained in the same way as $\rho$, with $p_{k(X)}\coloneqq p \circ i =\Spec(k \hookrightarrow k(X))$. We have the commutative diagram
\[
\xymatrixcolsep{70pt}
 \xymatrix{
 H\Z^{0,0}(X)\ar@{=}[d]\ar[r]^{i^*}_\sim&H\Z^{0,0}(k(X))\ar@{=}[d]\\
  [1_X, p^*H\Z]_{\SH(X)}\ar[r]^{i^*}\ar[d]_\cong^{\cup \; \th^{H\mathbb{Z}}(L)}&[1_{k(X)}, p_{k(X)}^*H\Z]_{\SH(k(X))}\ar[d]_\cong^{\cup \; \th^{p_{k(X)}^*H\mathbb{Z}}(i^*L)}\\
 [\Sigma^L1_X, \Sigma^{2,1}p^*H\Z]_{\SH(X)}\ar@{=}[d]\ar[r]^{\Th(i)^*}&[\Sigma^{i^*L}1_{k(X)}, \Sigma^{2,1}p_{k(X)}^*H\Z]_{\SH(k(X))}\ar@{=}[d]\\
 H\Z^{2,1}(\Th(L))\ar[r]^{\Th(i)^*}&H\Z^{2,1}(\Th(i^*L)) 
 }
 \]
 and a similar diagram for $L^\vee$. Let $\alpha_{k(X)}\in \Z=H\Z^{0,0}(k(X))$ be the element such that $\rho^*_{i^*L}(\th^{p_{k(X)}^*H\mathbb{Z}}(i^*L^{\vee}))=\alpha_{k(X)} \cdot \th^{p_{k(X)}^*H\mathbb{Z}}(i^*L)$. The maps $\rho^*_L$ and $\rho^*_{i^*L}$ give a map to the above diagram from the corresponding one for $L^\vee$, so we have $i^*\alpha=\alpha_{k(X)}$. But since $i^*:\Z=H\Z^{0,0}(X)\to H\Z^{0,0}(k(X))=\Z$ is the identity map on $\Z$, we have $\alpha_{k(X)}=\alpha$. This reduces to the case $X =\Spec k(X)$. 
 
 Changing notation, we may simply assume that $X=\Spec k$. $L$ is then a one dimensional $k$-vector space with zero section being the origin. The total space of $L$ is (non canonically) isomorphic to $\mathbb{A}^1$. Let us choose any isomorphism $L\xrightarrow{\sim} \mathbb{A}^1$. This also induces an isomorphism $L^{\vee}\xrightarrow{\sim}\mathbb{A}^1$ through $f\mapsto f(1)$. Through these isomorphisms, we identify $L^0$ and $(L^{\vee})^0$ with $\mathbb{G}_m$. The isomorphism of schemes $L^0 \xrightarrow{\sim}(L^{\vee})^0$ given by $v \mapsto f_v$, induces the automorphism $t \mapsto t^{-1}$ of $\mathbb{G}_m$. The isomorphisms $L\xrightarrow{\sim}\mathbb{A}^1\xleftarrow{\sim}L^{\vee}$ induce
    $$\Sigma^L1_X \xrightarrow{\sim}S^1\wedge \mathbb{G}_m \xleftarrow{\sim} \Sigma^{L^{\vee}}1_X,$$
    and through these, the isomorphism $\rho_L$ can be read as the automorphism $\id_{S^1}\wedge(t\mapsto t^{-1})$ of $S^1\wedge\G_m$, where $t$ is the canonical coordinate on $\G_m$. We have
    $$H\Z^{2,1}(\P^1) \cong H\mathbb{Z}^{2,1}(S^1\wedge\G_m)\cong  H\mathbb{Z}^{1,1} (\G_m) \xleftarrow{\sim} \mathbb{Z},$$
    with the last isomorphism sending $n\in\Z$ to $t^n\in H\mathbb{Z}^{1,1}(\G_m))$. Thus, $\rho_L$ induces the isomorphism $n\mapsto -n$ on $\Z\cong H\mathbb{Z}^{2,1}(\P^1)$.

    $\th^{H\mathbb{Z}(L)}$ and $\th^{H\mathbb{Z}}(L^{\vee})$ belong to $H\mathbb{Z}^{2,1}(\mathbb{P}^1)\cong \mathbb{Z}$, and they are generators. Since $L\simeq L^\vee\simeq \A^1_F$ as line bundles over $\Spec F$, we have $\th^{H\mathbb{Z}}(L)=\th^{H\mathbb{Z}}(L^{\vee})$, and since $\rho_L^*(n)=-n$, we conclude $\rho_L^*(\th^{H\mathbb{Z}}(L^{\vee}))= -\th^{H\mathbb{Z}}(L)$.
\end{proof}

Let us now consider $X \in \Sm/k$ and $e=\sum_{i=1}^n[L_i]-[T_X]+d_X[\sO_X] \in K_0(X)$, with $d_X \coloneqq \dim_k(X)$. Let us suppose $v \in K_0(X)$ being an element satisfying the identity
\begin{equation}
\label{eq:comparisontangentbundle}
    v=e - \sum_{i=1}^n[L_i^\vee] - d_X[\sO_X].
\end{equation}
In particular, $\rnk(v)=-\rnk(T_X)=-d_X$. In this case, \eqref{eq:Anaconstruction} gives the isomorphism
\begin{equation}
\label{eq:anatwist}
    \Anan_{-L_1, \ldots, -L_n}(v): \Sigma^{-T_X+\rnk(T_X)[\sO_X]}1_X \xrightarrow{\sim} \Sigma^{v-\rnk(v)[\sO_X]}1_X,
\end{equation}
where $\Anan_{-L_1, \ldots, -L_n}(v):=\Anan_{e;-L_1, \ldots, -L_n}$.

\begin{defn}
\label{defn:AnaDegree}
With $p:X\to \Spec k$ smooth and proper of dimension $d_X$, and $v \in K_0(X)$ as in \eqref{eq:comparisontangentbundle}, we have the \emph{twisted degree}
\[
\deg^\Anan_k:[1_X, \Sigma^{2d_X, d_X}p^*H\Z]_{\SH(X)}\to 
\Z
\]
defined as the composition \eqref{eq:anandeg} below, where we write $\Anan$ for the isomorphism \eqref{eq:anatwist}:
\begin{multline}
\label{eq:anandeg}
[1_X, \Sigma^{2d_X, d_X}p^*H\Z]_{\SH(X)} \xrightarrow{\sim}
[p^*H\Z, \Sigma^{2d_X, d_X}p^*H\Z]_{\Mod_{p^*H\Z}} \\
\xrightarrow[\sim]{\th_{H\Z\Mod}^f(v)^*} [\Sigma^{v-\rnk(v)[\sO_X]}1_X\wedge p^*H\Z, \Sigma^{2d_X, d_X}p^*H\Z]_{\Mod_{p^*H\Z}} \\
\xrightarrow[\sim]{(\Anan \wedge\id)^*}[\Sigma^{-T_X+\rnk(T_X)[\sO_X]}1_X\wedge  p^*H\Z, \Sigma^{2d_X, d_X}p^*H\Z]_{\Mod_{p^*H\Z}}\\
=[\Sigma^{-T_X}1_X\wedge  p^*H\Z, p^*H\Z]_{\Mod_{p^*H\Z}}
=[\Sigma^{-T_X}1_X, p^*H\Z]_{\SH(X)}
=[p_\#(\Sigma^{-T_X}1_X), H\Z]_{\SH(k)}\\
\xrightarrow{(p^\vee)^*}[1_k,H\Z]_{\SH(k)}
=H\Z^{0,0}(\Spec k)=\Z.
\end{multline}
\end{defn}

\begin{prop}
\label{prop:comparisondegrees}
Suppose we have $X\in \Sm/k$ irreducible, $v\in K_0(X)$ and line bundles $L_1,\ldots, L_r$ such that Ananyevskiy's isomorphism $\Anan_{L_1,\ldots, L_r}$ is defined. Suppose that $X$ is smooth and proper of dimension $d_X$ over $k$. Then, for $\alpha\in H\Z^{2d_X, d_X}(X)$ we have $\deg_k(\alpha)=(-1)^r\deg^\Anan_k(\alpha).$
\end{prop}

\begin{proof}
By rewriting the construction of the proper pushforward, we can see $\deg_k:H\Z^{2d_X, d_X}(X)\to \Z$ as the composition
\begin{multline}
\label{eq:rewritinf_deg}
H\Z^{2d_X, d_X}(X) =[1_X, \Sigma^{2d_X, d_X}p^*H\Z]_{\SH(X)} \xrightarrow{\sim}
[p^*H\Z, \Sigma^{2d_X, d_X}p^*H\Z]_{\Mod_{p^*H\Z}}\\
\xrightarrow{\th_{H\Z\Mod}^f(-T_X)^*}[\Sigma^{-T_X+d_X[\sO_X]}1_X\wedge p^*H\Z, \Sigma^{2d_X, d_X}p^*H\Z]_{\Mod_{p^*H\Z}}\\
=[\Sigma^{-T_X}1_X\wedge  p^*H\Z, p^*H\Z]_{\Mod_{p^*H\Z}}
=[\Sigma^{-T_X}1_X, p^*H\Z]_{\SH(X)}\\
=[p_\#(\Sigma^{-T_X}1_X), H\Z]_{\SH(k)}
\xrightarrow{(p^\vee)^*}[1_k,H\Z]_{\SH(k)}=H\Z^{0,0}(\Spec k)=\Z.
\end{multline}
Comparing \eqref{eq:rewritinf_deg} with \eqref{eq:anandeg} and applying Lemma~\ref{lem:Anan} concludes the proof.
\end{proof}

We want to apply this comparison to a certain kind of motivic cohomology classes. We need a preliminary definition.

\begin{defn}
\label{defn:HZclasses}
Let $\sE\in \SH(k)$ be an oriented ring spectrum, $c\in H\Z^{a,b}(\sE)=[\sE, \Sigma^{a,b}H\Z]_{\SH(k)}$, and $v\in K_0(X)$ for $X\in \Sm/k$. Let $p:X \to \Spec k$ be the structure map. Applying $p^*$ to $c$ gives the map $p^*(c):p^*\sE\to \Sigma^{a,b}p^*H\Z$. The Thom classes for $\sE$ give us the map $\th_\sE(v):\Sigma^{v-\rnk(v)[\sO_X]}1_X\to p^*\sE$.  
We have the map $\tilde{c}^\sE(v)'$ defined as the composition
\[
\Sigma^{v-\rnk(v)[\sO_X]}1_X\xrightarrow{\th_\sE(v)} p^*\sE\xrightarrow{p^*(c)}\Sigma^{a,b}p^*H\Z
\]
in $\SH(X)$, and by the adjunction $p_\# \dashv p^*$, the map
\[
\tilde{c}^\sE(v):p_\#\Sigma^{v-\rnk(v)[\sO_X]}1_X\to p_{\#}p^*\Sigma^{a,b}H\Z \to \Sigma^{a,b}H\Z,
\]
which can be seen as an element of $H\Z^{a,b}(p_\#\Sigma^{v-\rnk(v)[\sO_X]}1_X)$. Via the isomorphism $\Sigma^{a,b}\th_{H\Z}^f(-v)$, $\tilde{c}^\sE(v)$ gives an element $c^\sE(v)\in H\Z^{a,b}(X).$
\end{defn}

\begin{corollary}
\label{cor:TwistClassComp}
Let $Y\in \Sm/k$ be proper of dimension $d$ over $k$ and irreducible, and let $(v,\vartheta)$ be a stable twist of $-T_Y$. Let us suppose that, for a certain $e \in K_0(Y)$ and $L_1, \ldots, L_n$ line bundles on $Y$, we have $v=e-\sum_{i=1}^s[L_i^\vee]$, $-[T_Y]=e-\sum_{i=1}^n[L_i]$ in $K_0(Y)$, and $\vartheta$ is the isomorphism
\[
\Sigma^{-2d,-d}\Anan_{e, -L_1,\ldots, -L_n}(v):\Sigma^{-T_Y}1_Y\xrightarrow{\sim} \Sigma^v1_Y.
\]
If $\sE \in \SH(k)$ is an oriented commutative ring spectrum, and $c$ a class in $H\Z^{2d,d}(\sE)$,then 
\[
\deg_k c^\sE(v)=(-1)^n c\circ [Y,v,\vartheta]_\sE\in H\Z^{0,0}(\Spec k)=\Z.
\]
\end{corollary}

\begin{proof}
    Let $p:=p_Y:Y \to \Spec k$ be the structure map. We write $\Anan$ for $\Anan_{e, -L_1,\ldots, -L_n}(v)$.
    
    By looking at the definition of $[Y,v,\vartheta]_\sE$ (Definition \ref{defn:twistclass}), we see that $c\circ [Y,v,\vartheta]$ is the composition
    $$1_k\xrightarrow{p^\vee}p_{\#}\Sigma^{-T_Y}1_Y\xrightarrow[\sim]{p_\#\Sigma^{-2d,-d}\Anan} p_{\#}\Sigma^{v}1_Y\xrightarrow{\Sigma^{-2d,-d}\th_\sE(v)'} \Sigma^{-2d,-d}\sE\xrightarrow{\Sigma^{-2d-d}c}H\Z,$$
where $\th_\sE(v)'$ is the map corresponding to $\th_\sE(v)$ by adjunction $p_\# \dashv p^*$.

On the other hand, the composition $p_{\#}\Sigma^{v}1_Y\xrightarrow{\Sigma^{-2d,-d}\th_\sE(v)'} \Sigma^{-2d,-d}\sE\xrightarrow{\Sigma^{-2d-d}c}H\Z$ is the class $\Sigma^{-2d,-d}\tilde{c}^\sE(v)$ (see Definition \ref{defn:HZclasses}), then
$$c\circ [Y,v,\vartheta]=(p^\vee)^*(\Sigma^{-2d,-d}(p_\#\Anan)^*(\Sigma^{-2d,-d}\tilde{c}^\sE(v))).$$

We also see from Definition \ref{defn:HZclasses} that the isomorphism $H\Z^{2d,2}(p_\#\Sigma^{v+d[\sO_Y]}1_Y)\simeq H\Z^{2d,d}(Y)$ giving $c^\sE(v)$ is induced by the isomorphism $\th_{H\Z}^f(v)$, or equivalently, by $\th_{H\Z}^f(v)^*$.

Lastly, the isomorphism $p_\#\Anan$ in $\SH(k)$ corresponds to the isomorphism 
$$\Anan \wedge \id_{p^*H\Z}:\Sigma^{-T_Y+d[\sO_Y]}1_Y \wedge p_*H\Z \xrightarrow{\sim} \Sigma^{v +d[\sO_Y]}1_Y \wedge p^*H\Z$$
in $\Mod_{p^*H\Z}$ through the free-forgetful adjunction. 

Thus, by looking at the composition \eqref{eq:anandeg}, we see that $c\circ [Y,v,\vartheta]$ is the class $\deg_k^\Anan(c^\sE(v))$ in $H\Z^{0,0}(\Spec k)$. Applying Proposition \ref{prop:comparisondegrees} concludes the proof. 
\end{proof}

\section{Cellular structure of $\MSp$}
\label{section:Cellularity}

\begin{defn}
\label{defn:motivicCellular}
    We denote by $\langle 1_k \rangle$ the full localizing subcategory of $\SH(k)$ generated by the motivic spheres, or in other words, the smallest full subcategory containing $1_k$ and and closed with respect to coproducts $\oplus$, homotopy colimits, and suspensions $\Sigma^{a,b}$ for $a,b \in \Z$. We will say that a motivic spectrum $\E \in \SH(k)$ is \emph{cellular} if it belongs to the subcategory $\langle 1_k \rangle$.
\end{defn}

\begin{prop}
\label{prop:s.e.s.-cellularity}
    Let $\E' \to \E \to \E''$ be a cofiber sequence in $\SH(k)$ such that two of $\E'$, $\E$ and $\E''$ are cellular. Then so is the third.
\end{prop}

\begin{proof}
If $\E'$ and $\E$ are cellular, so is $\E''$, since $\langle 1_k \rangle$ is closed with respect to homotopy colimits, and in particular homotopy cofibers. For the other two cases, we note that $\E'$ is the homotopy cofiber of $\Sigma^{-1,0}\E \to \Sigma^{-1,0}\E''$, and $\E$ is the homotopy cofiber of $\Sigma^{-1,0}\E'' \to \E'$. In both cases, the result follows because $\langle 1_k \rangle$ is closed under $\Sigma^{-1,0}$-suspensions and homotopy cofibers.
\end{proof}

\begin{exmp}
\label{exmp:MGL_cellular}
The algebraic cobordism spectrum $\MGL$ is cellular (see \cite[Theorem 6.4]{dugisa:cellstructures} or the proof of \cite[Theorem 6.11]{lev:ellcoh}). In this case, the cellularity is induced by the Schubert cell decomposition of Grassmannians. 
\end{exmp}

To prove the cellularity of $\MSp$, we need to go through the geometry of quaternionic Grassmannians and their vector bundles. We noted that almost the same procedure was used in \cite[Proposition 3.1]{Rond:Cellularity}.

\subsection{Vector bundles over quaternionic Grassmannians}

 Let $V$ be a $(2n)$-dimensional vector space over $k$, and $\phi$ the symplectic form on $V$ represented by the standard symplectic matrix $=\phi_{2n}$ (see example \ref{exmp:sympl_bundles} (1)) with respect to a fixed basis $\langle x_1, \ldots, x_n,y_n, \ldots, y_1 \rangle$.

Over $\Spec k$, $\HGr(r,n)$ can be seen as the classifying space of the $(2r)$-dimensional symplectic vector subspaces of the symplectic vector space $(V,\phi)$. 

We denote by $E_i$ the $i$-dimensional subspace $\langle x_1, \ldots , x_i \rangle$ of $V$. Following \cite[Theorem 4.1]{panwal:grass}, we define $N^+\coloneqq \HGr(r,n)\cap \Gr(2r,E_1^{\perp})$. $N^+$ is a closed subscheme of $\HGr(r,n)$, and it is a rank $2r$ subbundle of the normal bundle $N_i$ of the obvious inclusion $i:\HGr(r,n-1) \to \HGr(r,n)$. More specifically, \cite[Theorem 4.1 (b)]{panwal:grass} gives $N_i=N^+\oplus N^-$, where $N^-=\HGr(r,n)\cap \Gr(2r,y_1^{\perp})$. Moreover, by \cite[Theorem 4.1 (c)]{panwal:grass}, both $N^+$ and $N^-$ are isomorphic to the tautological rank $2r$ symplectic bundle over $\HGr(r,n-1)$. In particular
$\dim(N^+)=\dim(\HGr(r,n-1))+2r=4nr-2r-4r^2.$ Finally, we denote by $Y$ the open complement of the closed subscheme $N^+$ in $\HGr(r,n)$. Before proving our main statement on the Thom spaces of vector bundles over $\HGr(r,n)$, we need to study the geometry of the open stratum $Y$, following the ideas of \cite{panwal:grass}.

$Y$ can be seen as the classifying space of $2r$-dimensional subspaces of $V$, symplectic with respect to the form $\phi$, that are not subspaces of $\langle x_1, \ldots, x_n,y_n, \ldots y_2 \rangle$. Let us consider the restriction of the tautological rank $2r$ bundle $\mathcal{U} \to Y$, with $\mathcal{U}\coloneqq E^\Sp_{2r,2n} \mid_Y$. Note that a point in $\mathcal{U}$ is a pair $(y,v)$ where $y$ is a point of $Y$, hence representing a vector subspace $[y]$, and $v$ a vector in $[y]$. Let $\lambda:\mathcal{U} \to \mathcal{O}_S$ be the map that takes $(y,(v_1,\ldots, v_{2n}))\in \mathcal{U}$ to $v_{2n}$, and let $Y_1 \coloneqq \lambda^{-1}(1)$. Let $g_1:Y_1 \to Y$ be the induced map.

Let now
$$Y_2 \coloneqq \{(y,e,f) \mid y \in Y, \;e \in [y],\; f \in [y], \; e_{2n}=0, \; f_{2n}=1, \; \phi(e,f)=1 \},$$
and let $g_2:Y_2 \to Y_1$ be the map defined by $(y,e,f) \mapsto (y,f)$.

\begin{lemma}
    The maps $g_1$ and $g_2$ are Zariski-locally trivial affine space bundles.
\end{lemma}

\begin{proof}
Since $Y_1=\lambda^{-1}(1)$, the map $g_1$ is the restriction of the projection $\mathcal{U}\to Y$ to a Zariski closed subscheme. $\Ker\lambda$ is a subbundle of rank $2r-1$ of $\mathcal{U}$, and fiberwise, it acts freely on $\mathcal{U}$ by translations. Since this action preserves the image of $\lambda$, it induces an action on $Y_1$. Since the action is fiberwise it preserves $Y$, and $Y$ is obtained as the quotient of $Y_1$ by this action. $\lambda$ is surjective by construction, thus, $Y_1 \to Y$ is a torsor for the vector subbundle $\Ker\lambda$, and then $g_1$  is a Zariski-locally trivial affine space bundle with fibre $\mathbb{A}^{2r-1}$.

In order to obtain the statement for $g_2$, let us consider the rank $2r-1$ vector bundle $g_1^*\Ker\lambda \to Y_1$. A point in this is given by a triple $(y,e,f)$ with $y$ in $Y$, and $e,f$ two vectors in $[y]$, with $\lambda(f)=1$ and $\lambda(e)=0$. We then get a surjective map of vector bundles 
\begin{align*}
    \lambda':g_1^*\Ker\lambda & \to \mathcal{O}_S\\
    (y,e,f) &\mapsto \psi(e,f),
\end{align*}
and $Y_2=\lambda'^{-1}(1)$. $Y_2$ is then a Zariski closed subscheme of the vector bundle $g_1^{-1}\Ker\lambda$ over $Y_1$, and $\Ker\lambda'$ is a subbundle of rank $2r-2$ that acts fiberwise on $g_1^{-1}\Ker\lambda$ by translations. The induced map $Y_2 \to Y_1$ is the map $g_2$. By applying the previous argument to this situation, we obtain that $g_2$ is a Zariski locally trivial affine space bundle with fiber $\mathbb{A}^{2r-2}$. 
\end{proof}

\begin{lemma}
\label{lemma:smallerGrassmannian}
    Let $\Tilde{Y}$ be the space
    $$\Tilde{Y} \coloneqq\{(e,f)\in V \times V \mid e_{2n}=0, \; f_{2n}=1, \; \phi (e,f)=1\}.$$
    Then, there is an isomorphism of $k$-schemes $Y_2 \simeq \Tilde{Y}\times \HGr(r-1,n-1)$.
\end{lemma}

\begin{proof}
    We can write $V$ as $\langle x_1 \rangle \oplus V' \oplus \langle y_1 \rangle$, with $V' = \langle x_2, \ldots, x_n, y_n, \ldots, y_2 \rangle$. $V'$ equipped with the symplectic structure induced by $\phi$ is a symplectic vector subspace of $V$, and we can identify $\HGr(r-1,n-1)$ with $\HGr(r-1,V')$. A point of $\Tilde{Y}\times \HGr(r-1,n-1)$ is then a triple $(e,f,x)$ with $e,f$ vectors of $V$ satisfying the three defining conditions of $\Tilde{Y}$, and $x$ representing a symplectic subspace $[x]$ of $V'$ of rank $2r-2$. One can associate to the pair of vectors $(e,f)$ two invertible matrices
$$\rho_e=
\begin{pmatrix}
    1 & 0 & 0 & & \ldots  &  & & & 0 \\
    e_2 & 1 & 0 & & & & & & 0 \\
    e_3 & 0 & 1 & 0 & & 0 & & & \ldots \\
    \ldots  & & & 1 & & & & & \ldots \\
    \ldots & & & & 1 & & & & \ldots \\
    \ldots & & 0 & & & 1 & & & \ldots \\
    e_{2n-2} & & & & & & 1 & & 0 \\
    e_{2n-1} & & & & & \ldots & 0 & 1 & 0 \\
    0 & -e_{2n-1} & -e_{2n-2} & \ldots & -e_{n+1} & e_{n} & \ldots & e_{2} & 1
\end{pmatrix},$$

$$\rho_f=
\begin{pmatrix}
    1 & -f_{2n-1} & -f_{2n-2} & \ldots & -f_{n+1} & f_n & \ldots & f_2 & f_1 \\
    0 & 1 & 0 & & & & & & f_2 \\
    0 & 0 & 1 & 0 & & & 0 & & f_3 \\
    \ldots  & & & 1 & & & & & \\
    \ldots & & & & 1 & & & & \ldots \\
    \ldots & & 0 & & & 1 & & & \ldots \\
    \ldots & & & & & & 1 & & \\
    0 & & & & & & 0 & 1 & f_{2n-1} \\
    0 & & & \ldots & & & & 0 & 1
\end{pmatrix}.
$$
Let now $\rho_{e,f}\coloneqq \rho_f \cdot \rho_e$. We get $\rho_{e,f}(x_1)=\rho_f(1, e_2, \ldots, e_{2n-1},0)=e$, and $\rho_{e,f}(y_1)=\rho_f(y_1)=f$. It is immediate to see that $\rho_{e,f}$ preserves the form $\phi$. Thus, for each symplectic subspace $W$ of $V$, $\rho_{e,f}$ maps $W^{\perp}$ isometrically to $(\rho_{e,f}(W))^{\perp}$. In particular, $\rho_{e,f}(V')=\rho_{e,f}(\langle x_1,y_1\rangle^{\perp}) \simeq \langle e,f \rangle^{\perp}$. We can then define a map:
\begin{align*}
    \gamma:\Tilde{Y}\times \HGr(r-1,n-1) & \to Y_2 \\
    ((e,f),x) & \mapsto (\rho_{e,f}([x])\oplus \langle e \rangle \oplus \langle f \rangle, e, f),
\end{align*}
and note that this is an isomorphism with inverse
\begin{align*}
    \gamma^{-1}: \; \;  Y_2 & \to \Tilde{Y}\times \HGr(r-1,n-1)\\
    (y,e,f) & \mapsto ((e,f), \rho_{e,f}^{-1}([y])\cap V').
\end{align*}
\end{proof}

We get a map $q:Y_2 \to \HGr(r-1,n-1)$ by identifying $Y_2$ with $\Tilde{Y}\times \HGr(r-1,n-1)$ thanks to Lemma \ref{lemma:smallerGrassmannian}, and applying the canonical projection. Moreover, $\Tilde{Y}$ is isomorphic to $\A^{4n-3}$, with inverse given by:
\begin{align*}
    \mathbb{A}^{4n-3} & \to \Tilde{Y}\\
    (x_1,\ldots,x_{4n-3}) &\mapsto (e_1,x_1,\ldots,x_{2n-2},0,x_{2n-1}, x_{2n}, \ldots, x_{4n-3},1),
\end{align*}
where $e_1=1-\phi_{2n-2}((x_1,\ldots,x_{2n-2}),(x_{2n},\ldots,x_{4n-3}))$. The map $q$ is then a trivial $\mathbb{A}^{4n-3}$-bundle, and as such, it is an $\mathbb{A}^1$-equivalence.

We summarize our discussion about the open stratum $Y$ in the following proposition.

\begin{prop} \label{prop:summary}
Let $Y=\HGr(r,n)\setminus N^+$. Then  we have a zig-zag diagram
$$
\begin{tikzcd}
Y & Y_1 \arrow[l,swap,"g_1"] & Y_2 \arrow[l,swap,"g_2"] \arrow[r,"q"] & \HGr(r-1,n-1)   
\end{tikzcd}
$$
where the maps $g_1,g_2$ and $q$ are Zariski locally trivial affine space bundles, hence $\mathbb{A}^1$-weak equivalences. Moreover, $q$ is a trivial affine space bundle with fiber $\mathbb{A}^{4n-3}$.
\end{prop}  

We can now prove the main fact of this subsection.

\begin{prop}
\label{prop:bundlesonHGr}
    For $v\in K_0(\HGr(r,n))$, $p_{\HGr(r,n)\#}\Sigma^v1_{\HGr(r,n)}\in \SH(k)$ is cellular.
\end{prop}

\begin{proof}
 We prove the statement for all $n \ge r \ge 0$ by induction on $r$.

For $r=0$, $\HGr(r,n)$ is isomorphic to $\Spec k$, and the statement is immediate. We can then assume $r > 0$ and proceed by induction on $n \ge r$.

Again, the case $n=r$ is trivial, so we assume $n>r+1$.

We rename the closed subscheme $N^+$ of $\HGr(r,n)$ as $Z$, and again $Y$ denotes its open complement. Since $Z \to \HGr(r,n-1)$ is a rank $2r$ vector bundle, $Z$ is smooth over $k$ of codimension $2r$ in $\HGr(r,n)$. Let $p$, $p_Z$, $p_Y$ denote the structure maps of $\HGr(r,n)$, $Z$ and $Y$.

From the diagram
\begin{equation}
\label{diag:loc.sequence.HGr}
\begin{tikzcd}
Z \arrow[r, hook, "i"] & \HGr(r,n) & Y \arrow[l, swap, hook, "j"]  
\end{tikzcd}
\end{equation}
we get the localization distinguished triangle
\begin{equation}
\label{eq:localization_HGr}
    j_!j^! \to \id_{\SH(\HGr(r,n))} \to i_*i^* \to j_!j^![1]
\end{equation}
of endofunctors of $\SH(\HGr(r,n))$, by Proposition \ref{prop:localization}. By composing this sequence with $p_{\#}$ and applying it to $\Sigma^v1_{\HGr(r,n)}$, we get the cofiber sequence in $\SH(k)$
\begin{equation}
    \label{eq:s.e.s.ofThomSpaces}
    p_{\#}j_!j^!\Sigma^v1_{\HGr(r,n)} \to p_{\#}\Sigma^v1_{\HGr(r,n)} \to p_{\#}i_*i^*\Sigma^v1_{\HGr(r,n)}.
\end{equation}
 
Thanks to Proposition \ref{prop:s.e.s.-cellularity}, it is enough show that the left-hand term and the right-hand term of \eqref{eq:s.e.s.ofThomSpaces} belong to $\langle 1_k \rangle$.

Let us see the left-hand term first. Because $j$ is an open immersion, we have $j_!\simeq j_\#$ and $j^!\simeq j^*$, then $$p_{\#}j_!j^!\Sigma^v1_{\HGr(r,n)}\simeq p_{Y\#}j^*\Sigma^v1_{\HGr(r,n)} \simeq p_{Y\#}\Sigma^{j^*v}1_Y.$$
Proposition~\ref{prop:summary} gives the diagram
$$
\begin{tikzcd}
Y & Y_1 \arrow[l,swap,"g_1"] & Y_2 \arrow[l,swap,"g_2"] \arrow[r,"q"] & \HGr(r-1,n-1)   
\end{tikzcd}
$$
where the maps $g_1,g_2$ and $q$ are $\mathbb{A}^1$-weak equivalences, in fact, affine space bundles. Moreover, $q$ is a trivial $\mathbb{A}^{4n-3}$-bundle, hence, it has a zero section $s_0:\HGr(r-1,n-1) \to Y_2$, which is an inverse of $q$ in $\sH(k)$. The composition $h \coloneqq g_1 \circ g_2 \circ s_0: \HGr(r-1,n-1) \to Y$ is thus an $\mathbb{A}^1$-weak equivalence. Now, $g_1,g_2$ and $q$ are affine space bundles, and $Y_1$, $Y_2$ and 
$\HGr(r-1,n-1)$ are smooth, hence, by $\A^1$-homotopy invariance of $K$-theory on regular schemes \cite[\S7 paragraph 1, and Proposition 4.1]{Quillen:K}, the maps $q^*:K_0(\HGr(r-1,n-1))\to K_0(Y_2)$ and $(g_1g_2)^*:K_0(Y)\to K_0(Y_2)$ are isomorphisms. Thus, there is an element $v'\in K_0(\HGr(r-1,n-1))$ such that $q^*v'=(g_1g_2)^*j^*v$ in $K_0(Y_2)$, and, since $g_1,g_2$ and $q$ are $\A^1$ weak equivalences, we have
$$
p_{Y\#}\Sigma^{j^*v}1_Y \simeq p_{(\HGr(r-1,n-1)\#}\Sigma^{v'}1_{\HGr(r-1,n-1)}
$$
in $\SH(k)$. 
By the induction hypothesis on $r$, we conclude that $p_{Y\#} \Sigma^{j^*v}1_Y$ belongs to $\langle 1_k \rangle$.

Let us now prove the claim for the right-hand term of \eqref{eq:s.e.s.ofThomSpaces}. Let us recall that $Z$ is isomorphic to the tautological rank $2r$ symplectic bundle over $\HGr(r,n-1)$, and as such, it is a smooth $k$-scheme, with a projection $q':Z\to \HGr(r,n-1)$ realizing $Z$ as a vector bundle over $\HGr(r,n-1)$. For simplicity, let us call $X \coloneqq \HGr(r,n)$. By Theorem \ref{thm:MorVoePurity}, we have
$$p_{\#}i_*i^*\Sigma^v1_X \simeq p_{Z\#}\Sigma^{N_i}i^*\Sigma^v1_X \simeq p_{Z\#}\Sigma^{N_i+i^*v}1_Z.$$
Since $q'$ is a vector bundle, there is an element $v''\in K_0(\HGr(r,n-1))$ such that $q'^*v''=i^*v+N_i$ in $K_0(Z)$. Since $q'$ is an $\A^1$-weak equivalence,
$$p_{Z\#}\Sigma^{N_i + i^*v}1_Z\simeq p_{\HGr(r,n-1)\#}\Sigma^{v''}1_{\HGr(r,n-1)}.$$
By our induction hypothesis, we have that $p_{Z\#}\Sigma^{N_i + i^*v}1_Z$ belongs to $\langle 1_k \rangle$.

Therefore, we conclude that $p_{\#}\Sigma^V1_{\HGr(r,n)}$ belongs to $\langle 1_k \rangle $.
\end{proof}

\begin{corollary}
\label{cor:cellularityofMSp}
    The symplectic cobordism spectrum $\MSp$ is cellular.
\end{corollary}

\begin{proof}
    By definition, $\MSp$ is isomorphic to the filtered colimit of the system
    $$\Sigma^\infty_{\P^1}\MSp_0 \to \Sigma^{-4,-2}\Sigma^\infty_{\P^1}\MSp_2 \to \Sigma^{-8,-4} \Sigma_{\P^1}^\infty \MSp_4 \to \ldots,$$
    which is in particular a homotopy colimit. Thus, it is enough to show that $\Sigma^\infty_{\P^1}\MSp_{2r}$ is cellular for all $r$. $\MSp_{2r}$ is in turn the homotopy colimit $\colim_n\Th(E_{r,n}^{\Sp})$, and $\Sigma^\infty_{\P^1}\Th(E_{r,n}^\Sp)$ is cellular for all $r,n$ by Proposition \ref{prop:bundlesonHGr}. This concludes the proof.
\end{proof}

\subsection{The motive of $\MSp$}

By adapting the methods of the previous subsection, one can express the cellularity of $\MSp$ in terms of its motive $\MSp \wedge H\Z$. We introduce some notation: for a free graded $\Z$-module $\Z \cdot x$, with $x$ in degree $n$, we write $\Z \cdot x \otimes \sE$ for the motivic spectrum $\Sigma^{-2n,-n}\sE \in \SH(k)$. This naturally extends to define $M_* \otimes \sE$ for $M_*=\oplus_{n \in \Z}M_n$ a direct sum of free graded $\Z$-modules.

\begin{prop}
\label{prop:MotiveOfMSp}
    The motive of $\MSp$ in $\DM(k)$ has the form
    $$\MSp \wedge H\Z \simeq \Z[b'_1,b_2',\ldots] \otimes H\Z,$$
    with $b_i'$ a polynomial generator of degree $-2i$.
\end{prop}

\begin{proof}
    First, we show that for all $n \ge r \ge 0$, the $H\Z$-module $\Sigma^\infty_{\P^1} \HGr(r,n)_+ \wedge H\Z$ is of the form $\oplus_s \Sigma^{4m_s,2m_s}H\Z$, for a finite number of non-negative integers $m_s$. 

    For $r=0$ or $r=n$, the statement trivially holds. We can then suppose $n>r>0$ and the claim true for all quaternionic Grassmannians $\HGr(r',n')$ with $r'\le r$ and $n'<n$, or $r'<r$ and $n'\le n$. 
    
    We take $Z \coloneqq N^+$, $X\coloneqq \HGr(r,n)$ with structure map $p$, and $i,j$ the respective inclusion as in \eqref{diag:loc.sequence.HGr}, and we consider again the localization distinguished triangle \eqref{eq:localization_HGr}. Applying $p_\#(-)(1_X)\wedge H\Z$ to \eqref{eq:localization_HGr} gives the distinguished triangle
    \begin{equation}
        \label{eq:loc.seq.HGr}
        p_\#j_!j^!1_X \wedge H\Z \to \Sigma^\infty_{\P^1}X_+ \wedge H\Z \to p_\# i_*i^*1_X \wedge H\Z \xrightarrow{\alpha} j_!j^!1_X[1] \wedge H\Z
    \end{equation}
    in $\DM(k)$. As in the proof of Proposition \ref{prop:bundlesonHGr}, we have $p_\#j_!j^!1_X \simeq p_{Y\#}1_Y$ and $p_\#i_*i^*1_X \simeq p_{Z\#}\Sigma^{N_i}1_Z$, with $N_i$ the normal bundle of $i$.
    
    For the remainder of the proof, except where indicated to the contrary, we will work in $\DM(k)$.

  Let us look at the right-hand term $p_{Z\#}\Sigma^{N_i}1_Z \wedge H\Z$ of \eqref{eq:loc.seq.HGr}. Since $p_{Z\#}$ satisfies projection formula against $p_Z^*$ by Remark \ref{rmk:smoothsharp-properties}, we have $p_{Z\#}\Sigma^{N_i}1_Z \wedge H\Z \simeq p_{Z\#}\Sigma^{N_i}p_Z^*H\Z,$
    and through the isomorphism $\th^f_{H\Z \Mod}(N_i)$, we get $p_{Z\#}\Sigma^{N_i}p_Z^*H\Z \xrightarrow{\sim} \Sigma^{4c,2c}p_{Z\#}p_Z^*H\Z$, with $2c$ the codimension of $Z$ in $X$. By the projection formula again, we have $p_{Z\#}p_Z^*H\Z \simeq \Sigma_{\P^1}^\infty Z_+\wedge H\Z$, but $Z$ is a vector bundle over $\HGr(r,n-1)$, so $\Sigma_{\P^1}^\infty Z_+ \simeq \Sigma_{\P^1}^\infty \HGr(r,n-1)_+$. We can then apply the induction hypothesis to get
    $$p_{Z\#}\Sigma^{N_i}1_Z \wedge H\Z\simeq \Sigma^{4c,2c}\oplus_i \Sigma^{4m_i^Z,2m_i^Z}H\Z$$
    for a finite number of non-negative integers $m_i^Z$.

    Let us look at the left-hand term $p_{Y\#}1_Y \wedge H\Z$ of \eqref{eq:loc.seq.HGr}. As in the proof of Proposition \ref{prop:bundlesonHGr}, there is an $\A^1$-weak equivalence $h:\HGr(r-1,n-1) \to Y$, then $p_{Y\#}1_Y \wedge H\Z \simeq \Sigma_{\P^1}^\infty \HGr(r-1,n-1)_+ \wedge H\Z$. We can then use again our induction hypothesis to get
    $$p_{Y\#}1_Y \wedge H\Z\simeq \oplus_{i'} \Sigma^{4m_{i'}^Y,2m_{i'}^Y}H\Z$$
    for a finite number of non-negative integers $m_{i'}^Y$.

    Now, the boundary map $\alpha$ in \eqref{eq:loc.seq.HGr} is an element in
    \begin{multline*}
         [p_\# i_*i^*1_X \wedge H\Z,j_!j^!1_X[1] \wedge H\Z]_{\DM(k)}\cong
         [\Sigma^{4c,2c}\oplus_i \Sigma^{4m_i^Z,2m_i^Z}H\Z,\Sigma^{1,0}\oplus_{i'} \Sigma^{4m_{i'}^Y,2m_{i'}^Y}H\Z]_{\DM(k)} \cong \\
         \oplus_{i,i'}[\Sigma^{4(c+m_i^Z),2(c+m_i^Z)}H\Z, \Sigma^{4m_{i'}^Y+1,2m_{i'}^Y}H\Z]_{\DM(k)} \cong \\
         \oplus_{i,i'}[H\Z, \Sigma^{4(m_{i'}^Y-m_i^Z-c)+1,2(m_{i'}^Y-m_i^Z-c)}H\Z]_{\DM(k)}.
    \end{multline*}
    By the free-forgetful adjunction, this corresponds to an element in
    $$\oplus_{i,i'}[1_k, \Sigma^{4(m_{i'}^Y-m_i^Z-c)+1,2(m_{i'}^Y-m_i^Z-c)}H\Z]_{\SH(k)} = \oplus_{i,i'}H\Z^{4(m_{i'}^Y-m_i^Z-c)+1,2(m_{i'}^Y-m_i^Z-c)}(\Spec k),$$
    but each abelian group $H\Z^{4(m_{i'}^Y-m_i^Z-c)+1,2(m_{i'}^Y-m_i^Z-c)}(\Spec k)$ is trivial by Theorem \ref{thm:MazWeiHZ}. We conclude that $\alpha=0$, and the sequence \eqref{eq:loc.seq.HGr} splits. Hence 
    $$\Sigma^\infty_{\P^1}\HGr(r,n)_+ \wedge H\Z \simeq \oplus_s\Sigma^{4m_s,2m_s}H\Z$$
    for a finite number of non-negative integers $m_s$, proving the claim.

    Now, let us recall that $\BSp_{2r}=\colim_N \HGr(r,rN)$. Since the $\text{free}_{H\Z}$ functor $(-)\wedge H\Z$ is left adjoint, it commutes with small colimits, hence 
    \begin{equation}
    \label{eq:MotiveOfBsp}
        \Sigma_{\P^1}^\infty \BSp_{2r \, +}\wedge H\Z \simeq \colim_N(\Sigma_{\P^1}^\infty \HGr(r,rN)_+\wedge H\Z).
    \end{equation}
    
    Now we claim that, once we fix $r$, for each non-negative integer $m$ the number of summands in the right-hand side for which $m_s=m$ is eventually constant for $n>>0$. We can assume that this is the case for $\Sigma^\infty_{\P^1} \HGr(r-1,n)_+ \wedge H\Z$ by induction. Noting that $Z$ has codimension $c=2r$, independent of $n$, the above argument for the identity $\Sigma^\infty_{\P^1} \HGr(r,n)_+ \wedge H\Z \simeq \oplus_s \Sigma^{4m_s,2m_s}H\Z$ shows that the pairs $(4m,2m)$ that occur for $\HGr(r,n+1)$ are the pairs that occur for $\HGr(r-1,n)$ together with pairs of the form $(4(r+m'), 2(r+m'))$ where $(4m',2m')$ is a pair occurring for $\HGr(r,n)$. This proves the claim.
    
    We can then rewrite \eqref{eq:MotiveOfBsp} as
    \begin{equation}
    \label{eq:MotiveBsp2}
        \Sigma_{\P^1}^\infty \BSp_{2r \, +}\wedge H\Z \simeq \oplus_s \Sigma^{4m_s,2m_s}H\Z,
    \end{equation}
    where now the sum will be infinite, but for all $m$ there is only a finite number of summands for which $m_s=m$.

    As before, by using projection formula of $p_{\BSp_{2r}\#}$ against $p_{\BSp_{2r}}^*$ and the isomorphism $\th^f_{H\Z \Mod}(E_{2r}^\Sp):\Sigma^{E_{2r}^\Sp}p_{\BSp_{2r}}^*H\Z \xrightarrow{\sim}\Sigma^{4r,2r}p_{\BSp_{2r}}^*H\Z$, we get
    $$p_{\BSp_{2r}\#}\Sigma^{E_{2r}^\Sp}1_{\BSp_{2r}}\wedge H\Z \simeq \Sigma^{4r,2r}p_{\BSp_{2r}\#}1_{\BSp_{2r}}\wedge H\Z.$$
    Thus, since $\Sigma_{\P^1}^\infty \MSp_{2r}\simeq p_{\BSp_{2r}\#}\Sigma^{E_{2r}^\Sp}1_{\BSp_{2r}}$, we can write
    $$\Sigma_{\P^1}^\infty \MSp_{2r} \wedge H\Z \simeq \Sigma^{4r,2r}\oplus_s \Sigma^{4m_s,2m_s}H\Z,$$
    with $m_s$ as in \eqref{eq:MotiveBsp2}. Since $\MSp =\colim_r(\Sigma^{-4r,-2r}\Sigma^\infty_{\P^1}\MSp_{2r})$, we can use again that $(-)\wedge H\Z$ commutes with small colimits and that, for each integer $m$, the number of summands in the right hand side of \eqref{eq:MotiveBsp2} such that $m_s=m$ becomes constant for $r>>0$, to get
    \begin{equation}\label{eqn:MSpDecom}
  \MSp \wedge H\Z \simeq \oplus_\alpha \Sigma^{4n_\alpha,2n_\alpha}H\Z,
  \end{equation}
    where, for all integers n, there are only finitely many indices $\alpha$ with $n_\alpha=n$, and each $n_\alpha$ is non-negative.

    We have then the following chain of isomorphisms:
  \begin{multline*}H\Z^{*,*}(\MSp)= [\MSp, \Sigma^{*,*}H\Z]_{\SH(k)}
        \cong [\MSp \wedge H\Z, \Sigma^{*,*}H\Z]_{\DM(k)}\\
        \overset{\eqref{eqn:MSpDecom}}{\cong} [\oplus_\alpha \Sigma^{4n_\alpha,2n_\alpha}H\Z,\Sigma^{*,*}H\Z]_{\DM(k)}
        \cong \prod_\alpha[H\Z,\Sigma^{*-4n_\alpha,*-2n_\alpha}H\Z]_{\DM(k)}\\
        \cong \prod_\alpha[1_k,\Sigma^{*-4n_\alpha,*-2n_\alpha}H\Z]_{\SH(k)}
       =\oplus_\alpha H\Z^{*,*}(\Spec k)\cdot b_\alpha,
   \end{multline*}
with $b_\alpha$ in bidegree $(4n_\alpha,2n_\alpha)$. But from Theorem \ref{thm:cohomologyofmsp}, with Remark \ref{rmk:homogeneous}, we know that $H\Z^{*,*}(\MSp) \simeq H\Z^{*,*}(\Spec k)[b_1,b_2,\ldots]$,
    with $b_i$ in bidegree $(4i,2i)$, in other words, $H\Z^{*,*}(\MSp)$ is a free bigraded $H\Z^{*,*}$-module with basis the monomials in $b_1, b_2,\ldots$. Thus, the bidegrees $(4n_\alpha,2n_\alpha)$ that occur in our decomposition \eqref{eqn:MSpDecom} are exactly the bidegrees of the monomials in $b_1, b_2,\ldots$.
    
    From this we see that 
    $$\MSp \wedge H\Z \simeq \Z[b_1',b_2',\ldots] \otimes H\Z,$$
with $b_i'$ a polynomial generator of degree $-2i$.
\end{proof}

\begin{corollary}
\label{cor:product_map}
    For any integer $m \ge 1$, the canonical map
    $$H\Z^{*,*}(\MSp)^{\otimes_{H\Z^{*,*}(\Spec k)}m} \to H\Z^{*,*}(\MSp^{\wedge m})$$
    is an isomorphism. The same holds if we replace $H\Z$ with $H\Z/\ell$ for any prime $\ell$.
\end{corollary}

\begin{proof}
    In the proof of Proposition \ref{prop:MotiveOfMSp} we have seen that we can write the motive of $\MSp$ as
    \begin{equation}
    \label{eq:MotiveMSp}
        \MSp \wedge H\Z \simeq \oplus_\alpha \Sigma^{4n_\alpha,2n_\alpha}H\Z,
    \end{equation}
    for some non-negative integer $n_\alpha$, and for each non-negative integers $n$, there are only finitely many indices $\alpha$ such that $n_\alpha=n$.

    By induction on $m$, we will also have 
    \begin{equation}
    \label{eq:MotiveMSp^n}
       (\MSp^{\wedge m}) \wedge H\Z \simeq (\MSp \wedge H\Z)^{\wedge_{H\Z} m} \simeq \oplus_\beta \Sigma^{4m_\beta,2m_\beta}H\Z, 
    \end{equation}
    with analogous conditions on the non-negative integers $m_\beta$. Thus, computing the motivic cohomology of $\MSp^{\wedge m}$ we have
    \begin{multline}
    \label{eq:chain1}
        H\Z^{p,q}(\MSp^{\wedge m})=[\MSp^{\wedge m},\Sigma^{p,q}H\Z]_{\SH(K)} \simeq [\MSp^{\wedge m} \wedge H\Z, \Sigma^{p,q}H\Z]_{\DM(k)} \\
        \simeq [\oplus_\beta \Sigma^{4m_\beta,2m_\beta}H\Z,\Sigma^{p,q}H\Z]_{\DM(k)}\simeq [1_k,\Sigma^{p-4m_\beta,q-2m_\beta}H\Z]_{\SH(k)} \\ \simeq \prod_\beta H\Z^{p-4m_\beta,q-2m_\beta}(\Spec k).
    \end{multline}
    Since $H\Z^{a,b}(\Spec k)$ vanishes for $b<0$ by Theorem \ref{thm:MazWeiHZ}, we also have 
    \begin{equation}
    \label{eq:chain2}
        \prod_\beta H\Z^{p-4m_\beta,q-2m_\beta}(\Spec k) \simeq \oplus_\beta H\Z^{p-4m_\beta,q-2m_\beta}(\Spec k).
    \end{equation}

    The expressions \eqref{eq:MotiveMSp} and \eqref{eq:MotiveMSp^n} give the isomorphism of $H\Z$-modules
    $$(\oplus_\alpha \Sigma^{4n_\alpha,2n_\alpha}H\Z)^{\wedge_{H\Z} m} \simeq \oplus_\beta \Sigma^{4m_\beta,2m_\beta} H\Z.$$
    With the product in motivic cohomology, this induces the isomorphism
    $$(\oplus_\alpha H\Z^{*-4n_\alpha,*-2n_\alpha}(\Spec k))^{\otimes_{H\Z^{*,*}(\Spec k)} m} \xrightarrow{\sim} \oplus_\beta H\Z^{*-4m_\beta,*-2m_\beta}(\Spec k),$$
    which can be rewritten, through \eqref{eq:chain2} and \eqref{eq:chain1}, as
    $$H\Z^{*,*}(\MSp)^{\otimes_{H\Z^{*,*}(\Spec k)} m} \xrightarrow{\sim} H\Z^{*,*}(\MSp^{\wedge m}).$$
\end{proof}

\begin{rmk}
\label{rmk:HZ/ell}
    Since Theorem \ref{thm:MazWeiHZ} remains true after replacing $H\Z$ with $H\Z/\ell$ for any prime $\ell$, the proofs of Proposition \ref{prop:MotiveOfMSp} and Corollary \ref{cor:product_map} also apply after replacing $H\Z$ with $H\Z/\ell.$
\end{rmk}

\section{The Adams spectral sequence for $\MSp$}
\label{section:A.s.s}

From now on, we fix an odd prime $\ell$ different from $\chr k$, and we use the shorthand $H^{*,*}(-)\coloneqq H\Z/\ell^{*,*}(-)$ for the mod-$\ell$ motivic cohomology.

We want to compute the mod-$\ell$ Adams spectral sequence for $\MSp$ following the same lines as \cite[Chapters 4-5-6]{lev:ellcoh}. We first need to briefly recall the construction of the mod-$\ell$ motivic Steenrod algebra, introduced in \cite{Voev-power}.

\subsection{The motivic Steenrod algebra}

Let us recall that a bistable cohomology operation of bidegree $(i,j)$ on a motivic spectrum $\E \in \SH(k)$ is a family $\{\phi_{p,q}\}_{p,q \in \Z}$ of natural transformations of functors in $\phi_{p,q}:\E^{p,q}(-) \to \E^{p+i,q+j}(-)$ that commute with arbitrary suspensions. The set of all bistable cohomology operations on $\E$ has a natural ring structure induced by composition, which makes it a bigraded $\Z/\ell$-algebra. By \cite[Theorem 1.1]{Hoy:Steenrod}, the algebra of bistable cohomology operations on $H\Z/\ell$ is naturally isomorphic to the $\Z/\ell$-algebra of shifted endomorphisms of $H\Z/\ell$, $\text{End}_{\SH(k)}^{*,*}(H\Z/\ell):=[H\Z/\ell,\Sigma^{*,*}H\Z/\ell]_{\SH(k)}$, where $\Phi:H\Z/\ell\to \Sigma^{i,j}H\Z/\ell$ defines the cohomology operation $\phi:=\{\phi_{p,q}:=\Sigma^{p,q}\Phi_*\}_{p,q\in\Z}$.

The short exact sequence
$$0 \to \Z/\ell \xrightarrow{\cdot \ell} \Z /\ell^2 \to \Z/\ell \to 0$$
induces the distinguished triangle
\begin{equation}
    \label{eq:Bockstein}
    H\Z/\ell \to H\Z/\ell^2\to  H\Z/\ell \xrightarrow{\beta_{0,0}} \Sigma^{1,0}H\Z/\ell
\end{equation}
in $\SH(k)$. The resulting map $\beta_{0,0}:H\Z/\ell\to \Sigma^{1,0}H\Z/\ell$ in $\SH(k)$ is the \emph{Bockstein homomorphism}. This defines a bistable cohomology operation $\beta=\{\beta_{p,q}:=\Sigma^{p,q}\beta_{0,0}\}_{p,q}$ of bidegree $(1,0)$ on $H\Z/\ell$, called the \emph{Bockstein operator}, which satisfies the following properties:
\begin{enumerate}
    \item $\beta \circ \beta = 0,$
    \item $\beta (xy)=\beta(x)\cup y + (-1)^p x \cup \beta(y)$, for $x \in (H\Z/\ell)^{p,*}(-)$ 
\end{enumerate}
(see \cite[Section 8]{Voev-power}).

In \cite[Section 9]{Voev-power}, for $i \ge 0$, Voevodsky defines reduced power operations $P^i$, with $P^0=1$, which we can consider as bistable cohomology operations of bidegree $(2i(\ell-1),i(\ell-1))$ on $H\Z/\ell$. 

\begin{defn}
    The \emph{motivic mod-$\ell$ Steenrod algebra} $A^{*,*}=A^{*,*}(k,\Z/\ell)$ is the subalgebra of the algebra of bistable cohomology operations on the mod $\ell$ motivic cohomology $H\Z/\ell$ over $k$ generated by the Bockstein operator $\beta$, the reduced power operations $\{P^i\}_{i \ge0}$ and operations $x \mapsto a \cdot x$ for $a \in (H\Z/\ell)^{*,*}(\Spec k)$.
\end{defn}

In fact, $A^{*,*}$ is isomorphic to the entire algebra $\text{End}_{\SH(k)}^{*,*}(H\Z/\ell)$ of bistable cohomology operations on $H\Z/\ell$ (see \cite[Therem 1.1 (1)]{Hoy:Steenrod}). In characteristic $0$, this was proved by Voevodsky (\cite[Theorem 3.49]{Voe:EM}).

The bigraded algebra $A^{*,*}$ is a left $H^{*,*}$-module via $(a\cdot\phi)(x):=a\cdot(\phi(x))$ for $a\in H^{*,*}$, $\phi\in A^{*,*}$ and $x\in H\Z^{*,*}(\E)$.

$A^{*,*}$ is equipped with a map $\Delta: A^{*,*} \to A^{*,*} \otimes_{H^{*,*}} A^{*,*}$ defined by
$\Delta (\gamma)= \sum_i \alpha_i \otimes \beta_i$, where the bistable operations $\alpha_i$, $\beta_i$ are determined by the relation
$$\gamma(x \cup y)= \sum_i(-1)^{ab_i} \alpha_i(x) \cup \beta_i(y),$$
for $x\in H^{a,a'}(\sE)$ and $\beta_i$ of bidegree $(b_i,b'_i)$. Voevodsky also defines an $H^{*,*}$-submodule $(A^{*,*} \otimes_{H^{*,*}} A^{*,*})_r$ for which the formula $(x\otimes y)\cdot (x'\otimes y')=xx'\otimes yy'$ gives a ring structure, and which contains the image of $\Delta$, and calls a \emph{coproduct}, which an abuse of notation, the co-associative map

\begin{equation}\label{eqn:ACoproduct}
\Psi:A^{*,*}\to (A^{*,*} \otimes_{H^{*,*}} A^{*,*})_r, 
\end{equation}
(see \cite[Lemmas 11.6-11.9]{Voev-power} for the details).

In \cite[Section 13]{Voev-power}, Voevodsky defines the \emph{Milnor basis} for $A^{*,*}$, given by cohomology operations $\rho(E,R)$, for $E=(\epsilon_0, \epsilon_1, \epsilon_2, \ldots)$ a sequence of zeros and ones that are almost all zeros, and $R=(r_1, r_2, r_3, \ldots)$ a sequence of non-negative integers that are almost all zero. Let us write $Q(E) \coloneqq \rho(E, \underaccent{\bar}{0})$ and $\mathcal{P}^R \coloneqq \rho(\underaccent{\bar}{0},R)$. For the sequence $e_i=(0,0, \ldots, 0,1,0, \ldots)$ having $1$ only in the $i$-th place, he defines $Q_i\coloneqq Q(e_i)$. In this basis, we have $P^i = \mathcal{P}^{i\cdot e_1}$ and $Q_0 = \beta$.

The motivic Steenrod algebra $A^{*,*}$ is the algebraic analogue of the classical Steenrod algebra $A^{\text{top}} \coloneqq A^{\top}_{\ell}$ described in \cite{mil:Steenrod}. Analogously tp $A^{*,*}$, $A^{\text{top}}$ acts on the singular cohomology $H^*(-,\Z/\ell)$, has a Bockstein operator $\beta_{\top}$ (see \cite[Chapter IV]{steenrod:cohomology}), power operations $P^i_{\top}$, and Milnor basis $\rho^{\top}(E,R)$. We let $Q^{\top}(E) \coloneqq \rho^{\top}(E, \underaccent{\bar}{0})$, $\mathcal{P}_{\top}^R \coloneqq \rho^{\top}(\underaccent{\bar}{0},R)$ and $Q_i^{\top} \coloneqq Q^{\top}(e_r)$. We have $P^i_{\top}=\mathcal{P}^{i \cdot e_1}_{\top}$, $P_{\top}^0=1$, and $Q_0^{\top}=\beta_{\top}$.

We now recall a few properties of the power operations.

\begin{lemma}
\label{lemma:someproperties}
\begin{enumerate}
    \item[(1)] $P^0= \id$, and, for any $i<0$, $P^i=0$.
    \item[(2)] For $x \in H^{2n,n}(\Spec k)$, $P^n(x)=x^\ell$.
    \item[(3)] For $x \in H^{p,q}(\Spec k)$, $n > p-q$ and $n\ge q$, we have $P^n(x)=0$.
    \item[(4)] $Q_i$ has bidegree $(2 \ell^i -1, \ell^i-1)$, and for each $k \ge 1$, $P^{k \cdot e_i}$ has bidegree $(2k(\ell^i-1),k(\ell^i-1))$.
\end{enumerate}
\end{lemma}

\begin{proof}
    $(1)$ is \cite[Theorems 9.4-9.5]{Voev-power}. $(2)$ is \cite[Lemma 9.8]{Voev-power}. $(3)$ is \cite[9.9]{Voev-power}. The formulas for the bidegrees of $Q_i$ and $P^{k \cdot e_i}$ come directly from their definitions in \cite[Section 13]{Voev-power}.
\end{proof}

Finally, let us define $B \subset A^{*,*}$ as the $\Z/\ell$-subalgebra generated by $\{Q_i\}_{i\ge 0}$, and $M_B \coloneqq A^{*,*}/A^{*,*}(Q_0,Q_1, \ldots)$ as the quotient of $A^{*,*}$ by its left ideal generated by the special elements $Q_i$. Analogously, we have $B^{\top} \subset A^{\top}$ and $M_B^{\top} \coloneqq A^{\top}/A^{\top}(Q_0^{\top}Q_1^{\top}, \ldots)$. 

\begin{rmk}\label{rmk:MBDescent} Taken $\E \in \SH(k)$  and  $u\in H\Z^{2n,n}(\E)$, the left $A^{*,*}$-module structure on $H\Z^{*,*}(\E)$ gives the bigraded map  $r_u:A^{*, *}\cdot u\to H\Z^{*,*}(\E)$ of left $A^{*,*}$-modules, defined by $r_u(\alpha\cdot u):= \alpha(u)$. If $H\Z^{2*+1,*}(\E)=0$, then $r_u$ descends to a map of $A^{*,*}$-modules $r_u:M_B\cdot u\to H\Z^{*,*}(\E)$. Indeed, since $Q_i$ has bidegree $(2\ell^i+1, \ell^i)$, $Q_i(u)$ lands in $H\Z^{2(\ell^i+n)+1, \ell^i+n}(\E)=0$ for each $i$. An analogous statement holds for $M_B^{\top}$ in the topological setting.
\end{rmk}

\subsection{The $\text{Ext}_{A^{**}}$-algebra of $H^{*,*}(\MSp)$}

Let us recall that, for two bigraded $A^{*,*}$-modules $N^{**}, M^{**}$, we have the trigraded abelian group  $\Ext_{A^{**}}(N^{**}, M^{**})$ defined by
$$\Ext_{A^{**}}^{a,(b,c)}(N^{**}, M^{**}):=\Hom_{D_{\text{bgr}}(A^{**})}^{b,c}(N^{**}, M^{**}[a]),$$
where $D_{\text{bgr}}(A^{**})$ denotes the derived category of bigraded $A^{**}$-modules, and $\Hom_{D_{\text{bgr}}(A^{**})}^{b,c}(-,-)$ is the $(b,c)$ bigraded component of $\Hom_{D_{\text{bgr}}(A^{**})}(-,-)$. This gives a natural tri-grading on the sum $\Ext_{A^{**}}^{*,(*,*)}(N^{**}, M^{**}) = \oplus_{a,b,c \in \Z}\Ext_{A^{**}}^{a,(b,c)}(N^{**}, M^{**})$ (we will omit the tri-grading from the notation). In particular, if we have $\E,\E' \in \SH(k)$ and two maps $H^{*,*}(\E) \to H^{*,*}[a]$ and $H^{*,*}(\E') \to H^{*,*}[b]$ in $D_{\text{bgr}}(A^{*,*})$, we get a map
$$H^{*,*}(\E) \otimes_{H^{*,*}}H^{*,*}(\E')\to H^{*,*} \otimes_{H^{*,*}}H^{*,*}[a][b]=H^{*,*}[a+b]$$
in the derived category of $(A^{*,*}\otimes_{H^{*,}}A^{*,*})_r$-modules, and through the map $\Psi$ introduced in \eqref{eqn:ACoproduct}, we can see this map in the derived category $D_{\text{bgr}}(A^{*,*})$.

Now, the description of the cohomology of $\MSp$ given in Theorem \ref{thm:cohomologyofmsp} shows that $H^{*,*}(\MSp)$ is a free $H^{*,*}$-module, in particular flat, and by Corollary \ref{cor:product_map} the canonical map
$$H^{*,*}(\MSp) \otimes_{H^{*,*}}H^{*,*}(\MSp) \to H^{*,*}(\MSp \wedge \MSp)$$
is an isomorphism. Therefore we have a product
\begin{equation}
    \label{eq:tensor_product_ext}
    \Ext_{A^{*,*}}(H^{*,*}(\MSp),H^{*,*})\otimes_{H^{*,*}}\Ext_{A^{*,*}}(H^{*,*}(\MSp),H^{*,*}) \to \Ext_{A^{*,*}}(H^{*,*}(\MSp \wedge \MSp), H^{*,*}).
\end{equation}
Finally, the pullback $\mu^*_\MSp$ of the multiplication map gives the map
$$(\mu_{\MSp}^*)^*:\Ext_{A^{*,*}}(H^{*,*}(\MSp \wedge \MSp),H^{*,*}) \to \Ext_{A^{*,*}}(H^{*,*}(\MSp),H^{*,*}),$$
and composing $(\mu_{\MSp}^*)^*$ with \eqref{eq:tensor_product_ext} gives the product
\begin{equation}
    \label{coproduct_ext_msp}
    \Ext_{A^{*,*}}(H^{*,*}(\MSp),H^{*,*})\otimes_{H^{*,*}}\Ext_{A^{*,*}}(H^{*,*}(\MSp),H^{*,*}) \to \Ext_{A^{*,*}}(H^{*,*}(\MSp), H^{*,*}).
\end{equation}
The co-associativity of the coproduct $\Psi$ makes \eqref{eq:tensor_product_ext} associative, thus \eqref{coproduct_ext_msp} is associative. We have then a natural structure of trigraded associative $\Z/\ell$-algebra on $\Ext_{A^{*,*}}(H^{*,*}(\MSp),H^{*,*})$. Following the strategy used in \cite{lev:ellcoh} for the case of MGL, we want to find a presentation for this trigraded $\Z/\ell$-algebra.

\begin{defn}
    A \emph{partition} $\omega$ of an integer $n\ge0$ is a sequence of integers $(\omega_1,\ldots, \omega_r)$ with $\omega_1\ge \omega_2\ge\ldots\ge \omega_r > 0$ such that $|\omega|:=\sum_i\omega_i=n$. We say that a partition $\omega$ is \textit{even} if all the terms $\omega_i$ are even.
\end{defn}

For a positive integer $k$, let us extend the natural bigrading on $H^{**}$ to a bigrading on the polynomial algebra $H^{*,*}[t_1,\ldots t_k]$ by giving $t_i$ bidegree $(2,1)$. As before, we have the polynomial ring $H^{*,*}[b_1, \ldots, b_k]$ with $b_i$ of bidegre $(4i,2i)$, and we get a map from $H^{*,*}[b_1, \ldots, b_k]$ to the subalgebra of symmetric polynomials in $t_1^2,\ldots, t_k^2$ by taking $b_s$ to the elementary symmetric polynomial
$$f_s \coloneqq \sum_{i_1,\ldots, i_s}t_{i_1}^2 \cdot \cdot \cdot t_{i_s}^2$$ 
for all $s \leq k$. Also, to each even partition $\omega=(2q_1,\ldots, 2q_s)$ of an even integer $2q$, we can associate the symmetric polynomial
$$f_{\omega} \coloneqq \sum_{i_1,\ldots, i_s}t_{i_1}^{2q_1} \cdot \cdot \cdot t_{i_s}^{2q_s}.$$
Since the $H^{*,*}$-algebra of symmetric polynomials in $t_1^2,\ldots, t_k^2$ is the polynomial algebra over $H^{*,*}$ in the  elementary symmetric polynomials in the same variables, we have a uniquely defined element $P_\omega\in H^{*,*}[b_1, \ldots, b_k]$ corresponding to $f_\omega$ under the association $b_s \mapsto f_s$. For each even partition $\omega$, let $u_{\omega}$ denote the element in $H^{*,*}(\MSp)$ corresponding to $P_{\omega}$, using Theorem~\ref{thm:cohomologyofmsp} to identify $H^{*,*}(\MSp)$ with $H^{*,*}[b_1, b_2,\ldots]$. In $H^{*,*}(\MSp)$, $u_{\omega}$ has bidegree $(2 |\omega|, |\omega|)$.

Since $H^{a,b}=0$ if $a>b$ or if $b<0$ by Theorem \ref{thm:MazWeiHZ}, we have $H^{2*+1, *}=0$. As $b_i$ has bidegree $(4i, 2i)$, it follows that $H^{2*+1,*}(\MSp)=(H^{*,*}[b_1, b_2,\ldots])^{2*+1,*}=0$. Thus, using Remark~\ref{rmk:MBDescent}, for each even partition $\omega$ we can define a map of left $A^{**}$-modules
\begin{align*}
    \Phi_{\omega}:M_Bu_{\omega} & \to H^{*,*}(\MSp)\\
    \alpha \cdot u_{\omega} & \mapsto \alpha(u_{\omega}).
\end{align*}

\begin{defn}
\label{defn:l-adicpartition}
    We say that a partition $\omega=(\omega_1, \ldots, \omega_s)$ is \textit{$\ell$-adic} if at least one of its terms $\omega_r$ is of the form $\omega_r=\ell^m-1$, for some $m\ge1$. 
\end{defn}

\begin{lemma}
\label{lemma:decompmsp}
    Let $P$ be the set of all even partitions $\omega$ that are not $\ell$-adic. Then the map
    $$\Phi \coloneqq \bigoplus_{\omega \in P} \Phi_{\omega}:\bigoplus_{\omega \in P}M_Bu_{\omega}\to H^{*,*}(\MSp)$$
    is an isomorphism of left $A^{**}$-modules.
\end{lemma}

Before proving Lemma \ref{lemma:decompmsp}, we need to review the analogous statement for $\MSp^{\top}$, studied in \cite[Chapter 1]{Thomcompl}. 

Let $H^*(-)$ denote the mod $\ell$ singular cohomology $H^*(-,\mathbb{Z}/\ell)$. We take $\Sp_{2r}^\top:=\Sp_{2r}(\C)$, with the classical topology. This gives the classifying space $\BSp_{2r}^\top$, as well as the tautological rank $2r$ symplectic bundle $E^{\Sp,\top}_{2r}\to \BSp_{2r}^\top$ and the Thom space $\MSp_{2r}^\top \coloneqq \Th(E^{\Sp,\top}_{2r})$. $\MSp^{\top}$ is the colimit of the shifted suspension spectra $\Sigma^{-4r}_{S^1}\Sigma^\infty_{S_1}\MSp_{2r}^\top$, and the cohomology $H^*(\MSp^\top)$ is the limit of the cohomologies $H^{*+4r}(\MSp_{2r}^\top)$, for each fixed degree $*$, which is eventually constant. The cohomology algebra $H^{*}(\BSp_{2r}^\top)$, described for instance in \cite{BorelSerre:Steenrod}, is the polynomial algebra over $\Z/\ell$ with generators the mod $\ell$ quaternionic Borel classes $b_j^\top(E^{\Sp,\top}_{2r})\in H^{4j}(\BSp_{2r}^\top)$, for $j=1,\ldots, r.$ Novikov \cite[Lemma 3]{Thomcompl} showed that the pullback by the zero section $z_{2r}^{\top*}:H^{*+4r}(\MSp_{2r}^\top)\to
 H^{*+4r}(\BSp_{2r}^\top)$ is injective, with image the ideal generated by the top Borel class $b_r^\top$. Writing $\BSp^\top$ as $\colim_r\BSp_{2r}^\top$ gives the isomorphism $H^*(\BSp^\top)\cong \lim_rH^*(\BSp_{2r}^\top)$, and similarly $H^*(\MSp^\top)\cong \lim_rH^{*+4r}(\MSp_{2r}^\top)$. Together with the Thom isomorphisms $H^*(\BSp_{2r}^\top)\cong H^{*+4r}(\MSp_{2r}^\top)$, this  gives the commutative diagram of isomorphisms
 $$
 \begin{tikzcd}
     H^*(\BSp^\top) \arrow[d, "\wr"] \arrow[r, dashed, "\sim"] & H^*(\MSp^\top) \arrow[d, "\wr"] \\
     \lim_rH^*(\BSp_{2r}^\top) \arrow[r, "\sim"] & \lim_r H^{*+4r}(\MSp_{2r}^\top)
 \end{tikzcd}
 $$
 Under the isomorphism $H^*(\BSp^\top)\xrightarrow{\sim}H^*(\MSp^\top)$, we can write $H^*(\MSp^\top) \cong \Z/\ell[b_1^\top, b_2^\top,\ldots].$
 
 The action on $H^*(\MSp^\top)$ is induced by the action on $H^{*+4r}(\MSp_{2r}^\top)$ for $r\gg0$, thus by the action on $H^{*+4r}(\BSp_{2r}^\top)$ via $z_{2r}^{\top*}$. Let us note that this does not identify with the action of $A^\top$ on $H^*(\BSp^\top)$ via the above isomorphism $H^*(\MSp^\top)\cong H^*(\BSp^\top)$. We can approximate $H^*(\MSp^\top)$ as $A^\top$-module by $H^{*+4r}(\HGr(r,n)^\top)$, with $\HGr(r,n)^\top \coloneqq \HGr(r,n)(\C)$ with the classical topology, and similarly, we can approximate $H^*(\MSp^\top)$ as $A^\top$-module by $H^{*+4r}(\HGr(r,n)^\top)$. Explicitly, $b_j^\top \in H^{4j}(\BSp^\top_{2r})$ corresponds to the elementary symmetric polynomial $\sum_{i_1,\ldots, i_j}t_{i_i}^2 \cdot \ldots \cdot t_{i_j}^2$ in $t_1^2,\ldots, t_r^2$, and to each even partition $\omega=(2q_1,\ldots,2q_s)$, we can assign the element
    $$v^{(r)}_{\omega}\coloneqq \sum_{i_1,\ldots, i_s}t_{i_1}^{2q_1}\cdot \ldots \cdot t_{i_s}^{2q_s}\in
 H^{2|\omega|}(\BSp_{2r}^\top).$$
 The $v^{(r)}_{\omega}$ are compatible in $r$, giving the elements $v_\omega\in H^{2|\omega|}(\BSp^\top)$, and we let $u_\omega\in H^{2|\omega|}(\MSp^\top)$ be the element corresponding to $v_\omega$ via the isomorphism $H^*(\MSp^\top)\cong H^*(\BSp^\top)$. Via the Thom isomorphism, $v^{(r)}_\omega$ maps to an element $v_\omega^{(r)'}\in H^{*+4r}(\MSp_{2r}^\top)$, and then $z_{2r}^*v_\omega^{(r)'}=v^{(r)}_{\omega}\cdot b_r^\top \in
 H^{2|\omega|+4r}(\BSp_{2r}^\top)$. Dropping the ${}^{(r)}$ from the notation, we see that the action of the Steenrod algebra on  $u_\omega\in H^{2|\omega|}(\MSp^\top)$ is given by the action on $v_{\omega}\cdot b_r^\top\in
 H^{2|\omega|+4r}(\BSp_{2r}^\top)$ for $r\gg0$.
    
 The natural product $\mu_{\MSp^\top}: \MSp^\top\wedge \MSp^\top\to \MSp^\top$ on $\MSp^\top$ is induced by the embeddings $\Sp(m)\times \Sp(n) \subset \Sp(m+n)$. Together with the K\"unneth formula, $\mu_{\MSp^\top}$ gives rise to the ``diagonal map''
 \[
 \delta^\top:=\mu_{\MSp^\top}^*:H^*(\MSp^\top)\to H^*(\MSp^\top)\otimes_{\Z/\ell}H^*(\MSp^\top).
 \] 
 
   By \cite[Lemma 7]{Thomcompl} the diagonal morphism on the generators $u_{\omega}^{\top}$ has the following form:
    \begin{equation}
    \label{eq:comultiplication}
        \delta^\top(u_{\omega}^{\top})= \sum_{{(\omega_1,\omega_2)=\omega} \atop{\omega_1 \neq \omega_2}}(u_{\omega_1}^{\top} \otimes u_{\omega_2}^{\top}+u_{\omega_2}^{\top}\otimes u_{\omega_1}^{\top}) + \sum_{(\omega_1,\omega_1)=\omega}(u_{\omega_1}^{\top}\otimes u_{\omega_1}^{\top}).
    \end{equation}
    
    As a module over $A^{\top}$, $H^*(\MSp^{\top})$ is the direct sum of modules $M_B^{\top}u_{\omega}^{\top}$ over even non-$\ell$-adic partitions $\omega$ (this follows by \cite[Lemma 4]{Thomcompl}). In other words, the map
    $$\Phi^{\top} \coloneqq \bigoplus_{\omega \in P} \Phi_{\omega}^{\top}:\bigoplus_{\omega \in P}M_B^{\top}u_{\omega}^{\top}\to H^*(\MSp^{\top})$$
    where $\Phi_{\omega}^{\top}$ is defined in the same way as we defined $\Phi_{\omega}$, is an isomorphism of graded left $A^{\top}$-modules.
    
    \begin{proof}[Proof of Lemma \ref{lemma:decompmsp}] By \cite[Chapter 6, Section 1]{steenrod:cohomology}, the $P_{\top}^i$ and $\beta_{\top}$ satisfy the Adem relations and $\beta_{\top}^2=0$, and these relations determine $A^{\top}$ as $\mathbb{Z}/\ell$-algebra. The same is true for the $P^i$ and $\beta$ by \cite[Theorem 5.1]{Hoy:Steenrod}. Hence, there is a unique ring homomorphism $\theta:A^{\top}\to A^{**}$ defined by sending $P_{\top}^i$ to $P^i$ and $\beta_{\top}$ to $\beta$. In particular, by \cite[Lemma 4.1]{lev:ellcoh}, if we let $A^{*,*}_0$ be the $\mathbb{Z}/\ell$-subalgebra of $A^{*,*}$ generated by $\beta$ and the $P^i$, $\theta$ induces an isomorphism of rings $A^{\top} \to A^{*,*}_0$, and the $H^{*,*}$-linear extension of $\theta$
    $$\id_{H^{*,*}}\otimes_{\mathbb{Z}/\ell}\theta: H^{*,*}\otimes_{\mathbb{Z}/\ell}A^{\top}\to A^{*,*}$$
    is an isomorphism of left $H^{*,*}$-modules. We similarly define $M_{B0}$ as the $\mathbb{Z}/\ell$-subalgebra of $M_B$ generated by the $P^i$.
    
    The action of $A^{*,*}$ on $H^{*,*}(\MSp)$ restricts to an action of $A^{*,*}_0$ on $H^{0,0}[b_1,b_2,\ldots]=\mathbb{Z}/\ell[b_1,b_2,\ldots]$ where $\beta$ acts trivially. We denote by $H_0^{*,*}(\MSp)$ the bigraded ring $\mathbb{Z}/\ell[b_1,b_2,\ldots]$. The map $\Phi_{\omega}$ of $A^{*,*}$-modules induces then a map $\Phi_{\omega}^0$ of $A_0^{*,*}$-modules, and we finally get a map
    $$\Phi^0 \coloneqq \bigoplus_{\omega \in P} \Phi_{\omega}^0:\bigoplus_{\omega \in P}M_{B0}u_{\omega}\to H^{*,*}_0(\MSp).$$
    
    Let us now note that, as left $H^{*,*}$-modules, $M_B$ is isomorphic to $H^{*,*}\otimes_{\mathbb{Z}/\ell} M_{B0}$, and $H^{*,*}(\MSp)$ is isomorphic to $H^{*,*}\otimes_{\mathbb{Z}/\ell}H^{*,*}_0(\MSp)$. In particular, we have $\Phi_{\omega}=\id_{H^{*,*}}\otimes \Phi_{\omega}^0$ and $\Phi=\id_{H^{*,*}}\otimes \Phi^0$ as homomorphisms of $\mathbb{Z}/\ell$-vector spaces. Thus, if $\Phi^0$ is an isomorphism of $A^{*,*}_0$-modules, $\Phi$ will be an isomorphism of $\mathbb{Z}/\ell$-vector spaces, and then also an isomorphism of left $A^{*,*}$-modules. It is then enough to prove that $\Phi^0$ is an isomorphism of $A^{*,*}_0$-modules.
    
    Let $\rho:H^*(\MSp^{\top})\to H^{*,*}_0(\MSp)$ be the map defined by taking $b_n^{\top}$ to $b_n$. It is immediate that $\rho$ is an isomorphism of graded rings. We have the commutative diagram 
    $$
    \begin{tikzcd}
        \bigoplus_{\omega \in P}M_B^{\top}u_{\omega}^{\top} \arrow[r, "\Phi^{\top}"] \arrow[d, swap, "\theta'"] & H^*(\MSp^{\top}) \arrow[d, "\rho"]\\
        \bigoplus_{\omega \in P}M_{B0}u_{\omega} \arrow[r, "\Phi^0"] & H^{**}_0(\MSp),
    \end{tikzcd}
    $$
    where $\theta'$ is the map induced by $\theta$. If $\rho$ is also an isomorphism of left $A^{*,*}_0$-modules (using $A^{*,*}_0\cong A^{\top}$ to define the module structure on $H^*(\MSp^{\top})$), all maps in the diagram except $\Phi^0$ are isomorphisms of left $A^{*,*}_0$-modules, then, so is $\Phi^0$. We now prove that that $\rho$ is an isomorphism of left $A^{*,*}_0$-modules.
    
    Since the action of $\beta$ is trivial, it is enough to show that $\rho$ preserves the action of the power operations, that is, $\rho(P^i_{\top}(b_n^{\top}))=P^i(b_n)$ for all $i$.
    
    For each positive integer $r$, the action of $P^i$ on $\mathbb{Z}/\ell[b_1,\ldots,b_r]\subset H^{*,*}_0(\MSp)$ is induced by the action on $H^{*,*}(\HGr(r,n))$ as follows. As in the topological case, the canonical map $\Sigma_{\P^1}^{-2r}\Sigma^\infty_{\P^1}\MSp_{2r}\to \MSp$ induces an isomorphism $H^{a,b}(\MSp)\cong H^{a+4r, b+2r}(\MSp_r)$ for $r\gg0$, namely for $r$ bigger than a certain positive integer $r_0$ depending on $(a,b)$. By Proposition~\ref{prop:CohBSpMSp}(2), the restriction by the zero section $z_{2r}^*:H^{*+4r, *+2r}
    (\MSp_{2r})\to H^{*+4r, *+2r}
    (\BSp_r)$ is injective, with image isomorphic to the ideal generated by the mod-$\ell$ Borel class $b_r(E^\Sp_{2r})$. Similarly, we have an isomorphism $\BSp_{2r}\simeq \colim_n\HGr(r, n)$, in $\sH(k)$, giving the approximation of $H^{*+4r, *+2r}
    (\BSp_r)$ by $H^{*+4r, *+2r}(\HGr(r, n))$ for $n\gg0$, and we can identify $H^{a,b}(\MSp)$ with the bidegree $(a+4r, a+2r)$ component of the ideal in $H^{*, *}(\HGr(r, n))$ generated by the mod $\ell$ Borel class $b_r(E^\Sp_{r,n})$, for $r,n\gg0$. As the motivic Steenrod operations are $\Sigma^{a,b}$-stable, it follows that this identification of $H^{a,b}(\MSp)$ with $H^{a+4r, b+2r}(\HGr(r, n))$ for $r,n\gg0$ is compatible with the respective $A^{*,*}$-module structures. 

    Thus, $P^i(b_j)\in H^{4(i+j), 2(i+j)}(\MSp)$ corresponds to $$P^i(b_j(E^\Sp_{r,n})\cdot b_r(E^\Sp_{r,n}))\in H^{4(i+j+r), 2(i+j+r)}(\HGr(r, n)), \; \; r,n\gg0.$$ 
    
    Let $\text{HFlag}(1^r;n)$ be the complete quaternionic flag variety associated to $\HGr(r,n)$ as in \cite[\S 3, p. 147]{panwal:grass}. By the quaternionic splitting principle \cite[Theorem 10.1]{panwal:grass}, the pullback of $E^\Sp_{2r,2n}$ by the map $\text{HFlag}(1^r;n)\to \HGr(r,n)$ splits as the orthogonal direct sum of the $r$ universal rank $2$ symplectic subbundles $(\mathcal{U}_2^{(1)},\phi_2^{(1)}) \perp \ldots \perp (\mathcal{U}_2^{(r)},\phi_2^{(r)})$, and the cohomology pullback maps injectively the Borel classes $b_j(E^\Sp_{2r,2n})$ to the elementary symmetric polynomials $e_j(\xi^\Sp_1,\ldots,\xi^\Sp_r)$ in the Borel roots $\xi^\Sp_m\coloneqq b_1(\mathcal{U}_2^{(m)},\phi_2^{(m)}) \in H^{4,2}(\text{HFlag}(1^r;n))$. Then $P^i(b_j(E^\Sp_{r,n})\cdot b_r(E^\Sp_{r,n}))$ is determined by $P^i((e_j \cdot e_r)(\xi^\Sp_1,\ldots,\xi^\Sp_r))$. Moreover, by the Cartan formula \cite[Proposition 9.7]{Voev-power}, $P^i((e_j \cdot e_r)(\xi^\Sp_1,\ldots,\xi^\Sp_r))$ is completely determined by the values $P^s(\xi^\Sp_m)$ for $s=0,1,\ldots,i$, but those values are in turn completely determined by the properties of reduced power operations, as follows.
    \begin{itemize}
    \item Since $P^0=\id$ by Lemma \ref{lemma:someproperties} (1), $P^0(\xi^\Sp_m)=\xi^\Sp_m$. 
    \item $P^1(\xi^\Sp_m)=0$ by looking at the degree and at the presentation of $H^{**}(\text{HFlag}(1^r;n))$ given by \cite[Theorem 11.1]{panwal:grass}.
    \item Since $\xi^\Sp_m$ has bidegree $(2\cdot 2,2)$, $P^2(\xi^\Sp_m)=(\xi^\Sp_m)^{\ell}$ by Lemma \ref{lemma:someproperties} (2). 
    \item $P^s(\xi^\Sp_m)=0$ for all $s \ge 3$ by Lemma \ref{lemma:someproperties} (3).
    \end{itemize}

    The analogous result works for the topological case. Namely, the same argument shows that $P^i_{\top}(b_j^{\top})$ is determined by the values $P^s_{\top}(\xi_m^{\top})$ for $s=0,1,\ldots,i$, with $\xi_m^{\top}$ the topological Borel roots of the complex symplectic flag manifold. For those values, the same expressions as before are valid, by the axioms of the power operations $P^i_{\top}$ (see for instance \cite[Chapter VI]{steenrod:cohomology}). 

    Thus, if we express $P^i_{\top}(b_n^{\top})$ as a polynomial $p$ in the variables $e^\top_j=e_j(\xi_1^{\top}, \ldots, \xi_r^{\top})$, the same polynomial $p$ in the variables $e_j=e_j(\xi^\Sp_1,\ldots, \xi^\Sp_r)$ will give an expression for $P^i(b_n)$, and thus:
    $$\rho(P^i_{\top}(b_n^{\top}))=\rho(p(b_1^{\top},b_2^{\top},\ldots))=p(b_1,b_2,\ldots)=P^i(b_n).$$
    \end{proof}
    
    Because of the decomposition of $H^{*,*}(\MSp)$ given by Lemma \ref{lemma:decompmsp}, the algebra $\Ext_{A^{*,*}}(H^{*,*}(\MSp),H^{*,*})$ is generated by the algebra $\Ext_{A^{*,*}}(M_B,H^{*,*})$ and the dual elements of the cohomology classes $u_{\omega}$. The algebra $\Ext_{A^{*,*}}(M_B,H^{*,*})$ is completely studied in \cite[Section 5.1]{lev:ellcoh}. In particular, we have the following computation.
  
  \begin{lemma}\label{lem:ExtMB}
 1. \cite[Lemma 5.1 and Lemma 5.4]{lev:ellcoh} There is an isomorphism of trigraded $\mathbb{Z}/\ell$-algebras
    $$\Ext_{A^{*,*}}(M_B,H^{*,*}) \simeq \Ext_B(\mathbb{Z}/\ell,H^{*,*}).$$
 2. \cite[Lemma 5.2]{lev:ellcoh} We have  
    \begin{itemize}
        \item $\underline{t>2u}$: $\Ext_B^{s,(t-s,u)}(\mathbb{Z}/\ell,H^{*,*})=0$.
        \item $\underline{t=2u}$: $\Ext_B^{s,(2u-s,u)}(\mathbb{Z}/\ell,H^{*,*})$ is a polynomial $\mathbb{Z}/\ell$-algebra in generators $\{h_r'\}_{r\ge0}$, with $\deg(h_r')=(1,(1-2\ell^r,1-\ell^r))$. 
        \item $\underline{t<2u}$: $\Ext_B^{0,(2u-1,u)}(\mathbb{Z}/\ell,H^{*,*})$ is $H^{1,1}$ if $u=1$, and $0$ otherwise, and the product map
        $$H^{1,1} \otimes \bigoplus_{s,u}\Ext_B^{s,(2u-s,u)}(\mathbb{Z}/\ell,H^{*,*}) \to \Ext_B^{s,(2u-s+1,u+1)}(\mathbb{Z}/\ell,H^{*,*})$$
        is surjective.
    \end{itemize}
    \end{lemma}
    
    Now, let us define, for all even non-$\ell$-adic partitions $\omega$, elements $z_{(\omega)} \in \Ext_{A^{*,*}}^{0,(-2|\omega|,-|\omega|)}$ characterized by the relations $(z_{\omega},u_{\omega'})= \delta_{\omega,\omega'}$, where the pairing $(z_{\omega},u_{\omega'})$ is $z_{\omega}(u_{\omega'})\in H^{*,*}$. Let us note that the diagonal map $\delta^\top$ described above can be read in $\bigoplus_{\omega \in P}M_B^{\top}u_{\omega}^{\top}$ through $\Phi^{\top}$, and it induces, via $\theta'$ and $\rho$, the diagonal map $\delta$ of $\bigoplus_{\omega \in P}M_Bu_{\omega}$. The expression \eqref{eq:comultiplication} for $\delta^{\top}(u_{\omega}^{\top})$ yields the analogous expression for $\delta(u_{\omega})$ through $\rho$. Then, via the multiplication map on $ \Ext_{A^{*,*}}(H^{*,*}(\MSp),H^{*,*})$ induced by $\rho$, we see that the pairing $(z_{\omega_1}\cdot z_{\omega_2},u_{\omega})$ can be computed by taking $(z_{\omega_1}\otimes z_{\omega_2},\delta(u_{\omega}))$ and turning $\otimes$ into products in $H^{*,*}$.
    
    Now we write $(\omega_1,\omega_2)$ for the concatenation: namely $((\omega_1,\ldots, \omega_r),(\omega'_1,\ldots, \omega'_s))$ is the sequence $(\omega_1,\ldots, \omega_r,\omega'_1,\ldots, \omega'_s)$, reordered to be non-increasing. It is then easy to see that $(z_{\omega_1}\otimes z_{\omega_2},\delta(u_{\omega}))$ is $1$ if $\omega=(\omega_1,\omega_2)$, and $0$ otherwise, which gives $z_{\omega_1}\cdot z_{\omega_2}=z_{(\omega_1,\omega_2)}$.
    
    Lemma~\ref{lemma:decompmsp} and Lemma~\ref{lem:ExtMB}  implies that the $H^{*,*}$-subalgebra
 $$\oplus_{k\ge0} \Ext_{A^{*,*}}^{0,(-4k,-2k)}(H^{*,*}(\MSp),H^{*,*})$$
 of $ \Ext_{A^{*,*}}^{*,(*,*)}(H^{*,*}(\MSp),H^{*,*})$ is the free $H^{*,*}$-module on the elements $z_\omega$, $\omega\in P$, the only indecomposable elements are the elements $z_{(2k)}\in \Ext_{A^{*,*}}^{0,(-4k,-2k)}(H^{*,*}(\MSp),H^{*,*})$ corresponding to the trivial partitions $\omega=(2k)$, with $2k$ not of the form $\ell^i-1$, and the subalgebra generated by the 
 $z_{(2k)}$ is the polynomial algebra over $H^{*,*}$ on the $z_{(2k)}$.
 
 Taking $k=0$, the element $z_{(0)}\in \Ext_{A^{*,*}}^{0,(0,0)}(H^{*,*}(\MSp),H^{*,*})$ dual to $u_{(0)}$ is the identity in the $\Ext$-algebra,  and the summand $\bigoplus_{s,u}\Ext_{A^{*,*}}^{s,(2u-s,s)}(M_Bu_{(0)},H^{*,*})$
of $\oplus_{s,u}\Ext_{A^{*,*}}^{s,(2u-s,s)}(H^{*,*}(\MSp),H^{*,*})$ is the polynomial algebra on generators $\{h_r'\}_{r\ge0}$ corresponding to the polynomial generators $\{h_r'\}_{r\ge0}$ described in Lemma~\ref{lem:ExtMB}. Moreover, we have $h_{r_1}'z_{\omega_1}\cdot h_{r_2}'z_{\omega_2}=h_{r_1}'h_{r_2}'z_{(\omega_1,\omega_2)}$, and for each partition $\omega\in P$, the assignment $x \mapsto x\cdot z_\omega$ gives an isomorphism
$$\bigoplus_{s,u}\Ext_{A^{*,*}}^{s,(2u-s,s)}(M_Bu_{(0)},H^{*,*}) \xrightarrow{\sim}\bigoplus_{s,u}\Ext_{A^{*,*}}^{s,(2u-s,s)}(M_Bu_{\omega},H^{*,*}).$$

Putting all the information from Lemma~\ref{lemma:decompmsp} and Lemma~\ref{lem:ExtMB} together with the product structure on $\{z_\omega\}$ described above, we get the following.

    \begin{prop}
    \label{prop:pres.ExtAlgebra}
        The $\mathbb{Z}/\ell$-algebra $\Ext_{A^{*,*}}(H^{*,*}(\MSp),H^{*,*})$ has the following presentation:
        \begin{itemize}
            \item $\Ext_{A^{*,*}}^{s,(t-s,u)}(H^{*,*}(\MSp),H^{*,*})=0$ if $t>2u$.
            \item $\Ext_{A^{*,*}}^{s,(2u-s,u)}(H^{*,*}(\MSp),H^{*,*})$ is polynomial in the generators:
            \begin{itemize}
                \item $1 \in \Ext_{A^{*,*}}^{0,(0,0)}(H^{*,*}(\MSp),H^{*,*})$;
                \item $z_{(2k)}\in \Ext_{A^{*,*}}^{0,(-4k,-2k)}(H^{*,*}(\MSp),H^{*,*})$, for $k\ge 1$, and $2k\neq \ell^i-1 \; \forall \; i \ge 0 $;
                \item $h_r' \in \Ext_{A^{*,*}}^{1,(1-2 \ell^r,1-\ell^r)}(H^{*,*}(\MSp),H^{*,*})$, for $r \ge 0$.
            \end{itemize}
            \item $\Ext_{A^{*,*}}^{0,(2u-1,u)}(H^{*,*}(\MSp),H^{**})$ is $H^{1,1}$ if $u=1$, and $0$ otherwise, and the product map
        \begin{multline*}
            H^{1,1} \otimes_{\Z/\ell} \bigoplus_{s,u}\Ext_{A^{*,*}}^{s,(2u-s,u)}(H^{*,*}(\MSp),H^{*,*}) = \\ \Ext_{A^{*,*}}^{0,(1,1)}(H^{*,*}(\MSp),H^{*,*}) \otimes_{\Z/\ell} \bigoplus_{s,u}\Ext_{A^{*,*}}^{s,(2u-s,u)}(H^{*,*}(\MSp),H^{*,*}) \to \\ \bigoplus_{s,u} \Ext_{A^{*,*}}^{s,(2u-s+1,u+1)}(H^{*,*}(\MSp),H^{*,*})
        \end{multline*}
        is surjective.
        \end{itemize}
    \end{prop}
    
    \begin{rmk} Comparing the descriptions of $H^{*,*}(\MSp)$ and $\Ext_{A^{*,*}}(H^{*,*}(\MSp),H^{*,*})$ given by Lemma \ref{lemma:decompmsp} and Proposition \ref{prop:pres.ExtAlgebra} with their analogues for $\MGL$ given by \cite[Lemma 5.9 and Proposition 5.7]{lev:ellcoh}, we see that $H^{*,*}(\MSp)$ is a summand of $H^{*,*}(\MGL)$ as $A^{*,*}$-modules, $\Ext_{A^{*,*}}(H^{*,*}(\MSp),H^{*,*})$ is the corresponding $H^{*,*}$-summand of  $\Ext_{A^{*,*}}(H^{*,*}(\MGL),H^{*,*})$,  and $\Ext_{A^{*,*}}(H^{*,*}(\MGL),H^{*,*})$ is the polynomial algebra over $\Ext_{A^{*,*}}(H^{*,*}(\MSp),H^{*,*})$ on generators $z_{(k)}\in \Ext_{A^{*,*}}^{0,(-2k,-k)}(H^{*,*}(\MGL),H^{*,*})$ for $k$ odd. All these comparison maps are induced by the map $\Phi: \MSp\to \MGL$ of Remark \ref{rmk:MSp-MGLThomClasses}. 
   \end{rmk}
   
   \subsection{The Adams spectral sequence}
   
   We denote by $\overline{H\Z/\ell}$ the homotopy cofiber of the unit map $\epsilon_{H\Z/\ell}:1_k \to H\Z/\ell$.  For $\E \in \SH(k)$, letting $\E_s\coloneqq \overline{H\Z/\ell}^{\wedge s} \wedge \E$ and $W_s(\E) \coloneqq H\Z/\ell \wedge \E_s$ gives the standard \emph{Adams tower} of $\E$ $\{\E_s,W_s(\E)\}_{s \ge 0}$ (see for example \cite[Section 6.1]{lev:ellcoh}).
   
   \begin{defn}
    \begin{enumerate}
        \item The \emph{$H\Z/\ell$-nilpotent completion} of $\E$, denoted by $\E^\wedge_{H\Z/\ell}$, is the homotopy limit of the Adams tower $\{\E_s,W_s(\E)\}_{s \ge 0}$.
        \item The (Mod-$\ell$) \emph{Adams spectral sequence for $\E$} is the spectral sequence associated to the Adams tower of $\E$. Explicitly, it is given by
    $$E_1^{s,t,u} \coloneqq W_s(\E)^{t,u} \Rightarrow (\E_{H\Z/\ell}^\wedge)^{t,u},$$
    with filtration degree $s$, cohomological bidegree $(t,u)$, and differentials $d_r^{s,t,u}:E_r^{s,t,u} \to E_r^{s+r,t+1,u}$.
    \end{enumerate}
   \end{defn}

We need to introduce the following technical condition.

\begin{defn}
\label{defn:WedgeOfCopiesHZ}
    For $\E \in \SH(k)$, we say that $\E$ is a \emph{motivically finite type wedge of copies of $H\Z/\ell$} if there exist $p,q \in \Z$ and a set of bidegrees $\{p_\alpha,q_\alpha\}{_\alpha\in S}$ such that 
    $$\Sigma^{p,q}\E \simeq \oplus_{\alpha \in S} \Sigma^{p_\alpha,q_\alpha} H\Z/\ell,$$
    with the bidegrees $(p_\alpha,q_\alpha)$ satisfying the conditions:
    \begin{enumerate}
        \item $p_\alpha \ge 2q_\alpha \ge 0$ for all $\alpha \in S$,
        \item For all $m \in \Z$, $q_\alpha \le m$ for almost all $\alpha$.
    \end{enumerate}
\end{defn}

\begin{prop}[\cite{DugIsa:Adams}, Remark 6.11 and Proposition 6.1]
\label{prop:a.s.s}
    Let us suppose that $\E \in \SH(k)$ is a cellular spectrum and $W_s(\E)$ is a motivically finite type wedge of copies of $H\Z/\ell$ for all $s$. Then the $E_2$-page of the Adams spectral sequence for $\E$ has the form
    $$E_2^{s,t,u}= \Ext_{A^{*,*}}^{s,(t-s,u)}(H^{*,*}(\E),H^{*,*}).$$
    If in addition, fixed $(t,u)$, we have $\lim_r^1 E_r^{s,t,u}(\E)=0$ for all $s$, the spectral sequence converges completely to $(\E_{H\Z/\ell}^\wedge)^{t,u}$. That is, the natural map $(\E_{H\Z/\ell}^\wedge)^{t,u} \to \lim_s K_s(\E)^{t,u}$ is an isomorphism, the associated filtration
    $$F_s(\E_{H\Z/\ell}^\wedge)^{t,u} \coloneqq \ker((\E_{H\Z/\ell}^\wedge)^{t,u} \to K_{s-1}(\E)^{t,u}) \subset (\E_{H\Z/\ell}^\wedge)^{t,u}$$
    is exhaustive, and the natural map
    $$F_s(\E_{H\Z/\ell}^\wedge)^{t,u} / F_{s+1}(\E_{H\Z/\ell}^\wedge)^{t,u} \to E_\infty^{s,t,u}(\E)$$
    is an isomorphism for all $s \ge 0$.
\end{prop}

See \cite[Proposition 6.6]{lev:ellcoh} for a proof (instead of ordinary cellularity they use the weaker hypothesis of $H\Z/\ell$-cellularity, but we don't need it).

\begin{prop}
\label{prop:a.s.s.multiplication}
    Let $\sE$ be a cellular ring spectrum with multiplication $\mu_{\sE}$, such that each $W_s(\sE)$ is a motivically finite type wedge of copies of $H\Z/\ell$. Let us also suppose that 
    $H^{*,*}(\sE)$ is flat over $H^{*,*}$ and the canonical map $H^{*,*}(\sE)\otimes_{H^{*,*}}H^{*,*}(\sE) \to H^{*,*}(\sE \wedge \sE)$ is an isomorphism. Then, the Adams spectral sequence for $\sE$ has a multiplicative structure induced by $\mu_\sE$ and compatible with the multiplicative structure on $(\sE_{H\Z/\ell}^\wedge)^{*,*}$. Also, the product on the $E_2$-terms induced by the multiplicative structure on the spectral sequence agrees with the associative $\Z/\ell$-algebra structure on $\Ext_{A^{*,*}}^{s,(t-s,u)}(H^{*,*}(\sE),H^{*,*})$ given as in \eqref{eq:tensor_product_ext} and the isomorphism 
$$E_2^{s,t,u}= \Ext_{A^{*,*}}^{s,(t-s,u)}(H^{*,*}(\sE),H^{*,*})$$    
given by Proposition \ref{prop:a.s.s}.
\end{prop}

See \cite[Proposition 6.7]{lev:ellcoh} for a proof.

Let us now recall that we have the unstable Hopf map $\eta_u: \A^2 \setminus \{0\} \to \P^1$ defined by $(x,y) \mapsto [x:y]$. We also have $\A^2 \setminus \{0\}\simeq \P^1 \wedge \G_m$ in $\sH_\bullet(k)$ by \cite[\S 3 Example 2.20]{morvoe:homotopytheory}. Then we consider the \emph{stable Hopf map} as the map $\eta:=\Sigma^{-2,-1}\Sigma^\infty_{\P^1}\eta_u: \Sigma^{1,1}1_k \to 1_k$.

\begin{defn}
     We denote by $\mathbb{S}_k/\eta^n$ the homotopy cofiber of $\eta^n:\Sigma^{n,n}1_k \to 1_k$. The  diagram
    $$
    \begin{tikzcd}
       \Sigma^{n+1,n+1}1_k \arrow[d, swap, "\eta"] \arrow[r, "\eta^{n+1}"] & 1_k \arrow[d, equal] \arrow[r] & 1_k/\eta^{n+1} \\
       \Sigma^{n,n}1_k \arrow[r, swap, "\eta^{n}"] & 1_k \arrow[r] & 1_k/\eta^n
    \end{tikzcd}
    $$
    induces the map $1_k/\eta^{n+1} \to 1_k/\eta^n$ on the cofibers. For $\E \in \SH(k)$ we then get a tower $\{1/\eta^n \wedge \E\}_{n>0}$ over $1/\eta \wedge \E$. We call the \emph{$\eta$-completion of $\E$}, denoted by $\E_\eta^\wedge$, the homotopy limit of this tower. 
\end{defn}

\begin{defn}
\label{defn:ellCompl}
For $m \in \Z$, we recall that the Moore spectrum $S\Z/m$ is the cofiber of the map $\mathbb{S}_k \xrightarrow{\times m} \mathbb{S}_k$ in $\SH(k)$. For $\E \in \SH(k)$, the \emph{$\ell$-adic completion of $\E$}, denoted by $\E_\ell^\wedge$, is the homotopy limit of the tower $\{S\Z/\ell^n \wedge \E\}_{n \ge 0}$ over $S\Z \wedge \E$.
\end{defn}

As a matter of notation, we will use the shorthand $\E^\wedge_{\eta,\ell} \coloneqq (\E_\eta^\wedge)_\ell^\wedge$ for the $(\eta, \ell)$-completion of $\E$.

In \cite{Man:completions}, Mantovani proved a comparison between the $H\Z/\ell$-nilpotent completion and the $(\eta,\ell)$-completion of a spectrum $\E$, with the only hypothesis of connectedness of $\E$, generalizing some results showed in \cite{ormsby:adams} (in particular the first claim of \cite[Theorem 1]{ormsby:adams}). In order to define connectedness, let us recall that for any motivic spectrum $\E$, we have a bigraded homotopy sheaf $\pi_{n,m}(\E)$ given by the Nisnevich sheafification of the presheaf $U \mapsto \E^{-n,-m}(U)$ on $\Sm/k$. One then uses the notations $\pi_p(\E)_q \coloneqq \pi_{p-q,-q}(\E)$ and $\pi_p(\E) \coloneqq \oplus_{q \in \Z} \pi_p(\E)_q$. 

\begin{defn}
    A motivic spectrum $\E \in \SH(k)$ is called \emph{$r$-connected} if for all $p \le r$, $\pi_p(\E)=0$, and it is called \emph{connected} if it is $r$-connected for some integer $r$.
\end{defn}

Combining some of the main results of \cite{Man:completions}, one gets the following.

\begin{thm}[\cite{Man:completions}, Theorem 4.3.5, Example 5.2, Theorem 7.3.4]
\label{thm:Mantovani}
    If $\E \in \SH(k)$ is a connected motivic spectrum, then $\E_{H\Z/\ell}^\wedge \simeq \E_{\eta, \ell}^\wedge$.
\end{thm}

\begin{prop}
\label{prop:mspconnective}
    $\MSp$ is a $(-1)$-connected spectrum.
\end{prop}
    
\begin{proof}
The argument is basically the same as in \cite[Corollary 2.3]{lev:ellcoh}, replacing $\Gr(n,n+m)$ with $\HGr(n,np)$, the tautological bundle over $\Gr(n,n+m)$ with $E_{n,np}^\Sp$ and $\Sigma^{-2n-n}$-suspensions with $\Sigma^{-4n-2n}$-suspensions. The symplectic structure on $E_{n,np}^\Sp$ does not play any role here.
\end{proof}

Theorem \ref{thm:Mantovani} and Proposition \ref{prop:mspconnective} give the following.

\begin{corollary}
\label{cor:EtaEllComplMSp}
    In $\SH(k)$, we have $\MSp_{H\Z/\ell}^\wedge \simeq \MSp_{\eta, \ell}^\wedge$.
\end{corollary} 

We can finally compute the Adams spectral sequence for $\MSp$.

\begin{prop}
\label{prop: a.s.s.MSp}
    The Adams spectral sequence for $\MSp$ is of the form
    $$E_2^{s,t,u}= \Ext_{A^{*,*}}^{s,(t-s,u)}(H^{*,*}(\MSp),H^{*,*}) \Rightarrow (\MSp_{H\Z /\ell}^\wedge)^{t,u},$$
    with differentials $d_r^{s,t,u}:E_r^{s,t,u} \to E_r^{s+r,t+1,u}$.
\end{prop}

\begin{proof} We need to verify the two conditions of Proposition \ref{prop:a.s.s} for $\MSp$. 

    The fact that $\MSp$ is cellular is Corollary \ref{cor:cellularityofMSp}. Therefore, we just need to show that $W_s(\MSp)$ is a motivically finite type wedge of copies of $H\Z/\ell$ for all $s$.

    Letting $W_s\coloneqq H\Z/\ell \wedge \overline{H\Z/\ell}^{\wedge s}$, we have $W_s(\MSp)=W_s \wedge \MSp = W_s \wedge_{H\Z/\ell} H\Z/\ell \wedge \MSp$. By Proposition \ref{prop:MotiveOfMSp} and Remark \ref{rmk:HZ/ell}, we can write $$H\Z/\ell \wedge \MSp \simeq \oplus_{\alpha \ge 0} \Sigma^{4n_\alpha,2n_\alpha}H\Z/\ell,$$
    for some non-negative integers $n_\alpha$ such that, for all $n$, there are only finitely many indices $\alpha$ with $n_\alpha=n$. Moreover, by \cite[Lemma 6.8]{lev:ellcoh}, we can write
    $$W_s = \oplus_{(p,q) \in S_s}\Sigma^{p,q}(H\Z/\ell)^{r_{p,q}},$$
    where $S_s = \{(p,q) \mid p+s \ge 2q \ge0 \}$, and for each $q$, $r_{p,q}=0$ for almost all $p$. We note that $(4n_\alpha,2n_\alpha)\in S_s$, and we conclude that we can write
    $$W_s(\MSp) \simeq \oplus_{(p,q) \in S_s}\Sigma^{p,q}(H\Z/\ell)^{n_{p,q}},$$
    where, again, $S_s = \{(p,q) \mid p+s \ge 2q \ge0 \}$, and for each $q$, $n_{p,q}=0$ for almost all $p$. In particular, $W_s(\MSp)$ is a motivically finite type wedge of copies of $H\Z/\ell$ for all $s$.
\end{proof}

\begin{prop}
\label{prop:a.s.s.MSp2}
    In the Adams spectral sequence for $\MSp$, if we let $E_r^{t,u}\coloneqq \oplus_sE_r^{s,t,u}$, we have $E_2^{2u,u}\simeq E_\infty^{2u,u}$, and the spectral sequence converges completely to $(\MSp_{\eta,\ell}^\wedge)^{2u,u}$.
\end{prop}

\begin{proof}
    For all $s,t,u$, we have $E_2^{s,t,u}= \Ext_{A^{*,*}}^{s,(t-s,u)}(H^{*,*}(\MSp),H^{*,*})$ by Proposition \ref{prop: a.s.s.MSp}. In particular, from Proposition \ref{prop:pres.ExtAlgebra}(1), we get $E_2^{2u+1,u}=0$. Hence, $d_r^{2u,u}$ vanishes for all $r \ge 2$. By Proposition \ref{prop:pres.ExtAlgebra}(2), $E_2^{r,2,1}=0$ for all $r \ge 2$. Therefore, $d_r^{0,1,1}$ will vanish for all $r \ge 2$. Finally, by Proposition \ref{prop:pres.ExtAlgebra}(3), $E_2^{0,1,1}=H^{1,1}$.

    By Proposition \ref{prop:a.s.s.multiplication}, the product on $E_2$ given by the multiplicative structure on the spectral sequence is the product of the $\Z/\ell$-algebra $\Ext_{A^{*,*}}(H^{*,*}(\MSp),H^{*,*})$ of Proposition \ref{prop:pres.ExtAlgebra}. In particular, the product
    $$E_2^{0,1,1}\otimes_{\Z/\ell} \bigoplus_s E_2^{s,2u-2,u-1} \to \bigoplus_s E_2^{s,2u-1,u}$$
    is surjective. Thus, from the vanishing of differentials $d_r^{0,1,1}$ and $d_r^{2u,u}$, we obtain the vanishing of differentials $d_r^{2u-1,u}$.

    Now, since the differentials $d_r^{2u,u}$ and $d_r^{2u-1,u}$ are zero for $r\ge2$, we have that $E_{r+1}^{2u,u}=E_r^{2u,u}$ for $r \ge 2$, from which it follows that $E_2^{2u,u}=E_\infty^{2u,u}$ and $\lim_r^1E_r^{s,2u,u}=0$. Hence, by Proposition \ref{prop:a.s.s}, the spectral sequence converges completely to $(\MSp_{H\Z/\ell}^\wedge)^{2u,u}$. Finally, by Corollary \ref{cor:EtaEllComplMSp}, we get $(\MSp_{H\Z/\ell}^\wedge)^{2u,u} \cong (\MSp_{\eta,\ell}^\wedge)^{2u,u}$, which concludes the proof.
\end{proof}

\section{Constructing symplectic bordism classes}
\label{section:Generators}

Proposition \ref{prop:K_0SpExtensionTh(-)} says that the symplectic Thom class construction descends to $K^\Sp_0(X)$. Thus, for $(v,\omega)\in K^\Sp_0(X)$ a virtual symplectic vector bundle on $X$ with virtual rank $2r$, we have the isomorphism
\[
\th^{\Sp,f}_{\MSp}(v,\omega): \Sigma^{-4r,-2r}\Sigma^vp_X^*\MSp\xrightarrow{\sim} p_X^*\MSp.
\]

We now want to define a symplectic Thom isomorphism for $\MSp$-cohomology.

\begin{defn}[Symplectic Thom Isomorphism] \label{def:symplThomisos} Let $X\in \Sm/k$ and $(v,\omega)\in K_0^\Sp(X)$ of virtual rank $2r$. We   define a natural isomorphism
\[
\Th_\MSp(v,\omega):\MSp^{a,b}(X)\xrightarrow{\sim} \MSp^{4r+a, 2r+b}(p_{X\#}\Sigma^v1_X)
\]
as the composition
\begin{multline*}
\MSp^{a,b}(X)=[p_\#1_X,\Sigma^{a,b}\MSp]_{\SH(X)}\simeq [1_X, p_X^*\Sigma^{a,b}\MSp]_{\SH(X)} \xrightarrow[\sim]{\Sigma^{4r+a, 2r+b}\Sigma^{-v}\th^{\Sp,f}_{\MSp}(v,\omega)_*} \\
[1_X, \Sigma^{-v}p_X^*\Sigma^{4r+a,2r+b}\MSp]_{\SH(X)}
\xrightarrow[\sim]{(1)} [\Sigma^v1_X, p_X^*\Sigma^{4r+a,2r+b}\MSp]_{\SH(X)} \\ \xrightarrow[\sim]{(2)}
[p_{X\#}\Sigma^v1_X, \Sigma^{4r+a,2r+b}\MSp]_{\SH(k)} =\MSp^{4r+a,2r+b}(p_{X\#}\Sigma^v1_X),
\end{multline*}
where $(1)$ is given by applying $\Sigma^v:\SH(X)\to \SH(X)$ and the $(2)$ is given by the adjunction $p_{X\#}\dashv p_X^*$.
\end{defn}

\begin{defn}\label{def:generalsymplthomclasses} For $X\in \Sm/k$ and $(v,\omega)\in K_0^\Sp(X)$ of virtual rank $2r$, we define the $\MSp$-cohomology class
\[
[(v,\omega)]_\MSp :=\Th_\MSp(v,\omega)(1_{\MSp^{0,0}(X)}) \in \MSp^{4r, 2r}(p_{X\#}\Sigma^v1_X).
\]
\end{defn}

In particular, for two symplectic vector bundles $(V,\omega)$, $(V',\omega')$ on $X$ of respective ranks $2r, 2r'$, we have the Thom isomorphism
\[
\Th_\MSp((V,\omega)-(V',\omega')):\MSp^{a,b}(X)\xrightarrow{\sim}
\MSp^{4(r-r')+a, 2(r-r')+b}(p_{X\#}\Sigma^{V-V'}1_X),
\]
and the $\MSp$ class $[(V,\omega)-(V',\omega')]_\MSp\in \MSp^{4(r-r'), 2(r-r')}(p_{X\#}\Sigma^{V-V'}1_X)$.

\begin{rmk}
\label{rmk:thomclass}
    In Definition \ref{def:generalsymplthomclasses}, if the negative contribution is null, the class $[(V,\omega)]_{\MSp}$ is just the symplectic Thom class $\th_\Sp^{\MSp}(V,\omega)$. The notation $[-]_{\MSp}$ is indeed intended to be a generalization of the symplectic Thom class to virtual symplectic vector bundles.
\end{rmk}

Now, let us consider the graded abelian group $\oplus_{v\in K_0^\Sp(X)}\MSp^{2\rnk(v), \rnk(v)}(p_{X\#}\Sigma^v1_X)$. We make this into a
$K_0^\Sp(X)$-graded ring by giving the product 
\begin{multline*}
\MSp^{2\rnk(v), \rnk(v)}(p_{X\#}\Sigma^v1_X)\times \MSp^{2\rnk(v'), \rnk(v')}(p_{X\#}\Sigma^{v'}1_X)\\\to \MSp^{2\rnk(v+v'), \rnk(v+v')}(p_{X\#}\Sigma^{v+v'}1_X)
\end{multline*}
induced by the product maps
\begin{multline*}
[\Sigma^v1_X,\Sigma^{a,b}p_X^*\MSp]_{\SH(X)}\times
[\Sigma^{v'}1_X,\Sigma^{a',b'}p_X^*\MSp]_{\SH(X)}\\\to
[\Sigma^v1_X\wedge_X\Sigma^{v'}1_X, \Sigma^{a+a',b+b'}p_X^*\MSp\wedge_Xp_X^*\MSp]_{\SH(X)} \to 
[\Sigma^{v+v'}1_X, \Sigma^{a+a',b+b'}p_X^*\MSp]_{\SH(X)},
\end{multline*}
where the last map is induced by the isomorphism $\Sigma^v1_X\wedge_X\Sigma^{v'}1_X \simeq  
\Sigma^{v+v'}1_X$ and the multiplication map $\mu_{p_X^*\MSp}:p_X^*\MSp\wedge_Xp_X^*\MSp\to
p_X^*\MSp$. We denote this product by
\begin{multline*}
\smile:\oplus_{v\in K_0^\Sp(X)}\MSp^{2\rnk(v), \rnk(v)}(p_{X\#}\Sigma^v1_X)\times \oplus_{v\in K_0^\Sp(X)}\MSp^{2\rnk(v), \rnk(v)}(p_{X\#}\Sigma^v1_X)\\
\to \oplus_{v\in K_0^\Sp(X)}\MSp^{2\rnk(v), \rnk(v)}(p_{X\#}\Sigma^v1_X).
\end{multline*}
Note that the unit $1_{\MSp^{0,0}(X)}\in \MSp^{0,0}(X)$ acts as (left and right) unit for this product.

\begin{prop}\label{prop:additivity} Let us take $X\in \Sm/k$. Then sending $(v,\omega)\in K_0^\Sp(X)$ to $[(v,\omega)]_\MSp\in \MSp^{2\rnk(v), \rnk(v)}(p_{X\#}\Sigma^v1_X)$  defines a multiplicative map
\[
[-]_\MSp:K_0^\Sp(X)\to \oplus_{v\in K_0^\Sp(X)}\MSp^{2\rnk(v), \rnk(v)}(p_{X\#}\Sigma^v1_X).
\]
That is,
\[
[(v,\omega)+(v',\omega')]_\MSp=[(v',\omega')]_\MSp\smile [(v',\omega')]_\MSp
\]
and $[0]_\MSp$ is the unit $1_{\MSp^{0,0}(X)}\in \MSp^{0,0}(X)$. 
\end{prop}

\begin{proof}  The multiplicativity follows directly from the multiplicativity of the symplectic Thom classes, as expressed in Lemma~\ref{lem:SympThomMult}.

The class $[0]_\MSp$ is the image of $1_{\MSp^{0,0}(X)}\in \MSp^{0,0}(X)$ under the symplectic Thom isomorphism $\th^{\Sp,f}_{\MSp}(0):p_X^*\MSp\to p_X^*\MSp$, namely the identity map. 
\end{proof}

\subsection{Cobordism classes from stable symplectic twists}

We recall that a \textit{symplectic variety} $(X,\omega_X)$ is a variety $X\in \Sm/k$ with a symplectic form $\omega_X$ on the tangent bundle $T_X$. In particular, this gives a class
$$[-(T_X,\omega_X)]_\MSp \in \MSp^{-4d,-2d}(p_{X\#}\Sigma^{-T_X}1_X),$$
with $2d=\dim_kX$.

If $f:X\to Y$ is a proper map in $\Sm/k$, and $X,Y$ are symplectic, a symplectic version of Construction \ref{defn:DualMap_Properpushforward} (2) gives maps $f_*:\MSp^{a,b}(p_{X!}(1_X)) \to \MSp^{a,b}(p_{Y!}(1_Y))$ through the symplectic Thom isomorphism. In particular, if $X \in \Sm/k$ is proper symplectic of dimension $2d$, with structure map $p_X$, we have a map 
$$p_{X*}:\MSp^{-4d,-2d}(p_{X\#}\Sigma^{-T_X}1_X) \to \MSp^{-4d,-2d}(1_k)=\MSp^{-4d,-2d}(k)=\MSp_{4d,2d}(k),$$
and we can define the class
$$[X,-(T_X,\omega_X)]_{\MSp} \coloneqq p_{X*}[-(T_{X},\omega_X)]_{\MSp} \in \MSp_{4d,2d}(k).$$

\begin{rmk}
\label{rmk:Msp-MGL}
    If $\Phi:\MSp \to \MGL$ is the natural map of Remark \ref{rmk:MSp-MGLThomClasses}, one sees by construction that $\Phi_*[X,-(T_X,\omega_X)]_{\MSp}=[X]_{\MGL} \in \MGL_{4d,2d}(k)$, where $[X]_{\MGL}$ is the algebraic cobordism class of $X$ defined by the $\GL$-orientation, see Definition~\ref{defn:MGLclasses}.
\end{rmk}

The standard non-trivial examples of symplectic varieties in algebraic geometry are not easy to handle, and it is not clear what their characteristic numbers $s_d$ are. In order to construct generators of $(\MSp_\eta^\wedge)^*$, we want to extend our symplectic Pontryagin-Thom construction to varieties with a stable symplectic twist. We then give the symplectic version of Definitions \ref{defn:stwist}, \ref{defn:twistclass}.

\begin{defn}\label{defn:sstwist} Given $Y\in \Sm/k$, a {\em stable symplectic twist of $-T_Y$} is a tuple  $((v, \omega), \vartheta, m)$, where $(v, \omega)\in K_0^\Sp(Y)$ is a virtual symplectic bundle and  $\vartheta$  is an isomorphism $\Sigma^{-T_Y}1_Y\xrightarrow{\sim} \Sigma^{2m,m}\Sigma^{v}1_Y$ in $\SH(Y)$.  
\end{defn}

\begin{defn}\label{defn:SymplecticTwistedClass} Let  $Y\in \Sm/k$ with structure morphism $p_Y:Y\to \Spec k$ proper of dimension $d_Y$  and let  $((v, \omega), \vartheta, m)$ be a stable symplectic twist of $-T_Y$. Let $2r=\rnk(v)$. We write the element $[v,\omega]_\MSp\in \MSp^{4r, 2r}(p_{X\#}\Sigma^v1_Y)$ as a map
\[
[v,\omega]_\MSp:p_{Y\#}\Sigma^v1_Y\to \Sigma^{4r,2r}\MSp
\]
in $\SH(k)$, and we define the class $[Y, (v, \omega), \vartheta, m]_\MSp\in \MSp^{4r+2m,2r+m}(k)$ as the composition
\[
1_k\xrightarrow{p_Y^\vee}p_{Y\#}\Sigma^{-T_Y}1_Y
\xrightarrow{p_{Y\#}\vartheta}
p_{Y\#}\Sigma^{2m,m}\Sigma^{v}1_Y\xrightarrow{\Sigma^{2m,m}[v,\omega]_\MSp}\Sigma^{2m+4r,m+2r}\MSp. 
\]
\end{defn}

In particular, this gives a twisted symplectic analogue of Definition \ref{defn:MGLclasses}. We want to construct symplectic bordism classes by using Ananyevskiy's isomorphism \eqref{eqn:AnanIso}.

\begin{exmp}[Odd dimensional projective spaces]
\label{exmp:symplprojspaces}
Let $X=\mathbb{P}^{2n+1}$ for a non-negative integer $n$. We have the Euler exact sequence
$$0 \to \mathcal{O}_X \to \mathcal{O}_X(1)^{\oplus 2n+2} \to T_X \to 0,$$
giving $\Sigma^{T_X \oplus \mathcal{O}_X}1_X \simeq \Sigma^{\mathcal{O}_X(1)^{2n +2}}1_X$. As in \eqref{eqn:AnanIso}, we get the isomorphism in $\SH(X)$
\begin{multline*}
 \Sigma^{\sO_X(1)^{n+1}\oplus \sO_X(-1)^{n+1}}1_X
 \xrightarrow{\Sigma^{\sO_X(1)^{n+1}}(\Anan_{\sO_X(-1),\ldots,\sO_X(-1)})}\\
 \Sigma^{\sO_X(1)^{n+1}}(\Sigma^{\sO_X(1)^{n+1}}1_X)\simeq
 \Sigma^{T_X\oplus \sO_X}1_X.
 \end{multline*}
 The bundle $\mathcal{O}_X(1)^{n+1}\oplus \mathcal{O}_X(-1)^{n+1}$ is of the form $V \oplus V^{\vee}$, thus, as in Example \ref{exmp:sympl_bundles}(2), it has a canonical symplectic structure given by the standard symplectic form.
Let us denote by $(W,\phi_W)$ be the symplectic vector bundle $(\mathcal{O}_X(1)^{n+1}\oplus \mathcal{O}_X(-1)^{n+1}, \phi_{2n+2})$. We then get the stable symplectic twist $(-(W,\phi_W), \vartheta, 1)$ of $-T_X$ given by the isomorphism
\[
\vartheta:=\Sigma^{-T_X- \sO_X(-1)^{n+1}}(\Anan_{\sO_X(-1),\ldots,\sO_X(-1)}):\Sigma^{-T_X}1_X\xrightarrow{\sim} \Sigma^{2,1}\Sigma^{-W}1_X.
\]
Via Definition~\ref{defn:SymplecticTwistedClass}, we obtain the twisted class
\[
[\P^{2n+1}, -(W,\phi_W), \vartheta, 1]_\MSp\in \MSp^{-4n-2, -2n-1}(k)=\MSp_{4n+2, 2n+1}(k).
\]
\end{exmp}
The problem with the symplectic class of Example \ref{exmp:symplprojspaces} is that its image through the motivic Hurewicz map $h_{H\Z}$ lies in homological bidegree $(4n+2,2n+1)$, and since the motivic homology of $\MSp$ is generated by elements of even weight, this class has trivial homology. We now refine this procedure to get classes with non-trivial homology. The construction is inspired by a topological analogue due to Stong \cite[Section 4]{Stong-cobordism}.
    
\begin{constr}
\label{constr:Stongvars} We work over an infinite field $k$.
We take a direct product of an even number $2r$ of odd dimensional projective spaces $$X=\mathbb{P}^{2n_1 +1} \times_k \ldots \times_k \mathbb{P}^{2n_{2r} +1}.$$
Let $\xi=\mathcal{O}(1,\ldots,1)$ be the line bundle $p_1^*\mathcal{O}(1)\otimes \ldots \otimes p_{2r}^* \mathcal{O}(1)$ over $X$, where $p_j$ is the projection on the $j$-th component. The rank $2$ vector bundle $\xi \oplus \xi$ over $X$ is globally generated. By Proposition \ref{prop:BertiniThm} below, there exists a global section $s$ of $\xi \oplus \xi$, such that its vanishing locus is a closed subvariety $i:Y\hookrightarrow X$ of codimension $2$ of $X$, smooth and proper over $k$. Let $\xi'$ denote the restriction $i^*\xi$. The normal bundle $N_i$ over $Y$ is $i^*(\xi \oplus \xi)=\xi'\oplus \xi'$. We get the short exact sequence
\[
0\to T_Y\to i^*T_X\to N_i\to 0
\]
of vector bundles over $Y$, which gives the canonical isomorphism $\Sigma^{-T_Y}1_Y\cong \Sigma^{N_i-i^*T_X}1_Y$. Following Example \ref{exmp:symplprojspaces}, we get symplectic bundles $(W_j,\phi_{W_j})$ on $\P^{2n_j+1}$, and isomorphisms
\[
\vartheta_j:\Sigma^{-T_{\P^{2n_j+1}}}1_{\P^{2n_j+1}}\xrightarrow{\sim}\Sigma^{2,1}\Sigma^{-W_j}1_{\P^{2n_j+1}};\quad j=1,\ldots 2r,
\]
inducing the isomorphism
\[
\wedge_{j=1}^{2r}\vartheta_j:\Sigma^{-T_X}1_X\xrightarrow{\sim} \Sigma^{4r,2r}\Sigma^{-\oplus_{j=1}^{2r}p_j^*W_j}1_X.
\]
In addition, we have the Ananyevskiy isomorphism
\[
\vartheta_0:=\Sigma^{\xi}\Anan_\xi:\Sigma^{N_i}1_Y=\Sigma^{\xi'\oplus \xi'}1_Y\xrightarrow{\sim} \Sigma^{\xi'\oplus \xi^{\prime\vee}}1_Y
\]
and the standard symplectic bundle $(\xi'\oplus \xi^{\prime\vee},\phi_2)$, as in Example \ref{exmp:sympl_bundles}(2) again. Putting this together gives the stable symplectic twist of $-T_Y$
\[
((\xi'\oplus \xi^{\prime\vee},\phi_2)-\oplus_{i=1}^{2r}i^*p_j^*(W_j, \phi_{W_j}), \vartheta_Y, 2r),
\]
with $\vartheta_Y$ being the isomorphism
\[
\Sigma^{-T_Y}1_Y\simeq \Sigma^{N_i-i^*T_X}1_Y\xrightarrow{\vartheta_0\wedge
(\wedge_{j=1}^{2r}i^*\vartheta_j)}\Sigma^{4r,2r}\Sigma^{(\xi'\oplus \xi^{\prime\vee})-\sum_ji^*p_j^*W_j}1_Y.
\]
Let $2n:=\dim_kX=2r+2\sum_{j=1}^{2r}n_j$. So $\dim_kY=2n-2$, and $\rnk((\xi'\oplus \xi^{\prime\vee}-\sum_ji^*p_j^*W_j)=-2n-2r+2$. Via Definition~\ref{defn:SymplecticTwistedClass}, this gives us the twisted class
\[
[Y, (\xi'\oplus \xi^{\prime\vee},\phi_2)-\sum_ji^*p_j^*(W_j,\phi_{W_j}), \vartheta_Y, 2r]_\MSp\in
\MSp^{-4n+4,-2n+2}(k).
\]
\end{constr}

\begin{prop}[Variant of Bertini's Theorem]
\label{prop:BertiniThm}
    Let $X$ be a smooth projective variety over $k$, and let $L$ be a very ample line bundle over $X$. Then there exists a Zariski open subset $U$ of the projective space $\mathbb{P}(H^0(X,L))$ on $H^0(X,L)$ such that, for each section $s$ in $U$, its vanishing locus $V(s)$ defines a codimension one closed subscheme of $X \times_k k(s)$ that is smooth over $k(s)$. In particular, if $k$ is infinite, there exists a section $s \neq 0$ in $H^0(X,L)$ such that $V(s)$ is a codimension one closed subscheme of $X$, smooth over $k$.
\end{prop}

\begin{proof}
    Since $X$ is projective, the $k$-vector space $H^0(X,L)$ of global sections is finite-dimensional, and we have $\P(H^0(X,L))\simeq \P^N$, with $N=h^0(X,L)-1$. Since $L$ is very ample, there exists a closed immersion $i:X \hookrightarrow \P(H^0(X,L))\simeq \P^N$ such that $L$ is the pullback $i^*\mathcal{O}(1)$ of the line bundle $\mathcal{O}(1)\to \P^N$, and the hyperplane sections of $X$ are exactly the closed subschemes $Z(s)\hookrightarrow X$ given by zero loci of global sections $s:X \to L$. 

    Let $d\coloneqq \dim_k X$. For each point $x \in X$, we consider the projective tangent space $\overline{T}_{(X,x)}$ of $X$ at the point $x$ as a $d$-dimensional projective subspace $\P^d \subset \P^N$. Let us consider the dual space $(\P^N)^*$ of $\P^N$, defined as the space parametrizing the hyperplane divisors in $\P^N$, which is isomorphic to $\P^N$, and take the subset $W \subseteq X \times (\P^N)^*$ defined by
    $$W \coloneqq \{(x,H) \in X \times (\P^N)^* \mid \overline{T}_{(x,X)}\subseteq H \}.$$

    It is easy to see that for any $d$-dimensional linear subspace $L\simeq \P^d \subseteq \P^N$, the set of hyperplanes of $\P^N$ containing $L$ can be parametrized by the set of hyperplanes in a complementary linear subspace $L'$, by taking $L'\simeq \P^{N-d-1}$ such that $L \cap L'=\{0\}$. Therefore, for each point $x \in X$, the fibre $f^{-1}(x) \subset W$ over $x$ through $f:W \hookrightarrow X \times \P^N \xrightarrow{p_1} X$ is isomorphic to $\P^{N-d-1}$, and $W \subset X \times (\P^N)^*$ is closed. We deduce that $W$ has dimension $N-1<N$. In particular, the closed subset $p_2(W)\subset (\P^N)^*\simeq \P^N$ is a proper closed subset. We can then consider the open complement $U \coloneqq \P^N \setminus p_2(W)$ of $\P^N$. 
    
    Let us now take a hyperplane $H \in U$. For every point $x \in X$, the intersection $H \cap \overline{T}_{X,x}$ is a codimension $1$ linear space in $\overline{T}_{X,x}$. This is true in particular for all $x \in X \cap H$, which means that $X \cap H$ is smooth over the field of definition $k(H)$ of $H$. We can then identify $(\P^N)^*$ with $\P(H^0(X,L)$, $U$ with an open subspace of $\P(H^0(X,L)$, and hyperplanes $H\in U$ with sections $s$ in $U$. We obtain that, for each $s$, $Z(s)$ is a codimension $1$ subscheme of $X \times_k k(s)$, smooth over $k(s)$.
\end{proof}

\begin{defn}
\label{def:symplecticclasses}
Let $X$ and $Y$ as in Construction \ref{constr:Stongvars} and let $d_Y:=\dim_kY$. Let us denote by
$(v_Y,\omega_Y)$ the virtual symplectic bundle $(\xi'\oplus \xi^{\prime\vee},\phi_2)-\sum_ji^*p_j^*(W_j,\phi_{W_j})$, and let $\vartheta_Y:\Sigma^{-T_Y}1_Y\xrightarrow{\sim} \Sigma^{4r,2r}\Sigma^{v_Y}1_Y$ be the isomorphism described in Construction~\ref{constr:Stongvars}. This gives the stable symplectic twist $((v_Y,\omega_Y), \vartheta_Y, 2r)$ of $-T_Y$, and the symplectic bordism class
\[
[Y, (v_Y,\omega_Y),\vartheta_Y, 2r]_\MSp\in \MSp^{-2d_Y, -d_Y}(k)=\MSp_{2d_Y, d_Y}(k).
\]
Since $\vartheta_Y$ and $2r$ are canonical, we can use the shorthand $[Y, (v_Y,\omega_Y)]_\MSp:=[Y, (v_Y,\omega_Y),\vartheta_Y, 2r]_\MSp$.
\end{defn}

 We have then achieved the following:

 \begin{thm}
 \label{thm:symplecticclasses}
    Let $k$ be an infinite field. For every $X \in \Sm/k$ given by a product of an even number of odd dimensional projective spaces, of total dimension $2d+2$, let $\xi$ be the very ample invertible sheaf $\sO_X(1,\ldots, 1)$. Then there is a codimension $2$ irreducible proper subvariety $Y$ of $X$, smooth over $k$, given as the zero subscheme of a section of $\xi\oplus \xi$, that defines a class $[Y, (v_Y,\omega_Y)]_\MSp \in \MSp_{4d,2d}(k)$.
 \end{thm}
 
 We now take care of the characteristic numbers $s_{(2d)}(Y)$ (Definition \ref{defn:NewtonClassesAndStuff} (4)).
 
  \begin{lemma}
\label{lemma:segrenumbers}
    Let $Y, X$ be as in Construction \ref{constr:Stongvars}, with $i:Y \hookrightarrow X$ the closed immersion. We again denote $\dim_kX$ by $2n$, and let us suppose that $X$ is not of the form $\P^1 \times \P^{2n-1}$. Let $2d \coloneqq 2n-2=\dim_kY$. Also, let $\alpha \coloneqq c_1(\xi) \in H\mathbb{Z}^{2,1}(X)$. Then
    $$s_{(2d)}(Y)= (-2)\cdot\deg_k(\alpha^{2d+2}).$$
\end{lemma}

\begin{proof}
    The short exact sequence
    $$0 \to T_Y \to i^*T_X \to N_i \to 0$$
    gives $c_{(2d)}(T_Y) + c_{(2d)}(N_i)=c_{(2d)}(i^*T_X)$, but since $X$ is a product of varieties each of dimension less than $2d$, Lemma \ref{lemma:newtonclasses}(3)-(4) gives $c_{(2d)}(i^*T_X)=i^*c_{(2d)}(T_X)=0$. Then $$c_{(2d)}(T_Y)=-c_{(2d)}(N_i)=-c_{(2d)}(i^*(\xi \oplus \xi))=-i^*c_{(2d)}(\xi \oplus \xi).$$
    Moreover, $c_{(2d)}(\xi \oplus \xi)=(c_1(\xi))^{2d}+(c_1(\xi))^{2d}=2\alpha^{2d}$. We have then
    \begin{equation}\label{eq:degree_through_alpha}
        \text{deg}_k(c_{(2d)}(T_Y))= \text{deg}_k(-2i^*\alpha^{2d}).
    \end{equation}
    
    We now claim that $i_*i^*x = \alpha^2 x$ for all $x \in H\mathbb{Z}^{4d,2d}(X)$. To see this, let $1^Y\in H\mathbb{Z}^{0,0}(Y)$, $1^X\in H\mathbb{Z}^{0,0}(X)$   be the respective units. Then the projection formula (Proposition~\ref{prop:PushPull}(2)) gives $ i_*i^*x=i_*(1^Y\cdot i^*x)= i_*(1^Y)\cdot x$. Next, we recall that $Y$ is the zero-locus of a global section $s$ of $\xi \oplus \xi$ that is transverse to the zero section $s_0$. This gives us the transverse cartesian diagram
\[
\xymatrix{
Y\ar[r]^i\ar[d]^-i&X\ar[d]^{s_0}\\
X\ar[r]^-s&\xi \oplus \xi.
}
\]
Thus, by the push-pull formula (Proposition~\ref{prop:PushPull}(1)) and Lemma~\ref{lem:FirstChernClassFacts}, we have $i_*(1^Y)=i_*i^*(1_X)=s^*s_{0*}(1_X)=c_2(\xi \oplus \xi).$ Finally, the Whitney sum formula for Chern classes gives $c_2(\xi \oplus \xi)=c_1(\xi)^2=\alpha^2$, 
so $ i_*i^*x=\alpha^2\cdot x$, proving the claim.
  
Thus we have
\begin{multline*}
s_{(2d)}(Y)\overset{\eqref{eq:degree_through_alpha}}{=}p_{Y*}(-2i^*\alpha^{2d}) = p_{X*}i_*(-2i^*\alpha^{2d}) \overset{\text{claim}}{=}
p_{X*}(-\alpha^2\cdot 2\alpha^{2d})
=(-2)\cdot\deg_k(\alpha^{2d+2}).
\end{multline*}    
\end{proof}

We can now refine Theorem \ref{thm:symplecticclasses} by using Lemma \ref{lemma:segrenumbers} and the ideas of \cite[Section 4]{Stong-cobordism} in topology.

\begin{thm}
\label{thm:symplclasses2} Let $k$ be an infinite field, and let $\ell$ be an odd prime different from $\chr k$. Then there exists a family $\{(Y_{2d}, (v_{2d}, \omega_{2d}))\}_{d \ge1}$, with $Y_{2d}\in \Sm/k$ proper of dimension $2d$ over $k$, and $(v_{2d}, \omega_{2d})\in K_0^\Sp(Y_{2d})$ a virtual symplectic bundle defining a symplectic twist of $T_{Y_{2d}}$, thereby giving  a class $[Y_d,(v_{2d}, \omega_{2d})]_{\MSp} \in \MSp_{4d,2d}(k)$, such that
    $$\nu_{\ell}(s_{2d}(Y_{2d}))=
    \begin{cases}
        0, \; \; \; \text{if} \;\;  2d \neq \ell^i-1 \; \forall \; i \\
        1, \; \; \; \text{if} \; \; 2d=\ell^r -1, \; r \ge 1.
    \end{cases}$$
\end{thm}

\begin{proof}
    
Let $d$ be any positive integer such that $2d \neq \ell^i-1$ for all $i$. Let
$$2d+2= a_0 +a_1\ell + \ldots +a_r \ell^r$$
be the $\ell$-adic expansion of $2d+2$. Let us define
$$X_{2d+2} \coloneqq (\mathbb{P}^1)^{\times a_0} \times (\mathbb{P}^{\ell})^{\times a_1} \times \ldots \times (\mathbb{P}^{\ell^r})^{\times a_r},$$
then $\dim(X_{2d+2})=2d+2$. Let $Y_{2d} \subset X_{2d+2}$ be the zero locus of a general section of the vector bundle $\xi \oplus \xi \to X_{2d+2}$ as in Construction \ref{constr:Stongvars}. Then $s_{(2d)}(Y_{2d})=\deg_k (-2\alpha^{2d+2})$ by Lemma \ref{lemma:segrenumbers}. 
Let $\alpha_{i,j}\coloneqq c_1(p_{i,j}^*\mathcal{O}(1))$, where $p_{i,j}$ is the projection of $X_{2d+2}$ onto the $j$-th component of $(\mathbb{P}^{\ell^i})^{\times a_i}$, so that 
$$\alpha=(\alpha_{0,1}+\alpha_{0,2}+\ldots +\alpha_{0,a_0}+\alpha_{1,1}+\ldots + \alpha_{1,a_1}+\ldots +\alpha_{r,a_r}).$$
In particular:
$$\alpha^{2d+2}=(\alpha_{0,1}+\alpha_{0,2}+\ldots +\alpha_{0,a_0}+\alpha_{1,1}+\ldots + \alpha_{1,a_1}+\ldots +\alpha_{r,a_r})^{a_0 +a_1\ell +\ldots +a_r \ell^r}.$$

Since $X_{2d+2}$ is a product of projective spaces, we have K\"unneth formula
$$H\Z^{2*, *}(X)\cong H^{2*,*}(\P^1)^{\otimes a_0}\otimes\ldots\otimes H^{2*,*}(\P^{\ell^r})^{\otimes a_r},$$
with all tensor products over $\Z$, and the isomorphism given by taking the respective pullbacks and cup products. This follows by repeated application of the projective bundle formula (Theorem \ref{thm:PBF}). For all projective spaces $\P^n$, the top degree term in $H^{2*,*}(\P^n)$ is $H^{2n,n}(\P^n)\cong \Z$, with generator $c_1(\mathcal{O}_{\P^n}(1))^n$, and $p_{\P^n*}(c_1(\mathcal{O}_{\P^n}(1))^n)=1\in H\Z^{0,0}(\Spec k)\cong \Z$. Thus, through a series of uses of the push-pull formula and the projection formula (Proposition \ref{prop:PushPull}(1)-(2)), we see that the top degree term in $H\Z^{2*, *}(X_{2d+2})$ is $H\Z^{4d+4, 2d+2}(X_{2d+2})\cong \Z$, with generator 
\[
\alpha_*^*:=\alpha_{0,1} \cdot\alpha_{0,2}\cdots\alpha_{0,a_0} \cdot \alpha_{1,1}^{\ell} \cdots \alpha_{1,a_1}^{\ell} \cdots \alpha_{r,1}^{\ell^r} \cdots \alpha_{r,a_r}^{\ell^r},
\]
and $\deg_k(\alpha_*^*)=1$. Therefore, $\deg_k(2\alpha^{2d+2})$ is the coefficient of 
$$\alpha_{0,1} \cdot \alpha_{0,2}\cdots\alpha_{0,a_0} \cdot \alpha_{1,1}^{\ell} \cdots \alpha_{1,a_1}^{\ell} \cdots \alpha_{r,1}^{\ell^r} \cdots \alpha_{r,a_r}^{\ell^r}$$
in the expansion of $2\alpha^{2d+2}$.

The multinomial theorem gives the combinatorial formula
$$(x_1 +x_2 + \ldots x_m)^n=\sum_{i_1 + \ldots +i_m=n}\frac{n!}{i_1!i_2!\ldots i_m!} x_1^{i_1}x_2^{i_2}\ldots x_m^{i_m},$$
and multinomial coefficients satisfy the identity
$$\frac{n!}{i_1!i_2!\ldots i_m!} =\binom{i_1}{i_1} \binom{i_1+i_2}{i_2}\cdots \binom{i_1+i_2+\ldots +i_m}{i_m}.$$
Thus, $\deg_k(2\alpha^{2d+2})$ is the coefficient
\begin{multline*}
     2 \cdot \frac{(2d+2)!}{1!1!\cdots 1!\ell! \cdots \ell! \ell^{2}! \cdots \ell^r!}= \\ = 2 \binom{1}{1}\binom{2}{1} \cdots \binom{a_0}{1} \binom{a_0+\ell}{\ell}\binom{a_0+2\ell}{\ell} \cdots \binom{a_0+a_1\ell}{\ell} \binom{a_0+a_1+\ell^2}{\ell^2} \cdots\binom{2d+2}{\ell^r}.
\end{multline*}
The product of the first $a_0$ binomial coefficients is $a_0!$. The product of the following $a_1$ binomial coefficients, reduced modulo $\ell$, is $a_1!$. And so on. Iterating this, we get
$$\deg_k(2\alpha^{2d+2}) \equiv\ 2a_0!a_1! \cdots a_r! \; \; \; \text{mod}\;\ell.$$
In particular, $\ell$ does not divide $\deg_k(2\alpha^{2d+2})$, so $\nu_{\ell}(s_{2d}(Y_d))=0$.

Let now $d$ be a positive integer such that $2d=\ell^r-1$ for some $r$. In this case, let us define
$$X_{2d+2} \coloneqq \mathbb{P}^1 \times (\mathbb{P}^{\ell^{r-1}})^{\ell},$$
so that $\dim(X_{2d+2})=\ell^r+1=2d+2$, and take $Y_{2d} \subset X_{2d+2}$ obtained again as in Construction \ref{constr:Stongvars}. Now we let $p_0$ be the projection of $X_{d+2}$ on the first factor $\mathbb{P}^1$, and $p_{1,j}:X_{2d+2} \to \mathbb{P}^{\ell^{r-1}}$ be the projection on the $j$-th component of $(\mathbb{P}^{\ell^{r-1}})^{\ell}$. As before, we write 
$$\alpha=\alpha_0 +\alpha_{1,1} +\ldots + \alpha_{1,\ell},$$
and then
$$\alpha^{2d+2}=(\alpha_0 +\alpha_{1,1} +\ldots + \alpha_{1,\ell})^{\ell^r+1}.$$
So, $\deg_k(2\alpha^{2d+2})$ is now the coefficient of
$\alpha_0 \cdot \alpha_{1,1}^{\ell^{r-1}} \cdot \alpha_{1,2}^{\ell^{r-1}} \cdots \alpha_{1,\ell}^{\ell^{r-1}}$
in the expansion of $2\alpha^{\ell^r+1}$. By using the multinomial theorem as before, we get
$$\deg_k(2\alpha^{2d+2})=2\binom{1}{1}\binom{1+\ell^{r-1}}{\ell^{r-1}}\binom{1+2\ell^{r-1}}{\ell^{r-1}} \cdots \binom{1+\ell^r}{\ell^{r-1}}.$$
Noting that:
$$\binom{1+j\ell^{r-1}}{\ell^{r-1}}=\binom{j\ell^{r-1}}{\ell^{r-1}}\frac{1+j\ell^{r-1}}{1+j\ell^{r-1}-\ell^{r-1}},$$
it is easy to see that $\nu_{\ell}(\deg_k(2\alpha^{2d+2}))=1$, concluding the proof.
\end{proof}

From now on, we can adopt the following notation.

\begin{defn}
\label{defn:SymplecticY's}
    For $\ell$ a fixed odd prime different from $\text{char}(k)$, we denote by
    $$[Y_{2d}, (v,\omega)]_{\MSp} \in \MSp_{4d,2d}(k)$$
    the class associated as in Definition \ref{def:symplecticclasses} to the variety $Y_{2d}$ defined in Theorem \ref{thm:symplclasses2}. 
\end{defn}

\begin{rmk}
\label{rmk:ComparingNewtonClasses}
    $s_{2d}(Y_{2d})$ also computes $\text{deg}_k(c_{(2d)}(i^*W-(\xi' \oplus \xi'^{\vee})))$. Indeed, Lemma \ref{lemma:newtonclasses}(2)-(3) gives 
    \begin{multline*}
        c_{(2d)}(\xi' \oplus \xi'^{\vee})=i^*c_{(2d)}(\xi) + i^*c_{(2d)}(\xi^{\vee})= i^*\alpha^{2d}+i^*(-\alpha)^{2d}=\\i^*(2\alpha^{2d})=i^*c_{(2d)}(\xi \oplus \xi)=c_{(2d)}(N_i)=-c_{(2d)}(T_{Y_{2d}}),
    \end{multline*}
    and Lemma \ref{lemma:newtonclasses}(1) gives $c_{(2d)}(i^*W)=i^*c_{(2d)}(W)=0$ because $W$ is the direct sum of pullbacks $p_j^*W_j$ on $Y_{2d}$, where $W_j$ is a vector bundle on a space $\P^{2n_j+1}$ of dimension less than $2d$.
\end{rmk}

\subsection{Detecting generators} We start by recalling a few facts on $\MGL$. In \cite[Lemma 5.9]{lev:ellcoh}, they give a description of the trigraded $\Z/\ell$-algebra $\Ext_{A^{*,*}}(H^{*,*}(\MGL),H^{*,*})$, as we did in Proposition \ref{prop:pres.ExtAlgebra} after replacing $\MGL$ with $\MSp$. In particular, $\Ext_{A^{*,*}}^{s,(2u-s,u)}(H^{*,*}(\MGL),H^{*,*})$ is polynomial in the generators:
\begin{equation}\label{eq:Ext_MGL_generators}
\end{equation}
\begin{itemize}
    \item[$-$] $1 \in \Ext_{A^{*,*}}^{0,(0,0)}(H^{*,*}(\MGL),H^{*,*})$;
    \item[$-$] $z'_{(k)}\in \Ext_{A^{*,*}}^{0,(-2k,-k)}(H^{*,*}(\MGL),H^{*,*})$, for $k\ge 1$, and $k\neq \ell^i-1 \; \forall \; i \ge 0 $;
    \item[$-$] $h_r' \in \Ext_{A^{*,*}}^{1,(1-2 \ell^r,1-\ell^r)}(H^{*,*}(\MGL),H^{*,*})$, for $r \ge 0$.
\end{itemize}

It is known that $\MGL$ is $(-1)$-connected (\cite[Corollary 2.3]{lev:ellcoh}) and $\eta$-completed, namely $\MGL^\wedge_\eta = \MGL$ (\cite[Lemma 2.1]{Rondigs:Hopf}). Thus, by Theorem \ref{thm:Mantovani}, one has $\MGL_{H\Z/\ell}^\wedge \simeq \MGL_{\eta,\ell}^\wedge$. Also, \cite[Lemma 6.3]{lev:ellcoh} shows that $(\MGL^\wedge_\ell)^* \cong (\MGL^*)^\wedge_\ell$. Putting this together gives
$$(\MGL_{H\Z/ \ell}^\wedge)^* \cong (\MGL^*)^\wedge _\ell.$$
Moreover, $\MGL$ is a cellular spectrum (Example \ref{exmp:MGL_cellular}), and $W_s(\MGL)$ is a motivically finite type wedge of copies of $H\Z/\ell$ (\cite[Proposition 6.10]{lev:ellcoh}), which means that $\MGL$ satisfies the conditions of Proposition \ref{prop:a.s.s}. This leads to the following.

\begin{thm}[\cite{lev:ellcoh}, Theorem 6.11, Proposition 6.12]
\label{thm:a.s.sMGL}
The motivic mod $\ell$ Adams spectral sequence for $\MGL$ is of the form
$$E_2^{s,t,u}=\Ext_{A^{*,*}}^{s,(t-s,u)}(H^{*,*}(\MGL),H^{*,*}) \Rightarrow (\MGL_{H\Z/\ell}^\wedge)^{t,u},$$
and the sequence converges completely to $(\MGL_\ell^\wedge)^{2u,u}$.
\end{thm}

\begin{rmk}
\label{rmk:liftingGenerators}
    As noted in \cite[Section 6.6]{lev:ellcoh} along the proof of Theorem B, the Adams spectral sequence for $\MGL$ has associated filtration $F^*(\MGL^*)^\wedge_\ell$, graded by $\oplus_u E_\infty^{2u,u}(\MGL)\simeq \oplus_u E_2^{2u,u}(\MGL) \coloneqq \oplus_u \oplus_s E_2^{s,(2u,u)}(\MGL)$, which is a graded $\Z/\ell$-algebra, with $E_2^{0,(0,0)}(\MGL)=\Z/\ell$ and $F^m((\MGL^0)_\ell^\wedge)=(\ell^m)\Z_\ell$. Thus, the polynomial generators $z'_{(k)}$ and $h'_r$ in \eqref{eq:Ext_MGL_generators} lift, respectively, to polynomial generators
    $$\tilde{z}'_{(k)} \in F^0((\MGL^{-2k,-k})^\wedge_\ell) \; \; \; \text{and} \; \; \; \tilde{h}'_r \in F^1((\MGL^{-2(\ell^r-1),-(\ell^r-1)})_\ell^\wedge)$$
    for the graded $\Z/\ell$-algebra $(\MGL^*)^\wedge_\ell$, and these generators are unique up to decomposable elements and elements in $\ell \cdot (\MGL^*)^\wedge_\ell$.
\end{rmk}

The natural map $\Phi:\MSp \to \MGL$ induces the map
$$\Phi_*:(\MSp_{\eta,\ell}^\wedge)^* \to (\MGL_\ell^\wedge)^*\simeq (\MGL^*)_\ell^\wedge.$$

\begin{prop}\label{prop:PolynomialMSp}The map $\Phi_*:(\MSp_{\eta, \ell}^\wedge)^* \to (\MGL^*)_\ell^\wedge=\Z_\ell[x_1, x_2,\ldots]$ is injective, with image the polynomial subring with generator $x_2, x_4,\ldots$.
\end{prop}

\begin{proof} The map $\Phi: \MSp \to \MGL$ induces a map between the two respective Adams spectral sequences. By looking at the decompositions of $H^{*,*}(\MSp)$ and $H^{*,*}(\MGL)$ given by Lemma \ref{lemma:decompmsp} and \cite[Lemma 5.9]{lev:ellcoh}, we see that the induced map between the two respective $E_2$-pages is an isomorphism on the summands of the $\Ext$-algebras corresponding to the duals of the summands $M_B u_\omega$, with $w\in P$, namely with $\omega$ being an even (necessarily non-$\ell$ adic) partition. In particular, the $E_2$ page for $\MSp$ is a split summand of the $E_2$ page for $\MGL$. Since both spectral sequences degenerate at $E_2$, this implies that the map 
$\Phi_*:(\MSp_{\eta, \ell}^\wedge)^* \to (\MGL^*)_\ell^\wedge$ is injective.

We consider the filtration $F^*(\MSp_{\eta,\ell}^\wedge)^*$ associated to the spectral sequence $E(\MSp)$ for $\MSp$. Analogously to the case of $F^*(\MGL^*)_\ell^\wedge$ in Remark \ref{rmk:liftingGenerators}, the generators $z_{(2k)}$ and $h_r$ of $E_2(\MSp)$ lift to generators $\tilde{z}_{(2k)}$ and $\tilde{h}_r$ of the graded $\Z_\ell$ algebra $(\MSp_{\eta,\ell}^\wedge)^*$, with $r\ge 0$, $k\ge 1$, $2k \neq \ell^i-i$ for any $i \ge 0$,  uniquely up to sums with decomposable elements and elements in $\ell \cdot (\MSp_{\eta,\ell}^\wedge)^*$. We see that, through the inclusion $\Phi_*:(\MSp_{\eta, \ell}^\wedge)^* \hookrightarrow (\MGL^*)_\ell^\wedge$, $\tilde{z}_{(2k)}$ and $\tilde{h}_r$ map to the polynomial generators $\tilde{z}'_{(2k)}$ and $\tilde{h}'_r$ respectively, up to decomposable elements and elements in $\ell \cdot (\MGL^*)^\wedge_\ell$. Indeed, $z_{(2k)}$ is the dual of $u_{(2k)}$,  hence it belongs to the summand of $E_2(\MSp)$ corresponding $M_B u_{(2k)}$, while the elements $h_r$ are generators of the $\Z/\ell$ algebra $\Ext_B(\Z/\ell,H^{*,*})\simeq \Ext_{A^{*,*}}(M_B,H^{*,*})$, so they belong to the summand corresponding to $M_B u_{(0)}$. 

Thus, the $\Z_\ell$-algebra generators $\tilde{z}_{(2k)}$ and $\tilde{h}_r$ of $(\MSp_{\eta, \ell}^\wedge)^*$ map to (possibly new) polynomial generators of the $\Z_\ell$-algebra $(\MGL^*)_\ell^\wedge=\Z_\ell[\{\tilde{z}'_{(k)}, \tilde{h}'_r\}_{k,r}]$ in all even degrees. Thus $\Phi_*$ identifies $(\MSp_{\eta, \ell}^\wedge)^*$ with this polynomial subalgebra of $(\MGL^*)_\ell^\wedge$.
\end{proof}

From the generating criterion for $(\MGL^*)^\wedge_\ell$ of Proposition \ref{prop:CriterionMGL}, we can obtain the following one.

\begin{prop}
\label{prop:CriterionMSp}
    A family of element $\{y_{2d}'\}_{d \ge 1}$, with $y_{2d}' \in (\MSp^\wedge_{\eta,\ell})^{-2d}$ form a family of polynomial generators of $(\MSp_{\eta,\ell}^\wedge)^*$ if and only if
    \begin{equation*}
        c_{(2d)}(\Phi_* y_{2d}')=
        \begin{cases}
            \lambda \in \Z_\ell^\times & \text{for} \; \; 2d \neq \ell^r-1 \; \forall r \ge 1 \\
            \lambda' \cdot \ell, \; \lambda' \in \Z_\ell^\times & \text{for} \; \; 2d= \ell^r-1, \; r\ge 1.
        \end{cases}
    \end{equation*}
\end{prop}

\begin{proof} We retain the notation used in the proof of Proposition~\ref{prop:PolynomialMSp}. 

From Remark \ref{rmk:liftingGenerators}, we see that the elements $\tilde{h}_r'$, for $r \ge 1$, give generators of $(\MGL^*)^\wedge_\ell$ of degrees $1-\ell^r$ respectively, and the elements $\tilde{z}_{(k)}'$ give generators in the remaining non-zero degrees. Thus, by Proposition \ref{prop:CriterionMGL}, we have $\nu_\ell(c_{(k)}(\tilde{z}_k'))=0$ and $\nu_\ell(c_{(\ell^r-1)}(\tilde{h}'_r))=1.$ Therefore, we also have
\begin{equation}
    \nu_\ell(c_{(2k)}(\Phi_* \tilde{z}_{2k}))=0 \; \; \; \text{and} \; \; \; \nu_\ell(c_{(\ell^r-1)}(\Phi_* \tilde{h}_r))=1.
\end{equation}
Let $y\in (\MGL^n)^\wedge_\ell$ be a polynomial generator. Then, modulo $\ell\cdot (\MGL^n)^\wedge_\ell$ and decomposable elements, $y$ is the product of a $\Z_\ell$-unit and a generator $\tilde{h}_r'$ (if $n=1-\ell^r$) or $\tilde{z}_k'$ (if $n$ is not of the form $1-\ell^r$). Since Newton classes vanish on decomposables, we see that for $y\in (\MGL^n)^\wedge_\ell$ arbitrary, $y$ is a polynomial generator if and only if $\nu_\ell(c_{(n)}(y))=1$ if $n$ is of the form  $1-\ell^r$, and $\nu_\ell(c_{(n)}(y))=0$ if $n$ is not of this form. 

From Proposition~\ref{prop:PolynomialMSp}, we see that the sub-$\Z_\ell$-algebra $(\MSp^\wedge_{\eta,\ell})^*$ is the polynomial sub-$\Z_\ell$-algebra of $(\MGL^*)^\wedge_\ell$
with polynomial generators the $\tilde{h}'_r$, $r\ge1$ and the $\tilde{z}_k'$ for $k$ even (and non-$\ell$-adic). Applying the above criterion, we see that an element $y_{2d}' \in (\MSp^\wedge_{\eta,\ell})^{-2d}$, $d \ge 1$, is a polynomial generator of $(\MSp^\wedge_{\eta,\ell})^*$ if and only if  
 $$
 \nu_\ell(c_{(2d)}(\Phi_* y_{2d}'))=
        \begin{cases}
            1 \; \; \; \; \text{for} \; \; 2d= \ell^r-1, \; r\ge 1 \\
            0 \; \; \; \text{for} \; \; 2d \neq \ell^r-1 \; \forall r \ge 1.
        \end{cases}
$$
This gives the result.
\end{proof}

Let $X=\prod_{i=1}^{2r}\mathbb{P}^{2n_i +1}$ and $Y\subset X$ as in Construction \ref{constr:Stongvars}, with $2d:=\dim_kY=\dim_kX-2$. Definition \ref{def:symplecticclasses} gives the stable symplectic twist $((v_Y,\omega_Y), \vartheta_Y, 2r)$ of $-T_Y$. In order to simplify the notation, we introduce the following shorthand.
 \begin{defn}
 \label{defn:convention}
 We let $-T_Y'\in K_0(Y)$ be the virtual vector bundle $\mathcal{O}_Y^{2r}+v_Y$.
 \end{defn} 

 In particular, the isomorphism $\vartheta_Y$ of Construction \ref{constr:Stongvars} can be read as $$\vartheta_Y:\Sigma^{-T_Y}1_Y \xrightarrow{\sim}\Sigma^{-T_Y'}1_Y.$$

 Let us recall that, by Remark \ref{rmk:NewtonClasses}, $c_{(2d)}(-T_Y)$ and $c_{(2d)}(-T_Y')$ are defined.
 
 \begin{prop}
 \label{prop:comparison_SegreNumbers}
    We have
    $$c_{(2d)}(\Phi_* [Y,(v_Y,\omega_Y)]_\MSp)=(-1)^{n_Y}\deg_k(c_{(2d)}(-T_Y)),$$
    where $n_Y=1 + \Sigma_{i=1}^{2r}(n_i+1)$
\end{prop}

\begin{proof} We first note that, for $x \in \MGL^{-n}=\MGL_{2n,n}(\Spec k)$, and $c_{(n)} \in H\Z^{2n,n}(\MGL)$, by Remark \ref{rmk:HurewiczMap} we have
\[
c_{(n)}(x) := \langle c_{(n)},h(x) \rangle =c_{(n)}\circ x\in H\Z^{0,0}(\Spec k)=\Z.
\]
Next, expanding the full definition of $[Y,(v_Y,\omega_Y)]_\MSp$ as $[Y,(v_Y,\omega_Y),\vartheta_Y, 2r]_\MSp$, we have
\[
\Phi_* [Y,(v_Y,\omega_Y),\vartheta_Y, 2r]_\MSp =[Y, v_Y+\sO_Y^{2r},\vartheta_Y]_\MGL\in 
\MGL^{-2d,-d}(\Spec k),
\]
where $v_Y\in K_0(Y)$ is the image of $(v_Y,\omega_Y)\in K_0^\Sp(Y)$ under the evident map 
$K_0^\Sp(Y)\to K_0(Y)$ forgetting the symplectic structure. Thus
\[
c_{(2d)}\circ \Phi_* [Y,(v_Y,\omega_Y),\vartheta_Y, 2r]_\MSp=c_{(2d)}\circ [Y, -T_Y',\vartheta_Y]_\MGL.
\]
Furthermore, $-T_Y'$ has virtual rank $-2d$ and $\vartheta_Y$ is a composition of $2r+1$ isomorphisms of the form $\Anan_{e,-L_1,\ldots, -L_s}$, $2r$ for the factors $\P^{2n_1+1}$ and one coming from the rank $2$ bundle $\xi'\oplus \xi'$. In particular, for each factor $\P^{2n_i+1}$ of $X$ we have $s=n_i+1$, and for $\xi'\oplus \xi'$ we have $s=1$. Thus, the sum of all the indices $s$ relative to the isomorphisms $\Anan_{e,-L_1,\ldots, -L_s}$ is $n_Y$, and it follows from Corollary~\ref{cor:TwistClassComp} that 
\[
c_{(2d)}\circ [Y, -T_Y',\vartheta_Y]_\MGL=(-1)^{n_Y}\cdot\deg_k (c_{(2d)}(-T_Y'))
=(-1)^{n_Y}\cdot\deg_k (c_{(2d)}(-T_Y')).
\]

Finally, up to isomorphisms, $-T_Y$ and $-T_Y'$ are both (virtual) sums in $K_0(Y)$ of classes of line bundles, where the only difference is that some of the line bundles $L_i$ appearing in $-T_Y$ get replaced with $L_i^\vee$. But if $-T_Y=\sum_i\epsilon_i[L_i]$, with $\epsilon_i\in\{\pm1\}$, then by additivity of the Newton classes we have  $c_{(2d)}(-T_Y)=\sum_i\epsilon_i\cdot c_1^{H\Z}(L_i)^{2d}$ and $c_{(2d)}(-T_Y')=\sum_i\epsilon_i\cdot c_1^{H\Z}(L_i^{\otimes \tau_i})^{2d}$ for suitable $\tau_i\in\{\pm1\}$.
Since $H\Z$ has additive formal group law (Example \ref{exmp:orientedHZ}), we have $c_1(L_i^{\otimes -1})^{2d} =(-c_1(L_i))^{2d}=c_1(L_i)^{2d}$,
so $c_{(2d)}(-T_Y')=c_{(2d)}(-T_Y)$, completing the proof. 
\end{proof}

Let us now consider the varieties $Y_{2d}$ constructed in Theorem \ref{thm:symplclasses2}, and their associated symplectic classes $[Y_{2d}]_\MSp \coloneqq [Y_{2d},(v_Y,\omega_Y)]_\MSp$ (Definition \ref{defn:SymplecticY's}).

Still by Remark \ref{rmk:NewtonClasses} and additivity of Newton classes we have $c_{(2d)}(-T_{Y_{2d}})=-c_{(2d)}(T_{Y_{2d}})$. Thus, Proposition \ref{prop:comparison_SegreNumbers} implies the identity 
\begin{equation}
\label{eq:generatingconditions}
   \nu_\ell(c_{(2d)}(\Phi_*[Y_{2d}]_\MSp))=\nu_\ell(s_{2d}(Y_{2d})). 
\end{equation}
The identity \eqref{eq:generatingconditions} and Theorem \ref{thm:symplclasses2} tell us that the symplectic classes $[Y_{2d}]_\MSp \in \MSp^{-4d,-2d}$ satisfy the generating criterion for $(\MSp^\wedge_{\eta,\ell})^*$ established in Proposition \ref{prop:CriterionMSp}.

We have then achieved the following.

\begin{thm}
\label{thm:ClassOfGenerators}
    Let $\ell$ be a fixed odd prime different from $\chr(k)$, and for $d \ge 1$ let $[Y_{2d}]_\MSp \coloneqq [Y_{2d},(v_Y,\omega_Y)]_\MSp \in \MSp^{-4d,-2d}(\Spec k)$ be the symplectic classes of Definition \ref{defn:SymplecticY's}. Then the images of the classes $[Y_{2d},(v_Y,\omega_Y)]_\MSp$ in $(\MSp_{\eta,\ell}^\wedge)^*$ give a family of polynomial generators for $(\MSp_{\eta,\ell}^\wedge)^*$ over $\Z_\ell$.
\end{thm}

\begin{prop}
\label{prop:eta-ellCompl}
    For each odd prime $\ell$ different from $\chr k$, we have $$((\MSp^\wedge_\eta)^*)^\wedge_\ell \cong (\MSp_{\eta,\ell}^\wedge)^*.$$
\end{prop}

\begin{proof}
     We first note that there is an injective map $((\MSp^\wedge_\eta)^*)^\wedge_\ell\hookrightarrow (\MSp_{\eta,\ell}^\wedge)$. The argument is the same as in the proof of \cite[Lemma 6.13 (3)]{lev:ellcoh}, after replacing $\MSL$ with $\MSp$. To show the surjectivity, it is sufficient to show that $(\MSp_{\eta,\ell}^\wedge)^*$, seen as a subalgebra of $(\MGL_{\eta,\ell}^\wedge)^*$ through $\Phi_*$, is generated, as $\Z_\ell$-module, by the images $\Phi_*[Y_{2d}, (v_Y,\omega_Y)]_\MSp$. This follows directly from Theorem \ref{thm:ClassOfGenerators}.
\end{proof}

From Proposition \ref{prop:CriterionMSp}, by working on the $\ell$-torsion of $(\MSp_\eta^\wedge)^*$ for each $\ell$ separately, we can deduce the following global criterion.

\begin{thm}
\label{thm:FinalResult}
    Let $({\overline{\MSp}^\wedge_\eta})^*[1/2p]$ denote the quotient of the graded ring $({\MSp^\wedge_\eta})^*[1/2p]$ by its maximal subgroup that is $\ell$-divisible for all odd primes $\ell \neq p$. Then $({\overline{\MSp}^\wedge_\eta})^*[1/2p]$ is a polynomial ring over $\Z[1/2p]$, and a family $\{y_{2d}'\}_{d \ge 1}$, with $y'_{2d} \in ({\overline{\MSp}^\wedge_\eta})^{-2d}[1/2p]$, forms a family of polynomial generators for $({\overline{\MSp}^\wedge_\eta})^*[1/2p]$ if and only if 
    \begin{equation}
    \label{eq:conditions_final}
        c_{(2d)}(\Phi_* y_{2d}')=
        \begin{cases}
            \lambda \in \Z[1/2p]^\times & \text{for} \; \; 2d \neq \ell^r-1, \; \; \text{for all prime} \;  \ell \neq 2,p; \; \forall r \ge 1 \\
            \lambda' \cdot \ell, \; \lambda' \in \Z_\ell^\times & \text{for} \; \; 2d= \ell^r-1, \; \; \ell \; \text{prime}, \; \ell \neq 2,p; \; r\ge 1.
        \end{cases}
    \end{equation}
\end{thm}

\begin{proof} $(\overline{\MSp}_\eta^\wedge)^*[1/2p]$ is a $\Z[1/2p]$-module. Then for all odd prime $\ell$ different from $p$, we have 
    $$(\overline{\MSp}_\eta^\wedge)^*[1/2p] \otimes_{\Z[1/2p]}\Z_\ell \cong  (\overline{\MSp}_\eta^\wedge)^* \otimes_\Z \Z_\ell \cong ((\MSp_\eta^\wedge)^*)^\wedge_\ell \cong (\MSp_{\eta,\ell}^\wedge)^*,$$
    where the last isomorphism is Proposition \ref{prop:eta-ellCompl}.

    By using Proposition \ref{prop:CriterionMSp} as a generating criterion for $(\overline{\MSp}_\eta^\wedge)^*[1/2p] \otimes_{\Z[1/2p]}\Z_\ell$, for each $\ell$, we obtain that an element $y'_{2d} \in {(\overline{\MSp}^\wedge_\eta})^{-2d}[1/2p]$ is a generator if and only if it satisfies the conditions \eqref{eq:conditions_final}.
\end{proof}

\begin{rmk}
    Let us note that Theorem \ref{thm:symplclasses2} also provides a family $\{[Y_{2d}, (v_Y,\omega_Y)]_\MSp\}_{d\ge 1}$ of polynomial generators for $({\overline{\MSp}^\wedge_\eta})^*[1/2p]$. If $2d=\ell^r-1$ for some odd prime $\ell \neq p$ and some $r \ge 1$, one takes the variety $Y_{2d} \subset \P^1 \times (\P^{\ell^{r-1}})^{\times \ell}$ constructed in the proof of Theorem \ref{thm:symplclasses2} for that specific $\ell$, and its respective symplectic class $[Y_{2d}, (v_Y,\omega_Y)]_\MSp$. If $2d \neq \ell^r-1$ for all $\ell,r$, one can take the variety $Y_{2d}\subset (\P^1)^{\times (2d+2)}$ constructed in the same proof for any $\ell > 2d+2$.
\end{rmk}

\printbibliography

\end{document}